\documentclass[11pt]{article}

\usepackage{amsfonts,exscale}
\usepackage{mleftright}

 \usepackage{srcltx,hyperref}

\usepackage{color}

\newcommand{\comment}[1]{}

\def\?{\ 
{\bf\color{red}???}\ 
\immediate\write16{}
\immediate\write16{Warning: There was still a question mark . . . }
\immediate\write16{}}

\long\def\forget#1{}

\newcommand{\PublOrArXiv}[2]{#2}

\forget{
\pdfoutput=1

\usepackage[pdftex]{hyperref}
\hypersetup{%
 pdftitle     = {Periods of local shtukas with complex multiplication},
 pdfauthor    = {Rajneesh Kumar Singh},
 pdfcreator = {pdfLaTeX},
 pdfproducer = {pdfLaTeX with hyperref}
}
}

\usepackage{geometry}
\usepackage{amsmath}
\usepackage{amssymb}
\usepackage{amscd}
\usepackage{textcomp}

\usepackage{amsthm}
\theoremstyle{plain}
\newtheorem{Lemma}{Lemma}[section]
\newtheorem{Theorem}[Lemma]{Theorem}
\newtheorem{Proposition}[Lemma]{Proposition}
\newtheorem{Corollary}[Lemma]{Corollary}

\theoremstyle{definition}
\newtheorem{Definition}[Lemma]{Definition}
\newtheorem{Example}[Lemma]{Example}
\newtheorem{Remark}[Lemma]{Remark}

\newtheorem{Convention}[Lemma]{Convention}
\newtheorem{Point}[Lemma]{}

\usepackage{xy}
\xyoption{all}

\newdir^{ (}{{}*!/-3pt/\dir^{(}}    
\newdir^{  }{{}*!/-3pt/\dir^{}}    
\newdir_{ (}{{}*!/-3pt/\dir_{(}}    
\newdir_{  }{{}*!/-3pt/\dir_{}}

\newcommand{\DS}{\displaystyle}
\newcommand{\TS}{\textstyle}
\newcommand{\SC}{\scriptstyle}
\newcommand{\SSC}{\scriptscriptstyle}

\newcounter{zahl}

\def\theenumi{(\alph{enumi})}

\def\p@enumii{\theenumi}

\newcommand{\Art}{{\rm Art}}
\DeclareMathOperator{\Aut}{Aut}
\newcommand{\Betti}{{\rm Betti}}

\DeclareMathOperator{\End}{End}

\DeclareMathOperator{\Frob}{Frob}
\DeclareMathOperator{\Gal}{Gal}
\DeclareMathOperator{\GL}{GL}
\DeclareMathOperator{\Koh}{H}

\DeclareMathOperator{\Hom}{Hom}

\DeclareMathOperator{\QEnd}{QEnd}
\DeclareMathOperator{\QHom}{QHom}

\DeclareMathOperator{\Spec}{Spec}
\DeclareMathOperator{\Spf}{Spf}

\DeclareMathOperator{\Tr}{Tr}
\DeclareMathOperator{\Var}{V}

\newcommand{\alg}{{\rm alg}}

\newcommand{\dR}{{\rm dR}}

\newcommand{\id}{{\rm id}}

\newcommand{\insep}{{\rm insep}}

\renewcommand{\mod}{\;{\rm mod}\;}

\newcommand{\nor}{{\rm nor}}
\DeclareMathOperator{\ord}{ord}

\DeclareMathOperator{\rk}{rk}
\newcommand{\sep}{{\rm sep}}

\renewcommand{\phi}{\varphi}
\renewcommand{\theta}{\vartheta}
\renewcommand{\epsilon}{\varepsilon}

\let\setminus\smallsetminus

\newcommand{\es}{\enspace}
\newcommand{\open}{^\circ}

\newcommand{\dual}{^{^\vee}}

\newcommand{\mal}{^{\SSC\times}}

\newcommand{\dbl}{{\mathchoice{\mbox{\rm [\hspace{-0.15em}[}}
                              {\mbox{\rm [\hspace{-0.15em}[}}
                              {\mbox{\scriptsize\rm [\hspace{-0.15em}[}}
                              {\mbox{\tiny\rm [\hspace{-0.15em}[}}}}
\newcommand{\dbr}{{\mathchoice{\mbox{\rm ]\hspace{-0.15em}]}}
                              {\mbox{\rm ]\hspace{-0.15em}]}}
                              {\mbox{\scriptsize\rm ]\hspace{-0.15em}]}}
                              {\mbox{\tiny\rm ]\hspace{-0.15em}]}}}}
\newcommand{\dpl}{{\mathchoice{\mbox{\rm (\hspace{-0.15em}(}}
                              {\mbox{\rm (\hspace{-0.15em}(}}
                              {\mbox{\scriptsize\rm (\hspace{-0.15em}(}}
                              {\mbox{\tiny\rm (\hspace{-0.15em}(}}}}
\newcommand{\dpr}{{\mathchoice{\mbox{\rm )\hspace{-0.15em})}}
                              {\mbox{\rm )\hspace{-0.15em})}}
                              {\mbox{\scriptsize\rm )\hspace{-0.15em})}}
                              {\mbox{\tiny\rm )\hspace{-0.15em})}}}}
\newcommand{\invlim}[1][]{\ifthenelse{\equal{#1}{}}
{\DS \lim_{\longleftarrow}}
{\DS \lim_{\underset{#1}{\longleftarrow}}}
}
\newcommand{\dirlim}[1][]{\ifthenelse{\equal{#1}{}}
{\DS \lim_{\longrightarrow}}
{\DS \lim_{\underset{#1}{\longrightarrow}}}
}
\newcommand{\ul}[1]{{\underline{#1}}}

\newcommand{\wt}[1]{{\widetilde{#1}}}
\newcommand{\wh}[1]{{\widehat{#1}}}

\newcommand{\BOne} {{\mathchoice{\hbox{\rm1\kern-2.7pt l\kern.9pt}}
                              {\hbox{\rm1\kern-2.7pt l\kern.9pt}}
                              {\hbox{\scriptsize\rm1\kern-2.3pt l\kern.4pt}}
                              {\hbox{\scriptsize\rm1\kern-2.4pt l\kern.5pt}}}}

\newcommand{\BC}{{\mathbb{C}}}

\newcommand{\BF}{{\mathbb{F}}}
\newcommand{\BG}{{\mathbb{G}}}

\newcommand{\BN}{{\mathbb{N}}}

\newcommand{\BP}{{\mathbb{P}}}
\newcommand{\BQ}{{\mathbb{Q}}}
\newcommand{\BR}{{\mathbb{R}}}

\newcommand{\BZ}{{\mathbb{Z}}}

\newcommand{\bB}{{\mathbf{B}}}

\newcommand{\sG}{{\mathscr{G}}}

\usepackage{mathrsfs}

\newcommand{\CC}{{\cal{C}}}

\newcommand{\CI}{{\cal{I}}}
\newcommand{\CJ}{{\cal{J}}}

\newcommand{\CM}{{\cal{M}}}

\newcommand{\CO}{{\cal{O}}}

\newcommand{\CR}{{\cal{R}}}

\newcommand{\FD}{{\mathfrak{D}}}

\newcommand{\Fa}{{\mathfrak{a}}}
\newcommand{\Fb}{{\mathfrak{b}}}
\newcommand{\Fd}{{\mathfrak{d}}}
\newcommand{\Ff}{{\mathfrak{f}}}

\newcommand{\Fq}{{\mathfrak{q}}}
\newcommand{\Fm}{{\mathfrak{m}}}

\def\longto{\longrightarrow}
\def\into{\hookrightarrow}
\let\onto\twoheadrightarrow
\def\longonto{\mbox{$\kern2pt\longto\kern-8pt\to\kern2pt$}}
\def\isoto{\stackrel{}{\mbox{\hspace{1mm}\raisebox{+1.4mm}{$\SC\sim$}\hspace{-3.5mm}$\longrightarrow$}}}
\def\longinto{\lhook\joinrel\longrightarrow}

\newbox\mybox
\def\arrover#1{\mathrel{
       \setbox\mybox=\hbox spread 1.4em{\hfil$\scriptstyle#1$\hfil}
       \vbox{\offinterlineskip\copy\mybox
             \hbox to\wd\mybox{\rightarrowfill}}}}

\def\ulCC{{\underline{\CC\!}\,}}
\def\ulM{{\underline{M\!}\,}}
\def\ulHM{{\underline{\hat M\!}\,}}
\def\ulHN{{\underline{\hat N\!}\,}}

\newcommand{\tminus}[1]{\ell^{\SSC -}_{#1}}
\newcommand{\tplus}[1]{\ell^{\SSC +}_{#1}}
\newcommand{\ttplus}[1]{\tilde\ell^{\SSC +}_{#1}}
\newcommand{\tplusminus}{\ell}

\begin{document}
\author{Urs Hartl and Rajneesh Kumar Singh\footnote{Both authors acknowledge support by the Deutsche Forschungsgemeinschaft (DFG) in form of SFB 878 and Germany's Excellence Strategy EXC 2044--390685587 ``Mathematics M\"unster: Dynamics--Geometry--Structure''. The first author was also supported  the DFG in form of Project-ID 427320536 -- SFB 1442.}}
\PublOrArXiv{
\title{Periods of Drinfeld modules and local shtukas with complex multiplication --- Erratum}
}
{
\title{Periods of Drinfeld modules and local shtukas with complex multiplication}
}

\maketitle

\PublOrArXiv{\begin{appendix}
\stepcounter{section}
}
{

\begin{abstract}
Colmez~\cite{Colmez93} conjectured a product formula for periods of abelian varieties over number fields with complex multiplication and proved it in some cases. His conjecture is equivalent to a formula for the Faltings height of CM abelian varieties in terms of the logarithmic derivatives at $s=0$ of certain Artin $L$-functions.

In a series of articles we investigate the analog of Colmez's theory in the arithmetic of function fields. There abelian varieties are replaced by Drinfeld modules and their higher dimensional generalizations, so-called $A$-motives. In the present article we prove the product formula for the Carlitz module and we compute the valuations of the periods of a CM $A$-motive at all finite places in terms of Artin $L$-series. The latter is achieved by investigating the local shtukas associated with the $A$-motive.  
\end{abstract}


%
%

\section{Introduction}
\setcounter{equation}{0}

In \cite{Colmez93} P.~Colmez considers product formulas for periods of abelian varieties. Let $X$ be an abelian variety defined over a number field $K$ with complex multiplication by the ring of integers in a CM-field $E$ and of CM-type $\Phi$. Let $\BQ^\alg$ be the algebraic closure of $\BQ$ in $\BC$, let $H_E:=\Hom_\BQ(E,\BQ^\alg)$ be the set of all ring homomorphisms $E\into\BQ^\alg$ and assume that $K$ contains $\psi(E)$ for every $\psi\in H_E$. For a $\psi\in H_E$ let $\omega_\psi\in\Koh^1_{\dR}(X,K)$ be a non-zero cohomology class such that $a^*\omega_\psi=\psi(a)\cdot\omega_\psi$ for all $a\in E$. For every embedding $\eta\colon K\into\BQ^\alg$, let $X^\eta$ and $\omega_\psi^\eta$ be deduced from $X$ and $\omega_\psi$ by base extension. Let $(u_\eta)_\eta\in\prod_{\eta\in H_K}\Koh_1(X^\eta(\BC),\BZ)$ be a family of cycles compatible with complex conjugation. Let $v$ be a place of $\BQ$. If $v=\infty$ the de Rham isomorphism between Betti and de Rham cohomology yields a complex number $\int_{u_\eta}\omega_\psi^\eta$ and its absolute value $\bigl|\int_{u_\eta}\omega_\psi^\eta\bigr|_\infty\in\BR$. If $v$ corresponds to a prime number $p\in\BZ$, we fix an inclusion $\BQ^\alg\into\BQ_p^\alg$. With this data Colmez~\cite{Colmez93} associates a period $\int_{u_\eta}\omega_\psi^\eta$ in Fontaine's $p$-adic period field $\bB_{\dR}$ and an absolute value $\bigl|\int_{u_\eta}\omega_\psi^\eta\bigr|_v\in\BR$. He considers the product $\prod_v\prod_{\eta\in H_K}\bigl|\int_{u_\eta}\omega_\psi^\eta\bigr|_v$ and (after some modifications) conjectures that this product evaluates to~$1$; see \cite[Conjecture~0.1]{Colmez93} for the precise formulation. This conjecture is equivalent to a conjectural formula for the Faltings height of a CM abelian variety in terms of the logarithmic derivatives at $s=0$ of certain Artin $L$-functions. Colmez proves the conjectures when $E$ is an abelian extension of $\BQ$. On the way, he computes $\prod_{\eta\in H_K}\bigl|\int_{u_\eta}\omega_\psi^\eta\bigr|_v$ at a finite place $v$ in terms of the local factor at $v$ of the Artin $L$-series associated with an Artin character $a^0_{E,\psi,\Phi}\colon\Gal(\BQ^\alg/\BQ)\to\BC$ that only depends on $E$, $\psi$ and $\Phi$ but not on $X$ and $v$; see \cite[Th\'eor\`eme~I.3.15]{Colmez93}. There has been further progress on Colmez's conjecture by Obus~\cite{Obus13}, Yang~\cite{Yang13}, Andreatta, Goren, Howard, Madapusi Pera \cite{AGHM}, Yuan, Zhang~\cite{YuanZhang15}, Barquero-Sanchez, Masri \cite{BSM} and others.

Our goal in this article is to develop the analog of Colmez's theory in the ``Arithmetic of function fields''. Here abelian varieties are replaced by Drinfeld modules \cite{Drinfeld,Goss} and their higher dimensional generalizations, so-called $A$-motives, which also generalize Anderson's \emph{$t$-motives} \cite{Anderson86}. To define them let $\BF_q$ be a finite field with $q$ elements, let $C$ be a smooth projective, geometrically irreducible curve over $\BF_q$, let $\infty\in C$ be a fixed closed point and let $A:=\Gamma(C\setminus\{\infty\},\CO_C)$ be the ring of regular functions on $C$ outside $\infty$. Let $Q$ be the fraction field of $A$ and let $K$ be a finite field extension of $Q$ contained in a fixed algebraic closure $Q^\alg$ of $Q$. We write $A_K:=A\otimes_{\BF_q}K$ and consider the endomorphism $\sigma:=\id_A\otimes\Frob_{q,K}$ of $A_K$, where $\Frob_{q,K}(b)=b^q$ for $b\in K$. For an $A_K$-module $M$ we set $\sigma^*M:=M\otimes_{A_K,\sigma}A_K$ and for a homomorphism $f\colon M\to N$ of $A_K$-modules we set $\sigma^*f:=f\otimes\id_{A_K}\colon\sigma^*M\to\sigma^*N$. Let $\gamma\colon A\to K$ be the inclusion $A\subset Q\subset K$, and set $\CJ:=(a\otimes1-1\otimes\gamma(a)\colon a\in A)\subset A_K$. Then $\gamma$ can be recovered as the homomorphism $A\to A_K/\CJ=K$. 

\begin{Definition}\label{DefAMotive}
An \emph{$A$-motive of rank $r$ over $K$} is a pair $\ulM=(M,\tau_M)$ consisting of a locally free $A_K$-module $M$ of rank $r$ and an isomorphism $\tau_M\colon\sigma^*M|_{\Spec A_K\setminus\Var(\CJ)}\isoto M|_{\Spec A_K\setminus\Var(\CJ)}$ of the associated sheaves outside $\Var(\CJ)\subset\Spec A_K$. We write $\rk\ulM:=r$. A \emph{morphism} between $A$-motives $f\colon(M,\tau_M)\to(N,\tau_N)$ is an $A_K$-homomorphism $f\colon M\to N$ with $f\circ\tau_M=\tau_N\circ\sigma^*f$.
\end{Definition}

Let us first give a rough sketch of the function field analog of Colmez's conjecture, before we explain more details and our main results later in this introduction. An $A$-motive has various (co-)homology realizations, for example a \emph{de Rham realization} $\Koh^1_\dR(\ulM,K)$, and if it is uniformizable also a \emph{Betti realization} $\Koh_{1,\Betti}(\ulM,A)$. For every \emph{place $v$ of $Q$}, that is a closed point $v\in C$ there is a comparison isomorphism between the Betti and de Rham cohomology of $\ulM$, which for $\omega\in\Koh^1_\dR(\ulM,K)$ and $u\in\Koh_{1,\Betti}(\ulM,A)$ is given by a pairing $\langle\omega,u\rangle_v$ and allows to define the absolute value $\bigl|\int_u\omega\bigr|_v:=\bigl|\langle\omega,u\rangle_v\bigr|_v\in\BR$. Now we say that $\ulM$ has \emph{complex multiplication} if $\QEnd_K(\ulM):=\End_K(\ulM)\otimes_AQ$ contains a commutative, semi-simple $Q$-algebra $E$ of dimension $\dim_QE=\rk\ulM$. Here semi-simple means that $E$ is a product of fields and we do not assume that $E$ is itself a field. Let $\ulM$ be a uniformizable $A$-motive over a finite Galois extension $K\subset Q^\alg$ of $Q$, which has complex multiplication by a \emph{separable} $Q$-algebra $E$. Let $0\ne\omega_\psi\in\Koh^1_\dR(\ulM,K)$ satisfy $a^*\omega_\psi=\psi(a)\cdot\omega_\psi$ for all $a\in E$, where $\psi:E\to K$ is a $Q$-homomorphism. Then in Theorem~\ref{ThmDMPeriod} and Corollary~\ref{CorDMPeriod} we will for all finite places $v$ compute $\bigl|\int_u\omega_\psi\bigr|_v$ and its average over all $Q$-automorphisms of $K$ in terms of the local factor at $v$ of an Artin $L$-series. The question analogous to \cite{Colmez93} is then, whether one can make sense of the product $\prod_{v}\bigl|\int_u\omega_\psi\bigr|_v$ over all places $v$ including $\infty$, and whether this product evaluates to $1$.

\medskip

After this vague sketch let us give more details and precise definitions in order to formulate our main results. We start by introducing the cohomology realizations of an $A$-motive $\ulM$ over $K$. First of all, there is the \emph{de Rham realization} $\Koh^1_\dR(\ulM,K):=\sigma^*M/\CJ\cdot\sigma^*M$ and for each maximal ideal $v\subset A$ a \emph{$v$-adic \'etale realization} $\Koh^1_v(\ulM,A_v)$ where $A_v$ is the $v$-adic completion of $A$; see Definition~\ref{dualTatemodule} below. We let $Q_v$ be the fraction field of $A_v$, and we let $Q_\infty$ be the $\infty$-adic completion of $Q$ and $\BC_\infty$ be the completion of a fixed algebraic closure of $Q_\infty$. We fix a $Q$-embedding $Q^\alg\into\BC_\infty$ and consider the base extension of $\ulM$ to $\BC_\infty$. There is a notion of $\ulM$ being \emph{uniformizable} and a uniformizable $\ulM$ has a \emph{Betti realization} $\Koh^1_\Betti(\ulM,A)$; see \cite[\S\,3.5]{HartlJuschka}. These realizations are related by period isomorphisms 
\begin{align*}
h_{\Betti,\,v}\colon \Koh^1_\Betti(\ulM,A)\otimes_A A_v\; & \isoto\;\Koh^1_v(\ulM,A_v)\quad \text{and}\\
h_{\Betti,\,\dR}\colon \Koh^1_\Betti(\ulM,A)\otimes_A\BC_\infty & \isoto \;\Koh^1_{\dR}(\ulM,K)\otimes_K\BC_\infty\;;
\end{align*}
see \cite[Theorem~3.23]{HartlJuschka}. Also for every place $v$ of $Q$, let $\BF_v$ be its residue field and set $q_v:=\#\BF_v=q^{[\BF_v:\BF_q]}$. Let $z:=z_v\in Q$ be a uniformizing parameter at $v$. Then there is a canonical isomorphism $A_v=\BF_v\dbl z_v\dbr$. Let $\zeta:=\zeta_v:=\gamma(z_v)$ denote the image of $z_v$ in $K$. We simply write $z$, resp.\ $\zeta$ for the elements $z\otimes1$, resp.\ $1\otimes\zeta$ of $Q\otimes_{\BF_q}K$. Then the power series ring $K\dbl z-\zeta\dbr$ in the ``variable'' $z-\zeta$ is canonically isomorphic to the completion of the local ring of $C_K:=C\times_{\BF_q}K$ at $\Var(\CJ)$; see \cite[Lemma~1.2 and 1.3]{HartlJuschka}, and thus independent of $v$. We always consider the embedding $Q\into K\dbl z-\zeta\dbr$ given by $z\mapsto z=\zeta+(z-\zeta)$. The de Rham realization lifts to $\Koh^1_\dR(\ulM,K\dbl z-\zeta\dbr):=\sigma^*M\otimes_{A_K}K\dbl z-\zeta\dbr$, which is the analog of the (conjectural) $q$-de Rham cohomology of Bhatt, Morrow and Scholze \cite{BMS15,BMS16,Scholze_q-dR}, and the vector space $\Koh^1_\dR(\ulM,K\dbl z-\zeta\dbr)[\tfrac{1}{z-\zeta}]$ over the field $K\dpl z-\zeta\dpr:=K\dbl z-\zeta\dbr[\tfrac{1}{z-\zeta}]$ contains the $K\dbl z-\zeta\dbr$-lattice $\Fq^\ulM:=\tau_M^{-1}(M\otimes_{A_K}K\dbl z-\zeta\dbr)$, which is called the \emph{Hodge-Pink lattice of $\ulM$} and is the analog of the Hodge-filtration of an abelian variety; see \cite[Remark~5.13]{HartlKim}. 

If $v\ne\infty$ we also fix a $Q$-embedding of $Q^\alg$ into a fixed algebraic closure $Q_v^\alg$ of $Q_v$ and we let $\BC_v$ be the $v$-adic completion of $Q_v^\alg$. Again we denote the image of $z_v$ in $Q_v^\alg$ and $\BC_v$ by $\zeta_v$. We let $K_v\subset Q_v^\alg$ be the induced completion of $K$ and we let $R$ be its valuation ring. There is a period isomorphism by \cite[Remark~4.16]{HartlKim}
\[
h_{v,\dR}\colon\Koh^1_v(\ulM,A_v)\otimes_{A_v}\BC_v\dpl z_v-\zeta_v\dpr\;\isoto\;\Koh^1_\dR\bigl(\ulM,K\dbl z_v-\zeta_v\dbr\bigr)\otimes_{K\dbl z_v-\zeta_v\dbr}\BC_v\dpl z_v-\zeta_v\dpr\,.
\]
The field $\BC_v\dpl z_v-\zeta_v\dpr$ is the analog of Fontaine's $p$-adic period field $\bB_\dR$; see \cite[Remark~4.17]{HartlKim}.

\medskip

Let $\ulM$ have complex multiplication by a commutative, semi-simple $Q$-algebra $E$ of dimension $\dim_QE=\rk\ulM$. Let $\CO_E$ be the integral closure of $A$ in $E$. It is a locally free $A$-module of $\rk_A\CO_E=\dim_QE$. We let $H_E:=\Hom_Q(E,Q^\alg)$ be the set of $Q$-homomorphisms $\psi\colon E\to Q^\alg$ and we assume that $K$ contains $\psi(E)$ for every $\psi\in H_E$. Then by Lemma~\ref{LemDecompdR} in the appendix there is a decomposition $E\otimes_QK\dbl z-\zeta\dbr=\prod_{\psi\in H_E}K\dbl y_\psi-\psi(y_\psi)\dbr$, where $y_\psi$ is a uniformizer at a place of $E$ such that $\psi(y_\psi)\ne0$. Again by \cite[Lemma~1.2 and 1.3]{HartlJuschka} the factors are obtained as the completion of $\CO_E\otimes_A A_K=\CO_E\otimes_{\BF_q}K$ along the kernels $(a\otimes1-1\otimes\psi(a)\colon a\in\CO_E)$ of the homomorphisms $\psi\otimes\id_K\colon\CO_E\otimes_{\BF_q}K\to K$ for $\psi\in H_E$. 
Correspondingly $\Koh^1_\dR(\ulM,K\dbl z-\zeta\dbr)$ decomposes into eigenspaces 
\[
\Koh^\psi(\ulM,K\dbl y_\psi-\psi(y_\psi)\dbr)\;:=\;\Koh^1_\dR(\ulM,K\dbl z-\zeta\dbr)\otimes_{E\otimes_QK\dbl z-\zeta\dbr}\,K\dbl y_\psi-\psi(y_\psi)\dbr
\]
each of which is free of rank $1$ over $K\dbl y_\psi-\psi(y_\psi)\dbr$. There are integers $d_\psi$ such that the Hodge-Pink lattice is $\Fq^\ulM=\prod_\psi(y_\psi-\psi(y_\psi))^{-d_\psi}\Koh^\psi(\ulM,K\dbl y_\psi-\psi(y_\psi)\dbr)$. The tuple $\Phi:=(d_\psi)_{\psi\in H_E}$ is the \emph{CM-type} of $\ulM$. 

If we fix elements $u\in\Koh_{1,\Betti}(\ulM,Q):=\Hom_A\bigl(\Koh^1_\Betti(\ulM,A),Q\bigr)$ and $\omega\in\Koh^1_\dR(\ulM,K\dbl z-\zeta\dbr)$ we can define
\begin{align}
\label{EqIntInfty1} \langle\omega,u\rangle_\infty\;:=\es & u\otimes\id_{\BC_\infty}\bigl(h_{\Betti,\dR}^{-1}(\omega\mod z-\zeta)\bigr)\;\in\;\BC_\infty \qquad\text{and}\\ 
\label{EqIntInfty2} \bigl|\TS\int_u\omega\bigr|_\infty\;:=\es & \bigl|\langle\omega,u\rangle_\infty\bigr|_\infty\;\in\;\BR\,,
\end{align}
where $|\,.\,|_v$ is the normalized absolute value on $\BC_v$ with $|\zeta_v|_v=(\#\BF_v)^{-1}=q_v^{-1}$ for every place $v$. We also consider the valuation $v\colon\BC_v\mal\to\BQ$ on $\BC_v$ with $v(\zeta_v)=1$. The expressions in \eqref{EqIntInfty1} and \eqref{EqIntInfty2} only depend on the image of $\omega$ in $\Koh^1_\dR(\ulM,K)$. Also at a finite place $v\ne\infty$ of $Q$ we consider on elements $x\ne0$ of the discretely valued field $\BC_v\dpl z_v-\zeta_v\dpr$ the valuation $\hat v(x):=\ord_{z_v-\zeta_v}(x)$, and in addition we define 
\begin{align*}
& |x|_v:=\bigl|\bigl((z_v-\zeta_v)^{-\hat v(x)}\cdot x\bigr)\mod z_v-\zeta_v\bigr|_v \qquad\text{and} \\[1mm]
& v(x):=-\log|x|_v\,/\log q_v \qquad\text{induced from}\\[1mm]
& \bigl((z_v-\zeta_v)^{-\hat v(x)}\cdot x\bigr)\mod z_v-\zeta_v\in\BC_v\,.
\end{align*}
Note that $|x|_v$ and $v(x)$ are not a norm, respectively a valuation, because they do not satisfy the triangle inequality. The value $|x|_v$ does not depend on the choice of the uniformizer $z_v$ of $A_v$, because if $\tilde z_v=\sum_{n=0}^\infty b_nz_v^n=:f(z_v)$ with $b_n\in\BF_v$ is another uniformizer and $\tilde\zeta_v= f(\zeta_v)$, then $\tfrac{\tilde z_v-\tilde\zeta_v}{z_v-\zeta_v}\equiv f'(\zeta_v)\mod z_v-\zeta_v$ in $\CO_{\BC_v}\dbl z_v\dbr=\BF_v\dbl z_v\dbr\wh\otimes_{\BF_v,\gamma}\CO_{\BC_v}$ by Lemma~\ref{LemDerivative} in the appendix and $f'(\zeta_v)\in\BF_v\dbl\zeta_v\dbr\mal$ with inverse $\tfrac{dz_v}{d\tilde z_v}\big|_{\tilde z_v=\tilde\zeta_v}$.
We define
\begin{align}
\label{EqIntv1} \langle\omega,u\rangle_v\;:=\;&u\otimes_{\BC_v\dpl z_v-\zeta_v\dpr}\bigl(h_{\Betti,v}^{-1}\circ h_{v,\dR}^{-1}(\omega)\bigr) \;\in\; \BC_v\dpl z_v-\zeta_v\dpr\qquad\text{and}\\
\label{EqIntv3} \TS\bigl|\int_u\omega\bigr|_v\;:=\;&\bigl|\langle\omega,u\rangle_v\bigr|_v\;:=\;\bigl|\bigl((z_v-\zeta_v)^{-\hat v(\langle\omega,u\rangle_v)}\cdot\langle\omega,u\rangle_v\bigr)\mod z_v-\zeta_v\bigr|_v\;\in\;\BR\,.
\end{align}
We will show in Theorem~\ref{ThmDMPeriod} below that if $E$ is separable over $Q$ and if $\omega\in\Koh^\psi(\ulM,K\dbl y_\psi-\psi(y_\psi)\dbr)$ has non-zero image in $\Koh^1_\dR(\ulM,K)$, then the absolute value \eqref{EqIntv3} only depends on that image.

With these definitions we can now consider the product $\prod_v\bigl|\int_u\omega\bigr|_v$, or equivalently its logarithm $\log\prod_v\bigl|\int_u\omega\bigr|_v=-\sum_v v\bigl(\int_u\omega\bigr)\log q_v$. Like in Colmez's theory, these products or sums do not converge and one has to give a convergent interpretation to their finite parts $\prod_{v\ne\infty}\bigl|\int_u\omega\bigr|_v$, respectively $-\sum_{v\ne\infty} v\bigl(\int_u\omega\bigr)\log q_v$; see Convention~\ref{Convention} below. To formulate the convention we make the following

\begin{Definition}\label{DefArtinMeasure}
For $F=Q$ or $F=Q_v$ let $F^\sep$ be the separable closure of $F$ in $F^\alg$ and let $\sG_F:=\Gal(F^\sep/F)$. For a finite field extension $F'$ of $F$ let $H_{F'}:=\Hom_F(F',F^\alg)$ be the set of $F$-homomorphisms $\psi\colon F'\to F^\alg$. Let $\CC(\sG_F,\BQ)$ be the $\BQ$-vector space of locally constant functions $a\colon\sG_F\to\BQ$ and let $\CC^0(\sG_F,\BQ)$ be the subspace of those functions which are constant on conjugacy classes, that is, which satisfy  $a(h^{-1}gh)=a(g)$ for all $g,h\in\sG_F$. Then the $\BC$-vector space $\CC^0(\sG_F,\BQ)\otimes_\BQ\BC$ is spanned by the traces of representations $\rho\colon\sG_F\to\GL_n(\BC)$ with open kernel for varying $n$ by \cite[\S\,2.5, Theorem~6]{SerreLinRep}. Via the fixed embedding $Q^\sep\into Q_v^\sep$ we consider the induced inclusion $\sG_{Q_v}\subset\sG_Q$ and morphism $\CC(\sG_Q,\BQ)\to\CC(\sG_{Q_v},\BQ)$. If $\chi$ is the trace of a representation $\rho\colon\sG_Q\to\GL_n(\BC)$ with open kernel we let $L(\chi,s):=\prod_{\text{all }v}L_v(\chi,s)$, respectively $L^\infty(\chi,s):=\prod_{v\ne\infty}L_v(\chi,s)$ be the Artin $L$-function of $\rho$, respectively without the factor at $\infty$. It only depends on $\chi$ and converges for all $s \in \BC$ with $\CR e(s)>1$; see \cite[pp.~126ff]{Rosen02}. We also set
\begin{align}\label{EqZFct}
Z(\chi,s)\;:=\; & \frac{\tfrac{d}{ds}L(\chi,s)}{L(\chi,s)}\;=\;-\sum_{\text{all }v}Z_v(\chi,s)\log q_v\qquad\text{and}\\
Z^\infty(\chi,s)\;:=\; & \frac{\tfrac{d}{ds}L^\infty(\chi,s)}{L^\infty(\chi,s)}\;=\;-\sum_{v\ne\infty}Z_v(\chi,s)\log q_v\qquad\text{with}\\
 Z_v(\chi,s)\;:=\; & \frac{\tfrac{d}{ds}L_v(\chi,s)}{-L_v(\chi,s)\cdot\log q_v}\;=\;\frac{\tfrac{d}{dq_v^{-s}}L_v(\chi,s)}{q_v^s\cdot L_v(\chi,s)}\;.
\end{align}
Moreover, we let $\Ff_\chi$ be the Artin conductor of $\chi$. It is an effective divisor $\Ff_\chi=\sum_v \mu_{\Art,v}(\chi)\cdot(v)$ on $C$; see \cite[Chapter~VI, \S\S\,2,3]{SerreLF}, where $\mu_{\Art,v}(\chi)$ is denoted $f(\chi,v)$. We set 
\begin{align}\label{EqMuArt}
\mu_\Art(\chi) \; := \; & \deg(\Ff_\chi)\log q \; := \; \sum_{\text{all }v}\mu_{\Art,v}(\chi)[\BF_v:\BF_q] \log q\; = \; \sum_{\text{all }v}\mu_{\Art,v}(\chi) \log q_v\quad\text{and}\\[2mm]
\mu^\infty_\Art(\chi) \; := \; & \sum_{v\ne\infty}\mu_{\Art,v}(\chi) \log q_v\,.
\end{align}
In particular, only finitely many values $\mu_{\Art,v}(\chi)$ are non-zero. By linearity we extend $Z^\infty(\,.\,,s)$ and $\mu^\infty_\Art$ to all $a\in\CC^0(\sG_Q,\BQ)$ and $Z_v(\,.\,,s)$ and $\mu_{\Art,v}$ to all $a\in\CC^0(\sG_{Q_v},\BQ)$. The map $Z_v(\,.\,,s)$ takes values in $\BQ(q_v^{-s})$. 
\end{Definition}

\medskip

In terms of this definition we prove in this article a formula for $\bigl|\int_u\omega\bigr|_v$ with $v\ne\infty$ for a uniformizable $A$-motive $\ulM$ over $K$ with complex multiplication by a semi-simple \emph{separable} CM-algebra $E$ of CM-type $\Phi=(d_\phi)_{\phi\in H_E}$ as follows. Let us assume that $\CO_E\subset\End_K(\ulM)$, that $K$ is a finite Galois extension of $Q$ which contains $\psi(E)$ for all $\psi\in H_E$, and that $\ulM$ has good reduction at all primes of $K$. (By unpublished results of Schindler~\cite{Schindler} this is no restriction of generality, because for every $A$-motive $\ulM'$ with complex multiplication by a semi-simple separable CM-algebra $E$ there is an $A$-motive $\ulM$ isogenous to $\ulM'$ such that the integral closure $\CO_E$ of $A$ in $E$ is contained in $\End_K(\ulM)$ and $\ulM'$ and $\ulM$ have good reduction everywhere after replacing $K$ by a finite separable extension. Moreover, every $A$-motive over a field extension of $Q$ with $\CO_E\subset\End_K(\ulM)$ is already defined over a finite separable extension $K$ of $Q$). For $\psi\in H_E$ we define the functions
\begin{align}\label{Eq:a_E}
a_{E,\psi,\Phi}\colon\;\sG_Q\to\BZ, & \quad g\mapsto d_{g\psi}\qquad\text{and}\\[2mm]
a^0_{E,\psi,\Phi}\colon\;\sG_Q\to\BQ, & \quad g\mapsto \tfrac{1}{\#H_K}\TS\sum\limits_{\eta\in H_K}d_{\eta^{-1}g\eta\psi}\label{Eq:a0_E}
\end{align}
which factor through $\Gal(K/Q)$. In particular, $a_{E,\psi,\Phi}\in\CC(\sG_Q,\BQ)$ and $a^0_{E,\psi,\Phi}\in\CC^0(\sG_Q,\BQ)$ is independent of $K$. 

Note that $\Koh_{1,\Betti}(\ulM,Q)$ is a free $E$-module of rank $1$ by \cite[Lemma~7.2]{BH2} and that the eigenspace $\Koh^\psi(\ulM,K)\,:=\,\Koh^\psi(\ulM,K\dbl y_\psi-\psi(y_\psi)\dbr)/\bigl(y_\psi-\psi(y_\psi)\bigr)\Koh^\psi(\ulM,K\dbl y_\psi-\psi(y_\psi)\dbr)$ in $\Koh^1_\dR(\ulM,K)$ of the character $\psi\colon E\to K$ is a $K$-vector space of dimension $1$ by Proposition~\ref{PropCMTateMod} below. For an $E$-generator $u\in\Koh_{1,\Betti}(\ulM,Q)$ and a generator $\omega_\psi\in\Koh^\psi(\ulM,K\dbl y_\psi-\psi(y_\psi)\dbr)$ as $K\dbl y_\psi-\psi(y_\psi)\dbr$-module we next define integers $v(\omega_\psi)$ and $v_\psi(u)$ for all $v\ne\infty$ which are all zero except for finitely many. Let $\CO_{E_v}:=\CO_E\otimes_AA_v$ and let $c\in E_v:=E\otimes_Q Q_v$ be such that $c^{-1}u$ is an $\CO_{E_v}$-generator of $\Koh_{1,\Betti}(\ulM,A)\otimes_A A_v$, which exists because $\CO_{E_v}$ is a product of discrete valuation rings. Then $c$ is unique up to multiplication by an element of $\CO_{E_v}\mal$ and we set
\begin{align}\label{EqVPsi}
v_\psi(u)\;:=\;v & \bigl(\psi(c)\bigr)\;\in\;\BZ\,, \\[2mm]
& \text{\color{red} (NOTE THAT \ $v_\psi(u)\;\in\;\BQ$ \ IN GENERAL; SEE ERRATUM~\ref{Erratum1})} \nonumber
\end{align}
where we extend $\psi\in H_E$ by continuity to $\psi\colon E_v\to Q_v^\alg$. Also let $\ul\CM=(\CM,\tau_\CM)$ be an $A$-motive over $R:=\CO_{K_v}$ with good reduction and $\ul\CM\otimes_{\CO_{K_v}}K_v\cong\ulM\otimes_K K_v$; see Example~\ref{AMotLocSht}. Then there is an element $x\in K_v\mal$, unique up to multiplication by $R\mal$, such that $x^{-1}\omega_\psi\mod y_\psi-\psi(y_\psi)$ is an $R$-generator of the free $R$-module of rank one 
\[
\Koh^\psi(\ul\CM,R)\;:=\;\bigl\{\omega\in\Koh^1_\dR(\ul\CM,R):=\sigma^*\CM\otimes_{A_R,\,\gamma\otimes\id_R}R\colon [b]^*\omega=\psi(b)\cdot\omega\es\forall\;b\in\CO_E\bigr\}\,,
\]
and we set
\begin{align}\label{EqValuationOfOmegaPsi}
v(\omega_\psi)\;:=\; & v(x)\;\in\;\BZ\,. \\[2mm]
& \text{\color{red} (NOTE THAT THIS DEFINITION IS WRONG; SEE ERRATUM~\ref{Erratum2})} \nonumber
\end{align}
This value only depends on the image of $\omega_\psi$ in $\Koh^1_\dR(\ulM,K)$. It also does not depend on the choice of the model $\ul\CM$ with good reduction, because all such models are isomorphic over $R$ by \cite[Proposition~2.13(ii)]{Gardeyn4}. In this situation our first main result is the following

\begin{Theorem}\label{MainThm}
Let $\omega_\psi$ be a generator of the $K\dbl y_\psi-\psi(y_\psi)\dbr$-module $\Koh^\psi(\ulM,K\dbl y_\psi-\psi(y_\psi)\dbr)$. For every $\eta\in H_K$ let $\ulM^\eta$ and $\omega_\psi^\eta\in\Koh^{\eta\psi}(\ulM^\eta,K\dbl y_{\eta\psi}-\eta\psi(y_{\eta\psi})\dbr)$ be obtained by extension of scalars via $\eta$, and choose an $E$-generator $u_\eta\in\Koh_{1,\Betti}(\ulM^\eta,Q)$. Then for every place $v\ne\infty$ of $C$ we have
\[
\TS\tfrac{1}{\#H_K}\sum\limits_{\eta\in H_K}v({\TS\int_{u_\eta}\omega_\psi^\eta}) \;=\; Z_v(a^0_{E,\psi,\Phi},1)-\mu_{\Art,v}(a^0_{E,\psi,\Phi})-\dfrac{v(\Fd_{\psi(E)/Q})}{[\psi(E):Q]}+\tfrac{1}{\#H_K}\sum\limits_{\eta\in H_K}\bigl(v(\omega_\psi^\eta)+v_{\eta\psi}(u_\eta)\bigr)\,,
\]
where $\Fd_{\psi(E)/Q}$ is the discriminant of the field extension $\psi(E)/Q$.
\end{Theorem}

We will prove this theorem at the end of Section~\ref{SectPeriods} by using the \emph{local shtuka} at $v$ attached to $\ulM$. The latter is an analog of the Dieudonn\'e-module of the $p$-divisible group attached to an abelian variety; see \cite[\S\,3.2]{HartlDict}. The theorem allows us to make the following convention which is the analog of \cite[Convention~0]{Colmez93}.

\begin{Convention}\label{Convention} Let $(x_v)_{v\ne\infty}$ be a tuple of complex numbers indexed by the finite places $v$ of $Q$. We will give a sense to the (divergent) series $\Sigma \stackrel{?}{=} \sum_{v\ne \infty} x_v$ in the following situation. We suppose that there exists an element $a\in\CC^0(\sG_Q,\BQ)$ such that $x_v = - Z_v(a,1) \log q_v$  for all $v$ except for finitely many. Then we let $a^*\in\CC^0(\sG_Q,\BQ)$ be defined by $a^*(g):=a(g^{-1})$. We further assume that $Z^\infty(a^*,s)$ does not have a pole at $s=0$, and we define the limit of the series $\sum_{v\ne \infty} x_v$ as
\begin{equation}\label{EqConvention}
\Sigma\;:=\;-Z^\infty(a^*,0) - \mu^\infty_\Art(a) +\sum_{v\neq \infty} \bigl(x_v + Z_v(a,1)\log q_v\bigr)
\end{equation}
inspired by Weil's \cite[p.~82]{WeilRH} functional equation 
\[
Z(\chi,1-s)\;=\;-Z(\chi^*,s)-(2\cdot genus(C)-2)\chi(1)\log q-\mu_\Art(\chi)
\]
deprived of the summands at $\infty$, where the genus term is considered as belonging to $\infty$.
\end{Convention}

The Convention~\ref{Convention} and the Theorem~\ref{MainThm} allow us to give a convergent interpretation to the sum $-\sum_v\sum_{\eta\in H_K} v\bigl(\int_{u_\eta}\omega_\psi^\eta\bigr)\log q_v$ and the product $\prod_v\prod_{\eta\in H_K}\bigl|\int_{u_\eta}\omega_\psi^\eta\bigr|_v$, and we can ask whether this product is $1$. In Section~\ref{SectCarlitz} we prove our second main result, namely that the answer to the question is ``yes'' in the easiest case of the \emph{Carlitz motive} which is related to the zeta function of $\BF_q[t]$ and is the analog of the multiplicative group $\BG_{m,\BQ}$ considered by Colmez. For general CM $A$-motives we plan to address the question in a sequel to this article and also discuss its consequences for the Faltings height of CM $A$-motives similar to \cite[Th\'eor\`eme~0.3 and Conjecture~0.4]{Colmez93} and conditions under which $Z^\infty(a^*,s)$ does not have a pole at $s=0$. 

Let us describe the structure of this article. In Section~\ref{LocSht} we recall from \cite{HartlKim} the definition of local shtukas, how to attach a local shtuka at $v\subset A$ to an $A$-motive $\ulM$ over $L$ with good reduction, and we discuss its cohomology realizations. In Section~\ref{Amotivescm} we define the notions of complex multiplication and CM-type of a local shtuka, and in Section~\ref{SectPeriods} we compute the periods and their valuations of a local shtuka with complex multiplication, and we prove Theorem~\ref{MainThm}. Finally in Appendix~\ref{Appendix} we prove the facts used above.

\section{The Carlitz-Motive}\label{SectCarlitz}
\setcounter{equation}{0}

Let $ A = \BF_q[t]$ and $C = \BP^1_{\BF_q}$. Let $K = \BF_q(\theta)$ be the rational function field in the variable $\theta$ and let  $\gamma : A\to K$ be given by $\gamma(t) = \theta$. We also set $z:=z_\infty:=\tfrac{1}{t}$ and $\zeta:=\zeta_\infty:=\tfrac{1}{\theta}$. It satisfies $|\zeta|_\infty=q^{-1}<1$. The \emph{Carlitz motive} over $K$ is the $A$-motive $\ulCC = (K[t], \tau_\CC = t-\theta)$ which is associated with the Carlitz module; see \cite{Carlitz35} or \cite[Chapter~3]{Goss}. It has rank $1$ and dimension $1$, and complex multiplication by the ring of integers $A$ in $E:=Q$ with CM-type $\Phi=(d_\id)$, where $H_E=\{\id\}$ and $d_\id=1$. As is well known, its cohomology satisfies $\Koh^1_\dR(\ulCC,K\dbl z-\zeta\dbr)=K\dbl z-\zeta\dbr$ and $\Koh^1_\Betti(\ulCC,A)= A\cdot\beta\tminus{\zeta}$, where $\beta\in\BC_\infty$ satisfies $\beta^{q-1}=-\zeta$ and $\tminus{\zeta}:=\prod_{i=0}^\infty(1-\zeta^{q^i}t)$; see for example \cite[Example~3.34]{HartlJuschka}. We denote the generator $1$ of $\Koh^1_\dR(\ulCC,K\dbl z-\zeta\dbr)$ by $\omega$ and we take $u\in\Koh_{1,\Betti}(\ulCC,\BF_q[t])$ as the generator which is dual to $\beta\tminus{\zeta}\in\Koh^1_\Betti(\ulCC,\BF_q[t])$. The de Rham isomorphism $h_{\Betti,\dR}$ sends $\beta\tminus{\zeta}$ to 
\[
\sigma^*(\beta\tminus{\zeta})\cdot\omega\;=\;\beta^q\sigma^*(\tminus{\zeta})\cdot\omega\;\in\;\Koh^1_\dR(\ulM,\BC_\infty\dbl z-\zeta\dbr)\;=\;\BC_\infty\dbl z-\zeta\dbr\cdot\omega, 
\]
respectively to $\beta^q\sigma^*(\tminus{\zeta})|_{t=\theta}\cdot\omega=\beta^q\prod_{i=1}^\infty(1-\zeta^{q^i-1})\cdot\omega\in\Koh^1_\dR(\ulM,\BC_\infty)=\BC_\infty\cdot\omega$. Here the coefficient $\beta^q\prod_{i=1}^\infty(1-\zeta^{q^i-1})$ is the function field analog of the complex number $(2i\pi)^{-1}$, the inverse of the period of the multiplicative group $\BG_{m,\BQ}$. We obtain
\[
\TS\bigl|\int_u\omega\bigr|_\infty\;=\;\bigl|\bigl(\beta^q\prod\limits_{i=1}^\infty(1-\zeta^{q^i-1})\bigr)^{-1}\bigr|_\infty\;=\;|\beta|_\infty^{-q}\;=\;q^{q/(q-1)}.
\]
At a finite place $v\subset\BF_q[t]$ let $v=(z_v)$ and $\zeta_v=\gamma(z_v)$. Then $\Koh^1_v(\ulM,A_v)= A_v\cdot(\tplus{\zeta_v})^{-1}$, where $\tplus{\zeta_v}:=\sum_{n=0}^\infty \tplusminus_nz_v^n\in\BC_v\dbl z_v\dbr$ with $\tplusminus_0^{q_v-1}=-\zeta_v$ and $\tplusminus_n^{q_v}+\zeta_v \tplusminus_n=\tplusminus_{n-1}$; see \cite[Example~4.19]{HartlKim}. This implies $|\tplusminus_n|=|\zeta_v|^{q_v^{-n}/(q_v-1)}<1$. The $v$-adic comparison isomorphism $h_{v,\dR}$ sends the generator $\tplus{\zeta_v}$ of $\Koh^1_v(\ulM,A_v)$ to 
\[
(z_v-\zeta_v)^{-1}(\tplus{\zeta_v})^{-1}\cdot\omega\;\in\;\Koh^1_\dR\bigl(\ulM,\BC_v\dpl z_v-\zeta_v\dpr\bigr)\;=\;\BC_v\dpl z_v-\zeta_v\dpr\cdot\omega, 
\]
where the coefficient of $\omega$ is the \emph{$v$-adic Carlitz period} which has a pole of order one at $z_v=\zeta_v$. So $\langle\omega,u\rangle_v=(z_v-\zeta_v)\tplus{\zeta_v}$ has $\hat v(\langle\omega,u\rangle_v)=1$ and 
\[
\TS\bigl|\int_u\omega|_v\;=\;\bigl|\tplus{\zeta_v}\mod z_v-\zeta_v\bigr|_v\;=\;\bigl|\sum\limits_{n=0}^\infty \tplusminus_n\zeta_v^n\bigr|_v\;=\;|\tplusminus_0|_v\;=\;q_v^{-1/(q_v-1)}
\]
So the product $\prod_v \bigl|{\TS\int_u\omega}\bigr|_v$ of the norms at all places has logarithm
\begin{equation}\label{logpro}
\log\prod_{\text{all }v} \bigl|{\TS\int_u\omega}\bigr|_v\;=\; \log\bigl|{\TS\int_u\omega}\bigr|_\infty+\log \prod_{v \ne \infty}\bigl|{\TS\int_u\omega}\bigr|_v\;=\;\frac{q}{q-1}\log q + \sum_{v\ne \infty }\frac{-1}{q_v-1}\log q_v\,. 
\end{equation}
Note that this series is not convergent, but that the sum over $v \ne \infty$ is equal to $\frac{\zeta'_A(1)}{\zeta_A(1)}$ and the summand at $\infty$ is equal to $\frac{\zeta'_A(0)}{\zeta_A(0)}$, where $\zeta_A$ is the zeta function associated with $A$, which does not converge at $s=1$. Namely, the zeta functions are defined as the following products which converge for $s \in \BC$ with $\CR e(s)>1$ 
\begin{align*}
& \zeta_C(s) \;:=\; \prod_{\text{all }v}(1-(\#\BF_v)^{-s})^{-1} \;=\; \prod_{\text{all }v}(1-q_v^{-s})^{-1} \;=\; \frac{1}{(1-q^{-s})(1-q^{1-s})} \qquad\text{and}\\
& \zeta_A(s) \;:=\; \prod_{  v \ne \infty}(1-(\#\BF_v)^{-s})^{-1} \;=\; \prod_{v \ne \infty}(1-q_v^{-s})^{-1} \;=\; \frac{1}{1-q^{1-s}}\ ;
\end{align*}
see for example\cite[Chapter~V, Example~2.1]{Silverman1}. In particular $(1-q_v^{-s})^{-1}=L_v(\BOne,1)$ and $\zeta_A(s)=L^\infty(\BOne,s)$ where $\BOne\colon g\mapsto1$ is the trivial character in $\CC^0(\sG_Q,\BQ)$, which equals $a_{Q,\id,\Phi}=a^0_{Q,\id,\Phi}$. Then $\frac{1}{q_v^s-1}= Z_v(\BOne,s)$ and $\tfrac{\zeta_A'(s)}{\zeta_A(s)}=Z^\infty(\BOne,s)$ in the notation of Definition~\ref{DefArtinMeasure}. Applying Convention~\ref{Convention} with $a=\BOne$ and $\mu_\Art(\BOne)=0$ and $genus(\BP^1_{\BF_q})=0$ we define the limit of the series as
\begin{align*}
\frac{q}{q-1}\log q + \sum_{v\ne \infty }\frac{-1}{q_v-1}\log q_v & := \frac{q}{q-1}\log q - \sum_{v\ne \infty }Z_v(\BOne,1) \log q_v \\
                                                                                & = \frac{q}{q-1}\log q - Z^\infty(\BOne,0) \\
                                                                               & =  \frac{q} {q-1}\log q - \frac{\zeta'_A(0)}{\zeta_A(0)}\\
                                                                               & = 0.
\end{align*}
So the value of the product $\prod_v \bigl|{\TS\int_u\omega}\bigr|_v$ is $1$ for the Carlitz motive.

\section{Local Shtukas}\label{LocSht}
\setcounter{equation}{0}

In the rest of the article we fix a place $v\ne\infty$ of $Q$. We keep the notation from the introduction, except that we write $L=\kappa\dpl\pi_L\dpr$ for the field $K_v$ and let $R=\kappa\dbl\pi_L\dbr$ be its valuation ring. We write $z:=z_v$. Then $A_v=\BF_v\dbl z\dbr$ and $Q_v=\BF_v\dpl z\dpr$. The homomorphism $\gamma\colon A\to K$ extends by continuity to $\gamma\colon A_v\to L$ and factors through $\gamma\colon A_v\to R$ with $\zeta:=\zeta_v=\gamma(z)\in\pi_L R\setminus\{0\}$. Let $R\dbl z\dbr$ be the power series ring in the variable $z$ over $R$ and $\hat\sigma$ the endomorphism of $R\dbl z\dbr$ with $\hat\sigma(z)=z$ and $\hat\sigma(b)=b^{q_v}$ for $b\in R$, where $q_v=\#\BF_v$. For an $R\dbl z\dbr$-module $\hat M$ we let $ \hat\sigma^* \hat M:= \hat M\otimes_{R\dbl z \dbr, \hat\sigma^*}R\dbl z\dbr$ as well as  $\hat M[\frac{1}{z-\zeta}]:=\hat M\otimes_{R\dbl z\dbr}R\dbl z\dbr[\frac{1}{z-\zeta}]$ and  $\hat M[\frac{1}{z}]:=\hat M\otimes_{R\dbl z\dbr}R\dbl z\dbr[\frac{1}{z}]$.

\begin{Definition} \label{LocalshtDef}
A \emph{local $\hat\sigma$-shtuka of rank $r$} over $R$ is a pair $\ulHM = (\hat M,\tau_{\hat M})$ consisting of a free $R\dbl z\dbr$-module $\hat M$ of rank $r$, and an isomorphism $\tau_{\hat M}:\hat\sigma^* \hat M[\frac{1}{z-\zeta}] \isoto \hat M[\frac{1}{z-\zeta}]$. It is \emph{effective} if $\tau_{\hat M}(\hat\sigma^*\hat M)\subset \hat M$ and \emph{\'etale} if $\tau_{\hat M}(\hat\sigma^*\hat M)= \hat M$. We write $\rk\ulHM$ for the rank of $\ulHM$.

A \emph{morphism} of local shtukas $f:\ulHM=(\hat M,\tau_{\hat M})\to\ulHN=(\hat N,\tau_{\hat N})$ over $R$ is a morphism of the underlying modules $f:\hat M\to \hat N$ which satisfies $\tau_{\hat N}\circ \hat\sigma^* f = f\circ \tau_{\hat M}$. We denote the $A_v$-module of homomorphisms $f\colon\ulHM\to\ulHN$ by $\Hom_R(\ulHM,\ulHN)$ and write $\End_R(\ulHM)=\Hom_R(\ulHM,\ulHM)$.

A \emph{quasi-morphism} between local shtukas $f\colon(\hat M,\tau_{\hat M})\to(\hat N,\tau_{\hat N})$ over $R$ is a morphism of $R\dbl z\dbr[\tfrac{1}{z}]$-modules $f\colon M[\tfrac{1}{z}]\isoto N[\tfrac{1}{z}]$ with $\tau_{\hat N}\circ\hat\sigma^*f=f\circ\tau_{\hat M}$. It is called a \emph{quasi-isogeny} if it is an isomorphism of $R\dbl z\dbr[\tfrac{1}{z}]$-modules. We denote the $Q_v$-vector space of quasi-morphisms from $\ulHM$ to $\ulHN$ by $\QHom_R(\ulHM,\ulHN)$ and write $\QEnd_R(\ulHM)=\QHom_R(\ulHM,\ulHM)$. 
\end{Definition}

Note that $\Hom_R(\ulHM,\ulHN)$ is a finite free $A_v$-module of rank at most $\rk\ulHM\cdot\rk\ulHN$ by \cite[Corollary~4.5]{HartlKim} and $\QHom_R(\ulHM,\ulHN)=\Hom_R(\ulHM,\ulHN)\otimes_{A_v}Q_v$. Also every quasi-isogeny $f\colon\ulHM\to\ulHN$ induces an isomorphism of $Q_v$-algebras $\QEnd_R(\ulHM)\isoto\QEnd_R(\ulHN)$, $g\mapsto fgf^{-1}$.

\begin{Example}\label{AMotLocSht} 
We assume that the $A$-motive $\ulM=(M, \tau_M)$ has \emph{good reduction}, that is, there exist a pair $\ul\CM = (\CM, \tau_{\CM})$ consisting of a locally free module $\CM$ over $A_R:=A\otimes_{\BF_q} R$ of finite rank and an isomorphism $\tau_\CM:\sigma^*\CM|_{\Spec A_R\setminus\Var(\CJ)} \isoto \CM|_{\Spec A_R\setminus\Var(\CJ)}$ of the associated sheaves outside the vanishing locus $\Var(\CJ)\subset\Spec A_R$ of the ideal $\CJ: = (a\otimes1-1\otimes \gamma( a): \ a\in A)\subset A_R$, such that $\ul\CM\otimes_RL\cong\ulM$. The reduction $\ul\CM \otimes_R \kappa$ is an $A$-motive over $\kappa$ of $A$-characteristic $v= \ker(\gamma: A\to \kappa)$. The pair $\ul\CM$ is called an \emph{$A$-motive over $R$} and a \emph{good model} of $\ulM$.

We consider the $v$-adic completions $A_{v, R}$ of $A_R$ and $\ul\CM \otimes_{A_R}A_{v, R}: = (\CM \otimes_{A_R}A_{v, R}, \tau_\CM\otimes \id)$ of $\ul\CM$. We let $d_v:=[\BF_v:\BF_q]$ and discuss the two cases $d_v = 1$ and $d_v> 1$ separately. If $d_v = 1$, and hence  $q_v = q$ and $\hat\sigma =\sigma$, we have $A_{v, R} = R\dbl z \dbr$, and $\ul\CM \otimes_{A_R}A_{v, R}$ is a local $\hat\sigma$-shtuka over $\Spec R$ which we denote by  $\ulHM_v(\ul\CM)$ and call the \emph{local shtuka at $v$ associated with $\ul\CM$}.

If $d_v> 1$, the situation is more complicated, because $\BF_v \otimes_{\BF_q}R$ and $A_{v,R}$ fail to be integral domains. Namely,
\[
\BF_v\otimes_{\BF_q} R = \prod_{\Gal(\BF_{v}/\BF_q)}\BF_v\otimes_{\BF_{v}}R = 
\prod_{i\in\BZ/d_v\BZ}\BF_v\otimes_{\BF_q} R\,/\,(a\otimes 1-1\otimes \gamma(a)^{q^i}:a\in \BF_{v})
\]
and $\sigma$ transports the $i$-th factor to the $(i+1)$-th factor. In particular $\hat\sigma$ stabilizes each factor. Denote by $\Fa_i$ the ideal of $A_{v,R}$ generated by $\{a\otimes 1-1\otimes \gamma(a)^{q^i}:a\in \BF_{v}\}$. Then
\[
A_{v,R} = \prod_{\Gal(\BF_{v}/\BF_q)}A_v\wh\otimes_{\BF_{v}}R = 
\prod_{i\in\BZ/d_v\BZ}A_{v,R}\,/\,\Fa_i.
\]
Note that each factor is isomorphic to $R\dbl z \dbr$ and the ideals $\Fa_i$ correspond precisely to the places $v_i$ of $C_{\BF_{v}}$ lying above $v$. The ideal $\CJ$ decomposes as follows $ \CJ\cdot A_{v,R}/\Fa_0 = (z-\zeta)$  and  $ \CJ\cdot A_{v,R}/\Fa_i = (1)$ for $i  \neq 0$. We define the \emph{local shtuka at $v$ associated with $\ul\CM$} as $\ulHM_v(\ul\CM):=(\hat M,\tau_{\hat M}):=\bigl(\CM\otimes_{A_R}A_{v,R}/\Fa_0\,,\,(\tau_\CM\otimes1)^{d_v}\bigr)$, where $\tau_\CM^{d_v}:=\tau_\CM\circ\sigma^*\tau_\CM\circ\ldots\circ\sigma^{(d_v-1)*}\tau_\CM$. Of course if $d_v=1$ we get back the definition of $\ulHM_v(\ul\CM)$ given above. Also note if $\ul\CM$ is effective, then $\CM/\tau_\CM(\sigma^*\CM)=\hat M/\tau_{\hat M}(\hat\sigma^*\hat M)$.

The local shtuka $\ulHM_v(\ul\CM)$ allows to recover $\ul\CM\otimes_{A_R}A_{v,R}$ via the isomorphism
\[
\bigoplus_{i=0}^{d_v-1}(\tau_\CM\otimes1)^i\mod\Fa_i\colon\Bigl(\bigoplus_{i=0}^{d_v-1}\sigma^{i*}(\CM\otimes_{A_R}A_{v,R}/\Fa_0),\;(\tau_\CM\otimes1)^{d_v}\oplus\bigoplus_{i\ne0}\id\Bigr)\;\isoto\;\ul\CM\otimes_{A_R}A_{v,R}\,,
\]
because for $i\ne0$ the equality $\CJ\!\cdot\! A_{v,R}/\Fa_i=(1)$ implies that $\tau_\CM\otimes 1$ is an isomorphism modulo $\Fa_i$; see \cite[Propositions~8.8 and 8.5]{BH1} for more details.
\end{Example}

Next we define the $v$-adic realization and the de Rham realization of a local shtuka $\ulHM = (\hat M, \tau_{\hat M})$ over $R$. Since $\tau_{\hat M}$ induces an isomorphism   $\tau_{\hat M}:  \hat\sigma^* \hat M  \otimes_{R\dbl z\dbr}L\dbl z\dbr \isoto \hat M  \otimes_{R\dbl z\dbr}L\dbl z\dbr$, because $z-\zeta\in L\dbl z\dbr\mal$, we can think of $\ulHM\otimes_{R\dbl z\dbr}L\dbl z\dbr$  as an \'etale local shtuka over $L$.

 \begin{Definition}\label{dualTatemodule} The \emph{$v$-adic realization} $\Koh^1_v(\ulHM, A_v)$ of a local $\hat\sigma$-shtuka $\ulHM = (\hat M, \tau_{\hat M})$ is the $\sG_L:=\Gal(L^{\sep}/L)$-module of $\tau$- invariants
\[
\Koh^1_v(\ulHM, A_v)\;:=\; (\hat M \otimes_{R\dbl z\dbr}L^\sep\dbl z\dbr)^{ \tau} : = \{m \in \hat M \otimes_{R\dbl z\dbr}L^\sep\dbl z\dbr:  \tau_{\hat M}( \hat\sigma_{\!\hat M}^* m) = m \} ,
\]
where we set $\hat\sigma_{\!\hat M}^*m:=m\otimes1\in\hat M\otimes_{R\dbl z\dbr,\hat\sigma}R\dbl z\dbr=:\sigma^*M$ for $m\in M$. One also writes sometimes $\check T_v\ulHM=\Koh^1_v(\ulHM, A_v)$ and calls this the \emph{dual Tate module of $\ulHM$}. By \cite[Proposition~4.2]{HartlKim} it is a free $A_v$-module of the same rank as  $\hat M$. We also write $\Koh^1_v(\ulHM,B):=\Koh^1_v(\ulHM,A_v)\otimes_{A_v}B$ for an $A_v$-algebra $B$.

If $\ulM = (M, \tau_M)$ is an $A$-motive over $L$ with good model $\ul\CM$ and $\ulHM = \ulHM_v(\ul\CM)$ is the local shtuka at $v$ associated with $\ul\CM$, then $\Koh^1_v(\ulHM, A_v)$ is by \cite[Proposition~4.6]{HartlKim} canonically isomorphic as a representation of $\sG_L$ to the \emph{$v$-adic realization of $\ulM$}, which is defined as 
\[
\Koh^1_v(\ulM, A_v)\;:=\;\{m\in M\otimes_{A_L}A_{v,L^\sep}\colon \tau_M(\sigma_{\!M}^*m)=m\}, 
\]
where we set $\sigma_{\!M}^*m:=m\otimes1\in M\otimes_{A_R,\sigma}A_R=:\sigma^*M$ for $m\in M$ and where $A_{v,L^\sep}$ is the $v$-adic completion of $A_{L^\sep}$.
\end{Definition}

\begin{Definition}\label{DefDeRham}
Let  $\ulHM = (\hat M, \tau_{\hat M})$  be a local $ \hat\sigma$-shtuka over $R$. We define the \emph{de Rham realizations} of  $\ulHM$ as
\begin{align*}
\Koh^1_{\dR}(\ulHM, R )\;:= \;& \hat\sigma^* \hat{ M}/(z-\zeta)\hat\sigma^* \hat M \;=\;\hat\sigma^* \hat{ M} \otimes_{R\dbl z\dbr, z\mapsto \zeta} R\,, \quad\text{as well as}\\[1mm]
\Koh^1_{\dR}(\ulHM, L \dbl z -\zeta\dbr)\;:=\; & \hat\sigma^* \hat{ M} \otimes_{R\dbl z\dbr}  L\dbl z -\zeta\dbr \quad\text{and}  \\[1mm]
\Koh^1_{\dR}(\ulHM, L )\;:= \;& \hat\sigma^* \hat{ M} \otimes_{R\dbl z\dbr, z\mapsto \zeta} L \; =\;  \Koh^1_{\dR}(\ulHM, L \dbl z -\zeta\dbr)\otimes_{ L \dbl z -\zeta\dbr} L\dbl z -\zeta\dbr/(z -\zeta)\\[1mm]
=\; & \Koh^1_{\dR}(\ulHM, R )\otimes_RL\,.
\end{align*}
It carries the \emph{Hodge-Pink lattice} $\Fq^\ulHM:=\tau_{\hat M}^{-1}(\hat M\otimes_{R\dbl z\dbr}L\dbl z-\zeta\dbr)\subset \Koh^1_{\dR}(\ulHM, L \dbl z -\zeta\dbr)[\tfrac{1}{z-\zeta}]$. We also write $\Koh^1_\dR(\ulHM,B):=\Koh^1_\dR(\ulHM,L\dbl z-\zeta\dbr)\otimes_{L\dbl z-\zeta\dbr}B$ for an $L\dbl z-\zeta\dbr$-algebra $B$.

If $\ulM = (M, \tau_M)$ is an $A$-motive over $L$ with good model $\ul\CM$ and $\ulHM = \ulHM_v(\ul\CM)$ is the local shtuka at $v$ associated with $\ul\CM$ and $d_v = [\BF_v: \BF_q]$ is as in Example \ref{AMotLocSht}, the map 
\[
\sigma^*\tau_M^{{d_v}-1}=\sigma^*\tau_M\circ\sigma^{2*}\tau_M\circ\cdots\circ\sigma^{({d_v}-1)*}\tau_M\colon\sigma^{{d_v}*}M\otimes_{A_R}A_{v,R}/\Fa_0\isoto\sigma^*M\otimes_{A_R}A_{v,R}/\Fa_0
\]
is an isomorphism, because $\tau_M$ is an isomorphism  over $A_{v, R}/\Fa_i$ for all  $i\neq 0$. Therefore it defines canonical isomorphisms of the de Rham realizations 
\begin{align*}
\sigma^*\tau_M^{{d_v}-1}\colon & \Koh^1_\dR\bigl(\ulHM,L\dbl z-\zeta\dbr\bigr)\isoto\Koh^1_\dR\bigl(\ulM,L\dbl z-\zeta\dbr\bigr)\quad\text{and}\\
\sigma^*\tau_M^{{d_v}-1}\colon & \Koh^1_\dR(\ulHM,L)\isoto\Koh^1_\dR(\ulM,L)\,,
\end{align*}
which are compatible with the Hodge-Pink lattices.
\end{Definition}

\begin{Remark}
By \cite[Theorem~4.15]{HartlKim} there is a canonical comparison isomorphism 
\begin{equation}\label{EqHv,dR}
h_{v,\dR}\colon\;\Koh^1_v\bigl(\ulHM,\BC_v\dpl z-\zeta\dpr\bigr)\;\isoto\;\Koh^1_\dR\bigl(\ulHM,\BC_v\dpl z-\zeta\dpr\bigr)
\end{equation}
which is equivariant for the action of $\sG_L$. For our computations below we need an explicit description of $h_{v,\dR}$. It is constructed as follows. The natural inclusion $\Koh^1_v(\ulHM, A_v) \subset \hat M \otimes_{R\dbl z\dbr}L^\sep\dbl z\dbr$ defines a canonical isomorphism of $L^\sep\dbl z\dbr$-modules
\begin{equation}\label{LocTatIso}
\Koh^1_v(\ulHM, A_v) \otimes_{A_v}L^\sep\dbl z\dbr \;\isoto\; {\hat M} \otimes_{R\dbl z\dbr}L^\sep\dbl z\dbr\,,
\end{equation}
which  is  $\sG_L$ and $\tau$-equivariant, where on the left module $\sG_L$ acts on both factors  and $\tau$ is $\id \otimes  \hat\sigma$ and on the right module $\sG_L$ acts only on $L^\sep\dbl z\dbr$ and $\tau$ is $(\tau_{\hat M}\circ  \hat\sigma^*_{\hat M}) \otimes \hat\sigma$. Since $(L^\sep)^{\sG_L} = L$ we obtain 
\[
 {\hat M} \otimes_{R\dbl z\dbr}L\dbl z\dbr \;=\; (\Koh^1_v(\ulHM, A_v)\otimes_{A_{v}}L^\sep\dbl z\dbr)^{\sG_L}.
\]
It turns out, see \cite[Remark~4.3]{HartlKim}, that the isomorphism (\ref {LocTatIso}) extends to an equivariant isomorphism 
\begin{equation}\label{LocTatIso2}
h\colon\, \TS \Koh^1_v(\ulHM, A_v) \otimes_{A_v}L^\sep {\langle{\frac{z}{\zeta}}\rangle} \; \isoto \; {\hat M} \otimes_{R\dbl z\dbr}L^\sep\langle \frac{z}{\zeta}\rangle\,,
\end{equation}
where for an $r\in\BR_{>0}$ we use the notation
\begin{eqnarray*}
L^\sep\langle\tfrac{z}{\zeta^r}\rangle & := & \Bigl\{\,\sum_{i=0}^\infty b_iz^i\colon \; b_i\in L^\sep,\,|b_i|\,|\zeta|^{ri}\to0\;(i\to+\infty)\,\Bigr\}\,.
\end{eqnarray*}
These are subrings of $L^\sep\dbl z\dbr$ and the endomorphism $\hat\sigma\colon\sum_i b_iz^i\mapsto\sum_i b_i^{q_v}z^i$ of $L^\sep\dbl z\dbr$ restricts to a homomorphism $\hat\sigma\colon L^\sep\langle\tfrac{z}{\zeta^r}\rangle\to L^\sep\langle\tfrac{z}{\zeta^{rq_v}}\rangle$. Note that the $\tau$-equivariance of $h$ means $h\otimes\id_{L^\sep\langle\frac{z}{\zeta^{q_v}}\rangle}=\tau_{\hat M}\circ\hat\sigma^*h$. Now the period isomorphism is defined as 
\begin{eqnarray}\label{EqTatIso3}
h_{v,\dR}:=(\tau_{\hat M}^{-1}\circ h)\otimes\id_{\BC_v\dpl z-\zeta\dpr}\colon\;\Koh^1_v\bigl(\ulHM,\BC_v\dpl z-\zeta\dpr\bigr) & \isoto & \hat\sigma^*\hat M\otimes_{R\dbl z\dbr}\BC_v\dpl z-\zeta\dpr\\
& = & \Koh^1_\dR\bigl(\ulHM,\BC_v\dpl z-\zeta\dpr\bigr)\,.\nonumber
\end{eqnarray}
\end{Remark}

\section{Local Shtukas with Complex Multiplication}\label{Amotivescm}
\setcounter{equation}{0}

\begin{Definition}\label{DefCME}
Let $\ulHM$ be a local $\hat\sigma$-shtuka over $R$ and assume that there is  a commutative, semi-simple $Q_v$-algebra $E_v\subset\QEnd_R(\ulHM):=\End_R(\ulHM)\otimes_{A_v}Q_v$ with $\dim_{Q_v}E_v=\rk\ulHM$. Then we say that $\ulHM$ has \emph{complex multiplication} (\emph{by $E_v$}).
\end{Definition}

Here again semi-simple means that $E_v$ is a direct product $E_v = E_{v,1}\times \cdots \times E_{v,s}$ of finite field extensions of $Q_v$. We do \emph{not} assume that $E_v$ is itself a field and in Section~\ref{Amotivescm} we do \emph{not} assume that the $E_{v,i}$ are separable over $Q_v$. We let $\CO_{E_v}$ be the integral closure of $A_v$ in $E_v$. It is a product $\CO_{E_v} = \CO_{E_{v,1}}\times \cdots\times \CO_{E_{v,s}}$ of complete discrete valuation rings where $\CO_{E_{v,i}}$ is the integral closure of $A_v$ in the field $E_{v,i}$. For every $i$ we write $\CO_{E_{v,i}}=\BF_{\tilde v_i}\dbl y_i\dbr$ and set $f_i:=[\BF_{\tilde v_i}:\BF_v]$ and $e_i:=\ord_{y_i}(z)$. Then $f_i\,e_i=[E_{v,i}:Q_v]$ and $e_i$ is divisible by the inseparability degree of $E_{v,i}$ over $Q_v$. Also we write $\tilde q_i:=\#\BF_{\tilde v_i}=q_v^{f_i}$. 

\begin{Proposition}\label{PropCMOE}
If $\ulHM$ has complex multiplication by $E_v$, then there is a local shtuka $\ulHM{}'$ over $R$ quasi-isogenous to $\ulHM$ with $\CO_{E_v}\subset\End_R(\ulHM{}')$.
\end{Proposition}

\begin{proof}
The $A_v$-submodule $T':=\CO_{E_v}\cdot\Koh^1_v(\ulHM,A_v)\subset\Koh^1_v(\ulHM,Q_v)$ is $\sG_L$-invariant and contains $\Koh^1_v(\ulHM,A_v)$. Since $\CO_{E_v}\subset\QEnd_R(\ulHM)=\End_R(\ulHM)\otimes_{A_v}Q_v$ there is an element $a\in A_v$ with $a\cdot\CO_{E_v}\subset\End_R(\ulHM)$, and therefore $a\cdot T'\subset\Koh^1_v(\ulHM,A_v)$ is a finitely generated $A_v$-module, that is an $A_v$-lattice. By \cite[Proposition~4.22]{HartlKim} there is a local shtuka $\ulHM{}'$ and a quasi-isogeny $f\colon\ulHM\to\ulHM{}'$ which maps $T'$ isomorphically onto $\Koh^1_v(\ulHM{}',A_v)$. In particular $\CO_{E_v}$ acts as $\sG_L$-equivariant endomorphisms of $\Koh^1_v(\ulHM{}',A_v)$. Since the functor $\ulHM{}'\mapsto\Koh^1_v(\ulHM{}',A_v)$ from local shtukas to $A_v[\sG_L]$-modules is fully faithful by \cite[Theorem~4.20]{HartlKim}, we see that $\CO_{E_v}\subset\End_R(\ulHM{}')$. 
\end{proof}

\begin{Definition}\label{DefCMOE}
 If $\CO_{E_v}\subset\End_R(\ulHM)$ we say that $\ulHM$ has \emph{complex multiplication by $\CO_{E_v}$}. This makes the underlying module $\hat M$ into a module over the ring $\CO_{E_v,R}:=\CO_{E_v}\otimes_{A_v}R\dbl z\dbr=\CO_{E_v}\wh\otimes_{\BF_v}R$. For $a\in \CO_{E_v}$ note that $a\otimes 1\in\CO_{E_v,R}$ acts  on $\hat M$ as the endomorphism $a$ and on $\hat\sigma^*\hat{ M}$ as the endomorphism $\hat\sigma^*a$ and  $\tau_{\hat M}$ is $\CO_{E_v}$-linear because $a\circ \tau_{\hat M} = \tau_{\hat M}\circ  \hat\sigma^* a $.
\end{Definition}

\begin{Point}
Let us assume that $L$ contains $\psi(E_v)$ for every $\psi\in H_{E_v}$. This implies $\psi(\CO_{E_v})\subset R$ for every $\psi\in H_{E_v}$. Under this assumption let us describe the ring $\CO_{E_v,R}=\prod_{i=1}^s\CO_{E_{v,i},R}$. Fix an $i$ and choose and fix an $\BF_v$-homomorphism $\BF_{\tilde v_i}\into \kappa$. Then $H_{\tilde v_i}:=\Hom_{\BF_v\!}(\BF_{\tilde v_i},\kappa)\cong\BZ/f_i\BZ$ under the map that sends $j\in\BZ/f_i\BZ$ to the homomorphism $(\lambda\mapsto\lambda^{q_v^j})\in H_{\tilde v_i}$. Also
\[
\BF_{\tilde v_i}\otimes_{\BF_v} R \;=\; \prod_{H_{\tilde v_i}}R \;=\;  \prod_{j\in\BZ/f_i\BZ}\BF_{\tilde v_i}\otimes_{\BF_v} R\,/\,(\lambda\otimes 1-1\otimes \lambda^{q_v^j}:\lambda\in \BF_{\tilde v_i})
\]
and  $\hat\sigma^*$ transports the $j$-th factor to the $(j+1)$-th factor. Denote by $\Fb_{i,j}\subset\CO_{E_{v,i}}$ the ideal  generated by $(\lambda\otimes 1-1\otimes\lambda^{q_v^j}:\lambda\in \BF_{\tilde v_i})$. Then
\begin{equation}\label{EqOER}
\CO_{E_{v,i},R}\;:=\;\BF_{\tilde v_i}\dbl y_i\dbr \otimes_{A_v}R\dbl z\dbr \;=\; \prod_{j\in\BZ/f_i\BZ} \BF_{\tilde v_i}\dbl y_i\dbr\otimes_{\BF_{ v}\dbl z\dbr}R\dbl z\dbr/ \Fb_{i,j} \;=\; \prod_{H_{\tilde v_i}}R\dbl y_i\dbr .
\end{equation}
\end{Point}

\begin{Definition}\label{DefIOfPsi}
For every $\psi\in H_{E_v}$ we let $i(\psi)$ be such that $\psi$ factors through the quotient $E_v\onto E_{v,i(\psi)}$ and we let $j(\psi)\in\BZ/f_{i(\psi)}\BZ$ be the element such that $\psi(\lambda)=\lambda^{q_v^{j(\psi)}}$ for all $\lambda\in\BF_{\tilde v_{i(\psi)}}$. Then the morphism $\psi\colon\CO_{E_v}\to R$ equals the composition $\CO_{E_v}\into\CO_{E_v,R}\onto\CO_{E_{v,i(\psi)},R}/\bigl(\Fb_{i(\psi),j(\psi)},\,y_{i(\psi)} - \psi(y_{i(\psi)})\bigr)$ and $H_{E_{v,i}}=\{\psi\in H_{E_v}\colon i(\psi)=i\}$.
\end{Definition}

\begin{Lemma}\label{LemFactors}
Let $p^m$ be the inseparability degree of $E_{v,i}$ over $Q_v$. Then in the $j$-th component $R\dbl y_i\dbr$ of \eqref{EqOER} we have
\begin{equation}\label{EqFactors}
z-\zeta\;=\;\epsilon\cdot\prod_{\psi\in H_{E_v}\colon(i,j)(\psi)=(i,j)}\bigl(y_i-\psi(y_i)\bigr)^{p^m}
\end{equation}
for a unit $\epsilon\in R\dbl y_i\dbr$.
\end{Lemma}

\begin{proof}
Set $y'_i:=y_i^{p^m}$ and let $P=P(z,Y)=\sum_{\mu,\nu}b_{\mu\nu}z^\mu Y^\nu\in \BF_{\tilde v_i}\dbl z\dbr[Y]$ with $b_{\mu\nu}\in\BF_{\tilde v_i}$ be the minimal polynomial of $y'_i$ over $\BF_{\tilde v_i}\dpl z\dpr$. It is an Eisenstein polynomial of degree $e_i/p^m$, because $\BF_{\tilde v_i}\dpl y'_i\dpr$ is purely ramified and separable over $\BF_{\tilde v_i}\dpl z\dpr$ by Lemma~\ref{LemmaInsepExt} in the appendix. In particular $b_{0,\nu}=0$ for $0\le\nu<e_i/p^m$, and $b_{1,0}\ne0$. Consider the polynomials $P^{(j)}(z,Y):=\sum_{\mu,\nu}b_{\mu\nu}^{q_v^j}z^\mu Y^\nu\in \BF_{\tilde v_i}\dbl z\dbr[Y]\subset R\dbl z\dbr[Y]$ and $P^{(j)}(\zeta,Y)\in R[Y]$. If $\psi\in H_{E_v}$ satisfies $(i,j)(\psi)=(i,j)$ then $P^{(j)}\bigl(\zeta,\psi(y'_i)\bigr)=\psi\bigl(P(z,y'_i)\bigr)=\psi(0)=0$. These zeroes $\psi(y'_i)$ of $P^{(j)}(\zeta,Y)$ in $L$ are pairwise different, because if $\psi(y'_i)=\wt\psi(y'_i)$ then $(i,j)(\psi)=(i,j)(\wt\psi)$ implies that $\psi$ and $\wt\psi$ coincide on $E_{v,i}$ and hence on $E_v$. It follows that $P^{(j)}(\zeta,Y)=\prod_{\psi\colon(i,j)(\psi)=(i,j)}(Y-\psi(y'_i))$ in $L[Y]$, whence already in $R[Y]$. In the $j$-th component $R\dbl y_i\dbr$ of \eqref{EqOER} we have $0=\sum_{\mu,\nu}b_{\mu\nu}z^\mu (y'_i)^\nu\otimes 1=\sum_{\mu,\nu}(y'_i)^\nu\otimes b_{\mu\nu}^{q_v^j}z^\mu=P^{(j)}(z,y'_i)$, and we compute 
\begin{eqnarray*}
\prod_{\psi\colon(i,j)(\psi)=(i,j)}\bigl(y_i-\psi(y_i)\bigr)^{p^m} & = &  P^{(j)}(\zeta,y'_i) \\[-3mm]
& = & P^{(j)}(\zeta,y'_i)-P^{(j)}(z,y'_i) \\[1mm]
& = & {\TS\sum\limits_{\mu,\nu}}b_{\mu\nu}^{q_v^j}(\zeta^\mu-z^\mu)(y'_i)^\nu \\
& = & (\zeta-z)\cdot{\TS\sum\limits_{\mu,\nu}}b_{\mu\nu}^{q_v^j}(\zeta^{\mu-1}+\zeta^{\mu-2}z+\ldots+z^{\mu-1})(y'_i)^\nu\,.
\end{eqnarray*}
The factor $\sum_{\mu,\nu}b_{\mu\nu}^{q_v^j}(\zeta^{\mu-1}+\zeta^{\mu-2}z+\ldots+z^{\mu-1})(y'_i)^\nu$ is congruent to $b_{1,0}^{q_v^j}\ne0$ modulo the maximal ideal $(\pi_L,y_i)\subset R\dbl y_i\dbr$ and therefore a unit in $R\dbl y_i\dbr$. This finishes the proof. Note that since 
\[
P^{(j)}(\zeta,y'_i)-P^{(j)}(z,y'_i)\;=\;{\TS\sum\limits_{n\ge1}}\tfrac{1}{n!}\,\tfrac{\partial^n P^{(j)}}{\partial z^n}(z,y'_i)\cdot(\zeta-z)^n\,,
\]
the proof could also be phrased by saying that $\tfrac{\partial P^{(j)}}{\partial z}(z,y'_i)=\sum_{\mu,\nu}\mu b_{\mu\nu}^{q_v^j}z^{\mu-1}(y'_i)^\nu$ lies in $\CO_{E'_{v,i}}\mal$. In fact this partial derivative is congruent to $b_{1,0}^{q_v^j}\ne0$ modulo $y'_i\cdot\CO_{E'_{v,i}}$.
\end{proof}

Let us draw a direct corollary from the proof of this lemma. To formulate it, recall that if $E_{v,i}$ is separable over $Q_v$, the different $\FD_{E_{v,i}/Q_v}$ of $E_{v,i}$ over $Q_v$ is defined as the ideal in $\CO_{E_{v,i}}$ which annihilates the module $\Omega^1_{\CO_{E_{v,i}}/A_v}$ of relative differentials.

\begin{Corollary}\label{CorDifferent}
If $E_{v,i}$ is separable over $Q_v$ then $\FD_{\phi(E_{v,i})/Q_v}=\Bigl(\tfrac{z-\zeta}{y_i-\phi(y_i)}\Big|_{y_i=\phi(y_i)}\Bigr)$ in $\CO_{\phi(E_{v,i})}$ for every $\phi\in H_{E_{v,i}}$.
\end{Corollary}

\begin{proof}
By \cite[Chapter~III, \S\,4, Proposition~8]{SerreLF} the different is multiplicative, that is $\FD_{E_{v,i}/Q_v}=\FD_{E_{v,i}/\BF_{\tilde v_i}\dpl z\dpr}\cdot\FD_{\BF_{\tilde v_i}\dpl z\dpr/Q_v}$. Moreover, $\FD_{\BF_{\tilde v_i}\dpl z\dpr/Q_v}=1$ because $\BF_{\tilde v_i}\dbl z\dbr$ is unramified over $A_v$. As in the proof of the preceding lemma let $P(z,Y)$ be the minimal polynomial of $y'_i$ over $\BF_{\tilde v_i}\dpl z\dpr$ and note that $y'_i=y_i$ under our separability assumption. Then $\tfrac{\partial P}{\partial z}(z,y_i)\in\CO_{E_{v,i}}\mal$ and
\[
\Omega^1_{\CO_{E_{v,i}}/\BF_{\tilde v_i}\dbl z\dbr} \;=\; \CO_{E_{v,i}}\langle dz,dy_i \rangle\big/\bigl(dz,\,\tfrac{\partial P}{\partial z}(z,y_i)dz+\tfrac{\partial P}{\partial Y}(z,y_i)dy_i\bigr)\;=\; \CO_{E_{v,i}}\cdot dy_i/(\tfrac{\partial P}{\partial Y}(z,y_i)dy_i)\,.
\]
We write $z=f(y_i)\in\BF_{\tilde v_i}\dbl y_i\dbr$. Then $0=\tfrac{d}{dy_i}P(f(y_i),y_i)=\tfrac{\partial P}{\partial z}(f(y_i),y_i)\tfrac{df(y_i)}{dy_i}+\tfrac{\partial P}{\partial Y}(f(y_i),y_i)$ and $\FD_{E_{v,i}/Q_v}=\bigl(\tfrac{\partial P}{\partial Y}(z,y_i)\bigr)=\bigl(\tfrac{df(y_i)}{dy_i}\bigr)$. Now Lemma~\ref{LemDerivative} in the appendix implies that $\FD_{\phi(E_{v,i})/Q_v}=\phi\bigl(\tfrac{df(y_i)}{dy_i}\bigr)=\Bigl(\tfrac{z-\zeta}{y_i-\phi(y_i)}\Bigr)\Big|_{y_i=\phi(y_i)}$.
\end{proof}

\bigskip

Now we explore the consequences of these decompositions for local shtukas with complex multiplication.

\begin{Proposition}\label{PropFree}
Let $\ulHM=(\hat M,\tau_{\hat M})$ have complex multiplication by $\CO_{E_v}$. Then the $\CO_{E_v,R}$-module $\hat M$ is free of rank $1$. In particular, $\ulHM_i:=\ulHM\otimes_{\CO_{E_v}}\CO_{E_{v,i}}$ is a local $\hat\sigma$-shtuka over $R$ with $\rk\ulHM_i = [E_{v,i}: Q_v]$ and $\ulHM=\bigoplus_{i=1}^s\ulHM_i$.
\end{Proposition}

\begin{proof}
By faithfully flat descent \cite[IV$_2$, Proposition 2.5.2]{EGA}, we may replace $R$ by a finite extension of discrete valuation rings. Therefore it suffices to prove the proposition in the case where $R$ contains $\psi(\CO_{E_v})$ for all $\psi\in H_{E_v}$. In this case $\CO_{E_v,R}$ is a product of two dimensional regular local rings $R\dbl y_i\dbr$ by \eqref{EqOER}. By \cite[\S6, Lemme~6]{Serre58} a finitely generated module $M$ over such a ring is free if and only if it is reflexive, that is $M$ is isomorphic to its bidual $M\dual{}\dual$, where $M\dual=\Hom_{R\dbl y_i\dbr}(M,R\dbl y_i\dbr)$. In particular $M\dual{}\dual$, which is isomorphic to $(M\dual{}\dual)\dual{}\dual$ is free. We apply this to $M:=\hat M\otimes_{\CO_{E_v,R}}R\dbl y_i\dbr$ and consider the base changes $M\otimes_{R\dbl y_i\dbr}L\dbl y_i\dbr=M\otimes_{R\dbl z\dbr}L\dbl z\dbr$ and $M\otimes_{R\dbl y_i\dbr}R\dbl y_i\dbr[\tfrac{1}{y_i}]=M\otimes_{R\dbl z\dbr}R\dbl z\dbr[\tfrac{1}{z}]$. Like $L\dbl y_i\dbr$ also $R\dbl y_i\dbr[\tfrac{1}{y_i}]$ is a principal ideal domain, because it is a factorial ring of dimension $1$. Using \cite[Proposition~2.10]{Eisenbud} and that both base changes of $M$ are torsion free, whence free, we see that the canonical morphism $M\to M\dual{}\dual$ is an isomorphism after both base changes. Since $R\dbl z\dbr=L\dbl z\dbr\cap R\dbl z\dbr[\tfrac{1}{z}]\subset L\dpl z\dpr$ and $M$ and $M\dual{}\dual$ are free $R\dbl z\dbr$-modules, $M$ equals the intersection $(M\otimes_{R\dbl z\dbr}L\dbl z\dbr)\cap (M\otimes_{R\dbl z\dbr}R\dbl z\dbr[\tfrac{1}{z}])$ inside $M\otimes_{R\dbl z\dbr}L\dpl z\dpr$, and likewise for $M\dual{}\dual$. This shows that $M\to M\dual{}\dual$ is an isomorphism and $M$ is free over $R\dbl y_i\dbr$.

It remains to compute the rank. Let $r_{i,j}:=\rk_{R\dbl y_i\dbr}(\hat M\otimes_{\CO_{E_v,R}}\CO_{E_{v,i},R}/\Fb_{i,j})$ for all $i=1,\ldots,s$ and all $j\in\BZ/f_i\BZ$. Then $\sum_{i,j}r_{i,j}\cdot e_i=\rk\ulHM$. We first prove that for a fixed $i$ all $r_{i,j}$ are equal. Since $(\hat\sigma^*\hat M)\otimes_{\CO_{E_v,R}}\CO_{E_{v,i},R}/\Fb_{i,j}=\hat\sigma^*(\hat M\otimes_{\CO_{E_v,R}}\CO_{E_{v,i},R}/\Fb_{i,j-1})\cong R\dbl y_i\dbr^{r_{i,j-1}}$, we can write the isomorphism $\tau_{\hat M} :  \hat\sigma^* \hat M[\frac{1}{z-\zeta}] \isoto \hat M[\frac{1}{z-\zeta}]$ in the form
\[
\prod_{j}R\dbl y_i\dbr[\tfrac{1}{z-\zeta}]^{r_{i,j-1}} \;\isoto\; \prod_{j}R\dbl y_i\dbr[\tfrac{1}{z-\zeta}]^{r_{i,j}},
\]
which gives us $r_{i,j-1}=r_{i,j}=:r_i$ for all $j$, and hence $\sum_ir_i\,f_i\,e_i=\rk\ulHM=\dim_{Q_v}E_v=\sum_i\dim_{Q_v}E_{v,i}=\sum_if_i\,e_i$. Thus if we prove that $r_i\ne0$ then all $r_i$ must be $1$ and so $\hat M$ is a free $\CO_{E_v,R}$-module of rank $1$ and $\rk\ulHM_i=f_i\,e_i=[E_{v,i}: Q_v]$. Now $r_i=0$ means that $\hat M\otimes_{\CO_{E_v}}\CO_{E_{v,i}}=(0)$, and hence $E_{v,i}$ acts as zero on $\ulHM$ in contradiction to $E_v\subset\QEnd_R(\ulHM)$. This finishes the proof.
\end{proof}

\begin{Proposition}\label{PropCMTateMod}
If $\ulHM$ has complex multiplication by a commutative semi-simple $Q_v$-algebra $E_v$ then $\Koh^1_v(\ulHM,Q_v)$ is a free $E_v$-module of rank $1$ and $\Koh^1_{\dR}(\ulHM, L\dbl z-\zeta\dbr)$ is a free $E_v\otimes_{Q_v} L\dbl z-\zeta\dbr$-module of rank one, where the homomorphism $Q_v=\BF_v\dpl z\dpr\to L\dbl z-\zeta\dbr$ is given by $z\mapsto z=\zeta+(z-\zeta)$. If we assume that $L\supset\psi(E_v)$ for all $\psi\in H_{E_v}$ then the decomposition \eqref{EqDecomp1} induces a decomposition
\begin{equation}
\label{EqPropCMTateMod1} \Koh^1_\dR(\ulHM,L\dbl z-\zeta\dbr) \;=\; \bigoplus_{\psi\in H_{E_v}}\Koh^\psi(\ulHM,L\dbl y_{i(\psi)}-\psi(y_{i(\psi)})\dbr),
\end{equation}
where $\Koh^\psi(\ulHM,L\dbl y_{i(\psi)}-\psi(y_{i(\psi)})\dbr)$ is free of rank $1$ over $L\dbl y_{i(\psi)}-\psi(y_{i(\psi)})\dbr$. In particular,
\begin{equation}\label{EqPropCMTateMod2} 
\Koh^1_\dR(\ulHM,L) \;=\; \bigoplus_{\psi\in H_{E_v}}\Koh^\psi(\ulHM,R)\otimes_RL
\end{equation}
is the decomposition into generalized eigenspaces of the $E_v$-action. Here 
\[
\Koh^\psi(\ulHM,R)\;:=\; \bigl\{\omega \in \Koh^1_{\dR}(\ulHM, R)\colon \bigl([a]^*-\psi(a)\bigr)^{[E_{v,i(\psi)}:Q_v]_\insep}\cdot \omega = 0 \es\forall\; a \in E_v\cap\End_R(\ulHM) \bigr\}
\]
is a free $R$-module of rank equal to the inseparability degree $[E_{v,i(\psi)}:Q_v]_\insep$ of $E_{v,i(\psi)}$ over $Q_v$. 
\end{Proposition}

\begin{proof}
By the faithfulness of the functor $\ulHM\to\Koh^1_v(\ulHM,Q_v)$ we have $E_v\subset\End_{Q_v}\Koh^1_v(\ulHM,Q_v)$. So the first statement follows from \cite[Lemma~7.2]{BH2}.

Since $\Koh^1_{\dR}(\ulHM, L\dbl z-\zeta\dbr)$ is an isogeny invariant, we may by Proposition~\ref{PropCMOE} assume that $\CO_{E_v}\subset\End_R(\ulHM)$ and then $\hat M$ is free of rank $1$ over $\CO_{E_v,R}$ by Proposition~\ref{PropFree}. It follows that $\Koh^1_\dR(\ulHM,L\dbl z-\zeta\dbr):=\hat\sigma^*\hat M\otimes_{R\dbl z\dbr}L\dbl z-\zeta\dbr\cong E_v\otimes_{Q_v} L\dbl z-\zeta\dbr$. Now we use Lemma~\ref{LemDecompdR}. In particular, \eqref{EqPropCMTateMod2} and the statement about $\Koh^\psi(\ulHM,R)$ follow from \eqref{EqDecomp2} and the equation $\Koh^1_\dR(\ulHM,L)=\Koh^1_\dR(\ulHM,R)\otimes_RL$.
\end{proof}

The proposition allows us to make two definitions.

\begin{Definition}\label{DefValuationOfomega}
Let $\ulHM$ have complex multiplication by $\CO_{E_v}$ and assume that $L\supset\psi(E_v)$ for all $\psi\in H_{E_v}$. Fix a $\psi\in H_{E_v}$ and let $i:=i(\psi)$. Let $\omega^\circ_\psi\in\Koh^\psi(\ulHM,L\dbl y_{i}-\psi(y_{i})\dbr)$ be an $L\dbl y_{i}-\psi(y_{i})\dbr$-generator whose reduction $\omega^\circ_\psi\mod(y_i-\psi(y_i)) \in \Koh^\psi(\ulHM, L)\big/(y_i-\psi(y_i))\Koh^\psi(\ulHM, L)$ is a generator of the free $R$-module of rank one $\Koh^\psi(\ulHM,R)\big/(y_i-\psi(y_i))\Koh^\psi(\ulHM,R)$. Such a $\omega^\circ_\psi$ is uniquely determined up to multiplication by an element of $R\mal+(y_i-\psi(y_i))L\dbl y_{i}-\psi(y_{i})\dbr$. Note that if $E_{v,i}$ is separable over $Q_v$ then $y_i-\psi(y_i)$ acts trivially on $\Koh^\psi(\ulHM, L)$ and $\Koh^\psi(\ulHM,R)$ is a free $R$-module of rank $1$. Also $L\dbl y_i-\phi(y_i)\dbr=L\dbl z-\zeta\dbr$.

If $\omega_\psi\in \Koh^\psi(\ulHM,L\dbl y_{i}-\psi(y_{i})\dbr)$ is any generator, there is an element $x\in L\dbl y_{i}-\psi(y_{i})\dbr\mal$ with $\omega_\psi=x\,\omega^\circ_\psi$. We define the \emph{valuation of $\omega_\psi$} as $v(\omega_\psi):=v\bigl(x\mod y_{i}-\psi(y_{i})\bigr)$. {\color{red} (NOTE THAT THIS DEFINITION IS WRONG; SEE ERRATUM~\ref{Erratum2})} It only depends on the image of $\omega_\psi$ in $\Koh^1_\dR(\ulHM,L)$ and is also independent of the choice of $\omega^\circ_\psi$. 

Note that if $\ulM = (M, \tau_M)$ is an $A$-motive over $L$ with good model $\ul\CM$ over $R$, and $\ulHM = \ulHM_v(\ul\CM)$ is the local shtuka at $v$ associated with $\ul\CM$ as in Example \ref{AMotLocSht}, then for an $L\dbl y_{i}-\psi(y_{i})\dbr$-generator $\omega_\psi\in \Koh^\psi(\ulM,L\dbl y_{i}-\psi(y_{i})\dbr)=\Koh^\psi(\ulHM,L\dbl y_{i}-\psi(y_{i})\dbr)$ the present definition of $v(\omega_\psi)$ coincides with the definition of $v(\omega_\psi)$ from \eqref{EqValuationOfOmegaPsi}.
\end{Definition}

\begin{Definition} \label{DefCMType}
A \emph{local CM-type} at $v$  is a pair $(E_v, \Phi)$ with $E_v$ a semisimple commutative $Q_v$-algebra and $\Phi = (d_{\psi})_{\psi \in H_{E_v}}$ a tuple of integers $d_\psi\in\BZ$.

If $\ulHM$ is a local shtuka with complex multiplication by a commutative semi-simple $Q_v$-algebra $E_v$ and if $L\supset\psi(E_v)$ for all $\psi\in H_{E_v}$ then the Hodge-Pink lattice $\Fq^\ulHM=\tau_{\hat M}^{-1}(\hat M\otimes_{R\dbl z\dbr}L\dbl z-\zeta\dbr)$ of $\ulHM$ satisfies $\Fq^\ulHM=\prod_{\psi\in H_{E_v}}\bigl(y_{i(\psi)}-\psi(y_{i(\psi)})\bigr)^{-d_\psi}\Koh^\psi(\ulHM,L\dbl y_{i(\psi)}-\psi(y_{i(\psi)})\dbr)$ for integers $d_\psi$ under the decomposition~\eqref{EqPropCMTateMod1}. We call $\Phi=(d_\psi)_{\psi\in H_{E_v}}$ the \emph{local CM-type} of $\ulHM$.

Note that if $\ulM = (M, \tau_M)$ is an $A$-motive over $L$ with good model $\ul\CM$ over $R$, and $\ulHM = \ulHM_v(\ul\CM)$ is the local shtuka at $v$ associated with $\ul\CM$ as in Example \ref{AMotLocSht}, and $E_v:=E\otimes_Q Q_v$, then we see from the isomorphism $\Koh^\psi(\ulM,L\dbl y_{i(\psi)}-\psi(y_{i(\psi)})\dbr)=\Koh^\psi(\ulHM,L\dbl y_{i(\psi)}-\psi(y_{i(\psi)})\dbr)$ that the local CM-type of $\ulHM$ is equal to the CM-type of $\ulM$ under the identification $H_E\isoto H_{E_v}$, which extends $\psi\colon E\to Q^\alg\subset Q_v^\alg$ to the completion $\psi\colon E_v\to Q_v^\alg$.
\end{Definition}

\section{Periods of Local Shtukas with Complex Multiplication}\label{SectPeriods}
\setcounter{equation}{0}

\begin{Point}\label{Point4.1}
In this section we let $\ulHM$ be a local $\hat\sigma$-shtuka over $R$ with complex multiplication by $\CO_{E_v}$ where $E_v$ is a commutative semi-simple $Q_v$-algebra as in the preceding section. From Theorem~\ref{ThmOmega} on we assume that the factors $E_{v,i}$ of $E_v$ are \emph{separable} field extensions of $Q_v$. Throughout we assume that $L\supset\psi(E_v)$ for all $\psi\in H_{E_v}$. Using Proposition~\ref{PropFree} we may choose a basis of $\ulHM$ and write it under the decomposition \eqref{EqOER} as 
\[
\ulHM\;\cong\;\prod_i\prod_{j\in\BZ/f_i\BZ}(R\dbl y_i\dbr,\tau_{i,j})\qquad\text{with}\quad \tau_{i,j}\in R\dbl y_i\dbr[\tfrac{1}{z-\zeta}]\mal\,.
\]
Let $c\in\BN_0$ be such that $(z-\zeta)^c\tau_{ij}, (z-\zeta)^c\tau_{ij}^{-1}\in R\dbl y_i\dbr$. Since the $y_i-\phi(y_i)$ for $\phi\in H_{E_v}$ with $(i,j)(\phi)=(i,j)$ are prime elements in the factorial ring $R\dbl y_i\dbr$, Lemma~\ref{LemFactors} applied to $(z-\zeta)^c\tau_{ij}\cdot(z-\zeta)^c\tau_{ij}^{-1}=(z-\zeta)^{2c}$ shows that 
\begin{equation}\label{EqTauij}
\tau_{i,j} \;=\; \epsilon_{i,j}\cdot \prod_{\phi\in H_{E_v}\colon (i,j)(\phi) = (i,j)}\bigl(y_i-\phi(y_i)\bigr)^{d_{\phi}}
\end{equation}
for a unit $\epsilon_{i,j}\in R\dbl y_i\dbr\mal$ and integers $d_\phi\in\BZ$. By Definition~\ref{DefCMType} the tuple $\Phi=(d_\phi)_\phi$ is the local CM-type of $\ulHM$.

Note that we can view $\ulHM$ as the tensor product $\ulHM_{E_v,0}\otimes\bigotimes_\phi\ulHM_{E_v,\phi}{}^{\otimes d_\phi}$ over $\CO_{E_v,R}$ of $\ulHM_{E_v,0}:=(\CO_{E_v,R},\tau_0=\prod_{i,j}\epsilon_{i,j})$ and all the $d_\phi$-th powers of $\ulHM_{E_v,\phi}:=(\CO_{E_v,R},\prod_{i,j}\tau_{\phi,i,j})$ where
\[  \tau_{\phi,i,j} =
    \begin{cases}
            1 &         \text{if } (i,j)\ne (i,j)(\phi),\\
            y_i-\phi(y_i) &         \text{if } (i,j)=(i,j)(\phi).
    \end{cases} 
\]
Likewise the cohomology realizations decompose as tensor products 
\[
\Koh^1_v(\ulHM,A_v)\;\cong\;\Koh^1_v(\ulHM_{E_v,0},A_v)\otimes\bigotimes_{\phi\in H_{E_v}}\Koh^1_v(\ulHM_{E_v,\phi},A_v)^{\otimes d_\phi}\,, 
\]
where the tensor product is over $\CO_{E_v}$, and 
\[
\Koh^1_\dR(\ulHM,L\dbl z-\zeta\dbr)\;\cong\;\Koh^1_\dR(\ulHM_{E_v,0},L\dbl z-\zeta\dbr)\otimes\bigotimes\limits_{\phi\in H_{E_v}}\Koh^1_\dR(\ulHM_{E_v,\phi},L\dbl z-\zeta\dbr)^{\otimes d_\phi}\,,
\]
where the tensor product is over $E_v\otimes_{Q_v}L\dbl z-\zeta\dbr$. For the purpose of computing the period isomorphism $h_{v,\dR}$, we may therefore treat all factors of $\tau_{i,j}$ separately; see \ref{Point4.4} below.
\end{Point}

\begin{Point}\label{Point4.2}
We first treat the case of $\ulHM_{E_v,0}=\bigl(\CO_{E_v,R},\tau_0=(\epsilon_{i,j})_{i,j}\bigr)$, where $\epsilon_{i,j}\in R\dbl y_i\dbr\mal$. We compute the $\tau$-invariants $\Koh^1_v(\ulHM_{E_v,0},A_v)$ as the set of tuples $(c_{i,j})_{i,j}$ with $c_{i,j}:=\sum_{n=0}^\infty c_{i,j,n} \,y_i^n\in L^\sep\dbl y_i\dbr$ subject to the condition 
\[
(c_{i,j})_{i,j} \; = \; \tau_0\circ\hat\sigma\bigl((c_{i,j})_{i,j}\bigr)\,,\quad\text{that is}\quad c_{i,j} \;=\; \epsilon_{i,j}\cdot\hat\sigma(c_{i,j-1}) \quad\text{for all $i,j$}.
\]
The latter implies $c_{i,j}=\epsilon_{i,j}\cdot\hat\sigma(\epsilon_{i,j-1})\cdot\ldots\cdot\hat\sigma^{j-1}(\epsilon_{i,1})\cdot\hat\sigma^j(c_{i,0})$ and $c_{i,0}=\epsilon_i\cdot\hat\sigma^{f_i}(c_{i,0})$, where we set $\epsilon_i:=\epsilon_{i,0}\cdot\hat\sigma(\epsilon_{i,f_i-1})\cdot\ldots\cdot\hat\sigma^{f_i-1}(\epsilon_{i,1})=\sum_{n=0}^\infty b_{i,n}\,y_i^n\in R\dbl y_i\dbr\mal$. In particular $b_{i,0}\in R\mal$. The resulting formulas for the coefficients
\[
c_{i,0,0} \;=\;b_{i,0}\cdot c_{i,0,0}^{\tilde q_i} \qquad\text{and}\qquad c_{i,0,n} -b_{i,0}\cdot c_{i,0,n}^{\tilde q_i} \;=\;{\TS\sum\limits_{\ell=1}^n}b_{i,\ell}\cdot c_{i,0,n-\ell}^{\tilde q_i},
\]
where $\tilde q_i=q_v^{f_i}$, lead to the formulas
\[
c_{i,0,0}^{\tilde q_i-1} \;=\;b_{i,0}^{-1} \qquad\text{and}\qquad \frac{c_{i,0,n}}{c_{i,0,0}} -\bigl(\frac{c_{i,0,n}}{c_{i,0,0}}\bigr)^{\tilde q_i} \;=\;\sum_{\ell=1}^n\frac{b_{i,\ell}}{b_{i,0}}\cdot \bigl(\frac{c_{i,0,n-\ell}}{c_{i,0,0}}\bigr)^{\tilde q_i}\,,
\]
which have solutions $c_{i,0,n}\in \CO_{L^\sep}$ with $c_{i,0,0}\in \CO_{L^\sep}\mal$. In particular, the field extension of $L$ generated by the $c_{i,j,n}$ is unramified. Then $(c_{i,j})_{i,j}$ is an $\CO_{E_v}$-basis of $\Koh^1_v(\ulHM_{E_v,0},A_v)$. Under the period isomorphism $h_{v,\dR}$ it is mapped to 
\[
(\epsilon_{i,j}^{-1}c_{i,j})_{i,j}\;\in\;\bigl(\CO_{E_v}\otimes_{A_v}\CO_{\BC_v}\dbl z\dbr\bigr)\mal\;\subset\; E_v\otimes_{Q_v}\BC_v\dbl z-\zeta\dbr\;=\;\Koh^1_\dR(\ulHM_{E_v,0},\BC_v\dbl z-\zeta\dbr)\,.
\]
\end{Point}

\begin{Point}\label{Point4.3}
Next we compute the period isomorphism for the local shtuka $\ulHM_{E_v,\phi}$ from above. For an element $0\ne\xi\in(\pi_L)\subset R$ we consider the equation
\begin{equation}\label{EqTPlus}
\hat\sigma^{f_i}(\tplus{y_i,\xi})\;=\;(y_i-\xi)\cdot\tplus{y_i,\xi}\qquad\text{for}\qquad\tplus{y_i,\xi}\;:=\;\sum_{n=0}^\infty \tplusminus_n y_i^n\;\in\;L^\sep\dbl y_i\dbr\,.
\end{equation}
The equation can be solved by taking $\tplusminus_n\in L^\sep$ with $\tplusminus_0^{\tilde q_i-1}=-\xi$ and $\tplusminus_n^{\tilde q_i}+\xi \tplusminus_n=\tplusminus_{n-1}$. This implies that $|\tplusminus_n|=|\xi|^{\tilde q_i^{-n}/(\tilde q_i-1)}<1$ and $\tplusminus_n\in\CO_{L^\sep}$. Note that this solution is not unique, but that every other solution $\ttplus{y_i,\xi}$ of \eqref{EqTPlus} is obtained by multiplying $\tplus{y_i,\xi}$ by an element of $\BF_{\tilde v_i}\dbl y_i\dbr=\CO_{E_{v,i}}$, because $\hat\sigma^{f_i}\Bigl(\tfrac{\ttplus{y_i,\xi}}{\tplus{y_i,\xi}}\Bigr)=\tfrac{(y_i-\xi)\cdot\ttplus{y_i,\xi}}{(y_i-\xi)\cdot\tplus{y_i,\xi}}=\tfrac{\ttplus{y_i,\xi}}{\tplus{y_i,\xi}}\in L^\sep\dbl y_i\dbr$ is invariant under $\hat\sigma^{f_i}$ and hence lies in $\BF_{\tilde v_i}\dbl y_i\dbr$.

According to the decomposition \eqref{EqOER} the $\tau$-invariants $\check u\in\Koh^1_v(\ulHM_{E_v,\phi},A_v)$ of $\ulHM_{E_v,\phi}$ have the form $\check u=(\check u_{i,j})_{i,j}\in\prod_{i,j}L^\sep\dbl y_i\dbr$ with $\check u=\tau_\phi\cdot\hat\sigma(\check u)$, that is 
\[
\check u_{i,j}\;=\;\begin{cases}
            \hat\sigma(\check u_{i,j-1}) &         \text{if } (i,j)\ne (i,j)(\phi)\,,\\
            \bigl(y_i-\phi(y_i)\bigr)\cdot\hat\sigma(\check u_{i,j-1}) &         \text{if } (i,j)=(i,j)(\phi)\,.
    \end{cases} 
\]
For $j,j'\in\BZ/f_i\BZ$ we denote by $(j,j')$ the representative of  $j-j'$ in $\{ 0,\ldots, f_i-1\}$. This implies that $\check u_{(i,j)(\phi)}=\bigl(y_{i(\phi)}-\phi(y_{i(\phi)})\bigr)\cdot\hat\sigma^{f_{i(\phi)}}(\check u_{(i,j)(\phi)})$, and $\check u_{i(\phi),j}= \hat\sigma^{(j,j(\phi))}(\check u_{(i,j)(\phi)})$, as well as $\check u_{i,j}\in\BF_{\tilde v_i}\dbl y_i\dbr$ for all $i\ne i(\phi)$ and all $j$. In particular an $\CO_{E_v}$-basis of $\Koh^1_v(\ulHM_{E_v,\phi},A_v)$ is given by
\begin{equation}\label{EqTateModMPhi}
\check u\;=\;(\check u_{i,j})_{i,j}\qquad\text{with}\qquad\check u_{i,j}\;=\;\hat\sigma^{(j,j(\phi))}(\tplus{y_i,\phi(y_i)})^{-\delta_{i,i(\phi)}}\;=\;(\tplus{y_i,\phi(y_i)^{q_v^{(j,j(\phi))}}})^{-\delta_{i,i(\phi)}}
\end{equation}
where $\delta_{i,i(\phi)}$ is the Kronecker $\delta$. The comparison isomorphism $h_{v,\dR}$ sends this $\check u$ to the element
\begin{equation}\label{EqPeriodMPhi}
\tau_\phi^{-1}\cdot \check u\;=\;\biggl(\Bigl(\bigl(y_i-\phi(y_i)\bigr)^{\delta_{j,j(\phi)}}\cdot\hat\sigma^{(j,j(\phi))}(\tplus{y_i,\phi(y_i)})\Bigr)^{-\delta_{i,i(\phi)}}\biggr)_{i,j}
\end{equation}
of $E_v\otimes_{Q_v}\BC_v\dpl z-\zeta\dpr\;=\;\Koh^1_\dR\bigl(\ulHM_{E_v,\phi},\BC_v\dpl z-\zeta\dpr\bigr)$.
\end{Point}

\begin{Point}\label{Point4.4}
Putting everything together we see that our $\ulHM\cong(\CO_{E_v,R},\prod_{i,j}\tau_{i,j})$ with $\tau_{i,j}$ from \eqref{EqTauij} has 
\[
\check u\;=\;(\check u_{i,j})_{i,j}\;=\;\Bigl(c_{i,j}\cdot\!\prod_{\phi\in H_{E_{v,i}}}\hat\sigma^{(j,j(\phi))}(\tplus{y_i,\phi(y_i)})^{-d_\phi}\Bigr)_{i,j}
\]
as an $\CO_{E_v}$-basis of $\Koh^1_v(\ulHM,A_v)\cong\Koh^1_v(\ulHM_{E_v,0},A_v)\otimes\bigotimes_{\phi\in H_{E_v}}\Koh^1_v(\ulHM_{E_v,\phi},A_v)^{\otimes d_\phi}$, where the tensor product is over $\CO_{E_v}$. Under $h_{v,\dR}$ this $\check u$ is mapped to the element
\begin{equation}\label{EqPeriodM}
\tau_{\hat M}^{-1}\cdot \check u\;=\;\biggl(\epsilon_{i,j}^{-1}c_{i,j}\cdot\!\prod_{\phi\in H_{E_{v,i}}}\Bigl(\bigl(y_i-\phi(y_i)\bigr)^{\delta_{j,j(\phi)}}\cdot\hat\sigma^{(j,j(\phi))}(\tplus{y_i,\phi(y_i)})\Bigr)^{-d_\phi}\biggr)_{i,j}
\end{equation}
of $E_v\otimes_{Q_v}\BC_v\dpl z-\zeta\dpr\;\cong\;\Koh^1_\dR\bigl(\ulHM,\BC_v\dpl z-\zeta\dpr\bigr)\cong\Koh^1_\dR(\ulHM_{E_v,0},\BC_v\dbl z-\zeta\dbr)\otimes\bigotimes\limits_{\phi\in H_{E_v}}\Koh^1_\dR(\ulHM_{E_v,\phi},\BC_v\dbl z-\zeta\dbr)^{\otimes d_\phi}$, where the tensor product is over $E_v\otimes_{Q_v}\BC_v\dbl z-\zeta\dbr$.
\end{Point}

\begin{Remark}\label{RemLubinTate}
Note that $\ulHM_{E_v,\phi}\otimes_{\CO_{E_v}}\CO_{E_{v,i}}$ with $i=i(\phi)$ is the local $\hat\sigma$-shtuka associated with a Lubin-Tate formal group, and so our treatment is analogous to Colmez's \cite[\S\,I.2]{Colmez93}. Namely, let $\hat G =\hat\BG_{a,R}= \Spf R\dbl X\dbr$ be the formal additive group over $R$ with an action of $\CO_{E_{v,i}}=\BF_{\tilde v_{i}}\dbl y_{i}\dbr$ given by 
\begin{eqnarray*}
[\lambda]:  X & \longmapsto & \phi(\lambda)\cdot X \;=\; \lambda^{q_v^{j(\phi)}}\!\cdot X \qquad\text{for }\lambda\in\BF_{\tilde v_{i}}\,, \\[2mm]
\,[y_{i}]: X & \longmapsto & X^{\tilde q_i} + \phi(y_{i}) \cdot X\,.
\end{eqnarray*} 
Then $\hat G$ is the Lubin-Tate formal group over $R$ associated with $\CO_{\phi(E_{v,i})}$; see \cite{LT}. It is a \emph{$z$-divisible local Anderson module} in the sense of \cite[Definition~7.1]{HartlSingh}. For an element $a\in\CO_{E_{v,i}}$ let $\hat G[a]:=\ker[a]$. Under the anti-equivalence between $z$-divisible local Anderson modules and effective local $\hat\sigma$-shtukas over $S= \Spec R$ from \cite[Theorem~8.3]{HartlSingh} the associated local shtuka is 
\[
\hat M \;:=\; \hat M(\hat G) \;:=\; \invlim[n]\, \Hom_R(\hat G[z^n], \BG_{a,R}) \;=\; \invlim[n]\, \Hom_R(\hat G[y_i^{ne_i}], \BG_{a,R}) \;=\; \bigoplus_{k=0}^{ f_i-1}R\dbl y_i\dbr \tau^k
\]
with $\tau^0:=\id\colon \hat G \isoto \hat \BG_{a,R}$ and $\tau^k:= \Frob_{ q_v,\hat\BG_{a,R}}\circ\tau^0\colon X\mapsto X^{q_v^k}$. It is an $\CO_{E_{v,i},R}$-module via the $\CO_{E_{v,i}}$-action on $\hat G[z^n]$ and the $R$-action on $\BG_{a,R}$, and is equipped with the Frobenius $\tau_{\hat M}\colon\hat\sigma^*\hat M\to\hat M$ given by $\hat\sigma_{\!\hat M}^*m\mapsto \Frob_{ q_v, \BG_{a,R}}\circ m$ for $m\in\hat M$. We set $\ulHM(\hat G):=(\hat M,\tau_{\hat M})$. In particular, we see that $\lambda\in\BF_{\tilde v_i}$ acts on $R\dbl y_i\dbr \tau^k$ as $\lambda^{q_v^{k+j(\phi)}}$ and so $\hat M/\Fb_{i,j}\hat M=R\dbl y_i\dbr \tau^{(j,j(\phi))}$ under the decomposition \eqref{EqOER}. Since $\tau_{\hat M}(\hat\sigma_{\!\hat M}^*\tau^{f_i-1})=\tau^{f_i}=[y_i]-\phi(y_i)\colon X\mapsto X^{\tilde q_i}=\bigl([y_i]-\phi(y_i)\bigr)(X)$, we see that
\[  \tau_{\hat M}\;=\;(\tau_{\hat M,j})_j\qquad\text{with}\qquad \tau_{\hat M,j} =
    \begin{cases}
            1 &         \text{if } j\ne j(\phi)\,,\\
            y_i-\phi(y_i) &         \text{if } j=j(\phi)\,,
    \end{cases} 
\]
that is, $\ulHM(\hat G)=\ulHM_{E_v,\phi}\otimes_{\CO_{E_v}}\CO_{E_{v,i}}$.

Moreover, if we want to also consider the other components of $\ulHM_{E_v,\phi}$ for $i\ne i(\phi)$ we take the divisible local Anderson module $\hat G_{E_v,\phi}:=\hat G\times\prod_{i\ne i(\phi)}(E_{v,i}/\CO_{E_{v,i}})_R$. It has local shtuka $\ulHM(\hat G_{E_v,\phi})=\ulHM(\hat G)\oplus\bigoplus_{i\ne i(\phi)}\ulHM\bigl((E_{v,i}/\CO_{E_{v,i}})_R\bigr)=\ulHM_{E_v,\phi}$, because $\ulHM\bigl((E_{v,i}/\CO_{E_{v,i}})_R\bigr)=(\CO_{E_{v,i},R},\tau=1)$.
\end{Remark}

\begin{Point}
We want to describe the Galois action of $\sG_L$ on $\Koh^1_v(\ulHM_{E_v,\phi},A_v)$. Recall from \cite[Definition~4.8, Proposition~4.9 and Remark~4.10]{HartlKim} that the Tate module of $\hat G$ is defined as $T_v\hat G:=\Hom_{A_v}\bigl(Q_v/A_v,\hat G(L^\sep)\bigr)$ and that there is a perfect pairing of $A_v$-modules
\begin{equation}\label{EqPairing}
T_v\hat G\times \Koh^1_v(\ulHM(\hat G),A_v)\;\longto\;\Hom_{\BF_v}(Q_v/A_v,\BF_v)\,,\quad(f, m) \longmapsto m\circ f\,,
\end{equation}
which is equivariant for the actions of $\sG_L$ and $\End_R\bigl(\ulHM(\hat G)\bigr)=\End_R(\hat G)^{\rm op}$. Here the $A_v$-module $\Hom_{\BF_v}(Q_v/A_v,\BF_v)\cong\wh\Omega^1_{A_v/\BF_v}\cong\BF_v\dbl z\dbr dz$ is free of rank one; see \cite[Equation~(4.5) before Proposition~4.9]{HartlKim}. We have already computed $\Koh^1_v(\ulHM_{E_v,\phi},A_v)=E_{v,i}\cdot(\check u_{i,j})_{i,j}$ in \eqref{EqTateModMPhi}. We will now compute $T_v\hat G_{E_v,\phi}$ and the action of $\sG_L$ on both $T_v\hat G_{E_v,\phi}$ and $\Koh^1_v(\ulHM_{E_v,\phi},A_v)$. Let again $i=i(\phi)$. Since $\CO_{E_{v,i}}$ acts on $\hat G(L^\sep)$ we have
\begin{align*}
T_v\hat G \es=\es & \Hom_{\CO_{E_{v,i}}}\bigl(\CO_{E_{v,i}}\otimes_{A_v}(Q_v/A_v),\hat G(L^\sep)\bigr)\\[2mm]
\es=\es & \Hom_{\CO_{E_{v,i}}}\bigl(E_{v,i}/\CO_{E_{v,i}},\hat G(L^\sep)\bigr) && \ni f\\[2mm]
\es=\es & \TS \bigl\{(P_{n})_{n}\in\prod\limits_{{n}\in\BN_{0}}\hat G[y_i^{n}](L^\sep)\colon [y_i](P_{n})=P_{n-1}\bigr\} && \ni (P_{n})_{n}:=\bigl(f(y_i^{-n})\bigr)_{n}\,,
\end{align*}
where $f$ is reconstructed from $(P_n)_n$ as $f(ay_i^{-n}):=[a](P_n)$ for $a\in\CO_{E_{v,i}}\mal$. From equation~\eqref{EqTPlus} we see that $\tplus{y_i,\phi(y_i)}=\sum_{n=0}^\infty \tplusminus_n y_i^n$ satisfies
\[
[y_i](\tplusminus_0)\;=\;\tplusminus_0^{\tilde q_i}+\phi(y_i)\tplusminus_0\;=\;0 \qquad\text{and}\qquad [y_i](\tplusminus_n)\;=\;\tplusminus_n^{\tilde q_i}+\phi(y_i)\tplusminus_n\;=\;\tplusminus_{n-1}\,.
\]
Thus $\tplusminus_{n-1}\in\hat G[y_i^n](L^\sep)$ and $T_v\hat G=\CO_{E_{v,i}}\cdot(\tplusminus_{n-1})_n$. To compute the $\sG_L$-action on $T_v\hat G$ we need the following
\end{Point}

\begin{Proposition}\label{PropCyclotomic}
Let $\BF_{\tilde v_i}\dpl\phi(y_i)\dpr_\infty:=\BF_{\tilde v_i}\dpl\phi(y_i)\dpr(\tplusminus_n:n\in\BN_0)$. Then there is an isomorphism of topological groups
\[
\chi\colon\Gal\bigl(\BF_{\tilde v_i}\dpl\phi(y_i)\dpr_\infty/\BF_{\tilde v_i}\dpl\phi(y_i)\dpr\bigr)\;\isoto\;\BF_{\tilde v_i}\dbl y_i\dbr\mal\;=\;\CO_{E_{v,i}}\mal
\]
satisfying $g(\tplus{y_i,\phi(y_i)}):=\sum_{n=0}^\infty g(\tplusminus_n)y_i^n=\chi(g)\cdot \tplus{y_i,\phi(y_i)}$ in $\BF_{\tilde v_i}\dpl\phi(y_i)\dpr_\infty\dbl y_i\dbr$ for $g$ in the Galois group. The isomorphism $\chi$ is independent of the choice of $\tplus{y_i,\phi(y_i)}$ and is called the \emph{cyclotomic character} of the field $E_{v,i}=\BF_{\tilde v_i}\dpl y_i\dpr$.
\end{Proposition}

\begin{proof}
The existence of $\chi$ follows from the equation $\hat\sigma^{f_i}(\tplus{y_i,\phi(y_i)})=(y_i-\phi(y_i))\cdot \tplus{y_i,\phi(y_i)}$, which implies that $\chi(g):=\frac{g(\tplus{y_i,\phi(y_i)})}{\tplus{y_i,\phi(y_i)}}$ is $\hat\sigma^{f_i}$-invariant, that is $\chi(g)\in\BF_{\tilde v_i}\dbl y_i\dbr\mal$. Furthermore, $\chi$ is an isomorphism because $\tplusminus_{n-1}$ is a uniformizing parameter of $\BF_{\tilde v_i}\dpl\phi(y_i)\dpr(\tplusminus_0,\ldots,\tplusminus_{n-1})$ and so the equations defining the $\tplusminus_n$ are irreducible by Eisenstein. Every other solution of \eqref{EqTPlus} is of the form $a\cdot\tplus{y_i,\phi(y_i)}$ with $a\in\BF_{\tilde v_i}\dbl y_i\dbr$ and so $g(a\cdot\tplus{y_i,\phi(y_i)})=a\cdot g(\tplus{y_i,\phi(y_i)})=\chi(g)\cdot a\cdot\tplus{y_i,\phi(y_i)}$. This shows that $\chi(g)$ does not depend on the solution $\tplus{y_i,\phi(y_i)}$.
\end{proof}

Let $\CI_L\subset\sG_L$ be the inertia subgroup and similarly for other fields. By local class field theory, see Lubin and Tate \cite[Corollary on p.~386]{LT}, the image of $g\in\CI_{\phi(E_{v,i})}$ in $\sG_{\phi(E_{v,i})}^{\rm ab}$ equals the norm residue symbol $\bigl(\chi(g)^{-1}|_{y_i=\phi(y_i)},\phi(E_{v,i})^{\rm ab}/\phi(E_{v,i})\bigr)$ where $\phi(E_{v,i})^{\rm ab}$ is the maximal abelian extension of $\phi(E_{v,i})$ in $Q_v^\sep$. In general, the homomorphism $\chi_{L}\colon\CI_{L}\to\CO_{L}\mal$ with $g|_{L^{\rm ab}}=(\chi_L(g)^{-1},L^{\rm ab}/L)$ is sometimes called the \emph{character of local class field theory} of the field $L$. So we see that $\chi(g)|_{y_i=\phi(y_i)}= \chi_{\phi(E_{v,i})}(g)$. If $L$ is separable over $\phi(E_{v,i})$ these characters are compatible for $g\in\CI_L$ in the sense that $\chi_{\phi(E_{v,i})}(g)=N_{L/\phi(E_{v,i})}(\chi_L(g))$. 

\begin{Point}
From $T_v\hat G=\CO_{E_{v,i}}\cdot(\tplusminus_{n-1})_n$ it follows that $g$ acts on $T_v\hat G$ in the same way as an endomorphism in $\CO_{E_{v,i}}\mal$. Let us compute this endomorphism. We write $\chi(g)=\sum_{k=0}^\infty a_k y_i^k$ with $a_k\in\BF_{\tilde v_i}$. Then the expansion $g(\tplus{y_i,\phi(y_i)})=\chi(g)\cdot\tplus{y_i,\phi(y_i)}=\sum_{n=0}^\infty\sum_{k=0}^na_k\tplusminus_{n-k}y_i^n$ implies that $g(\tplusminus_n)=\sum_{k=0}^na_k\tplusminus_{n-k}=\sum_{k=0}^n\phi(a_k^{q_v^{-j(\phi)}})[y_i^k](\tplusminus_n)$. Thus every element $g\in\CI_L$ acts on $T_v\hat G$ as the endomorphism $\sum_{k=0}^\infty a_k^{q_v^{-j(\phi)}}y_i^k=\hat\sigma^{-j(\phi)}(\chi(g))=\phi^{-1}\bigl(\chi(g)|_{y_i=\phi(y_i)}\bigr)=\phi^{-1}\circ\chi_{\phi(E_{v,i})}(g)\in\CO_{E_{v,i}}\mal$ and on $T_v\hat G_{E_v,\phi}$ as the endomorphism $\bigl(\hat\sigma^{-j(\phi)}(\chi(g))^{\delta_{i,i(\phi)}}\bigr)_i\in\CO_{E_v}\mal$.
\end{Point}

\begin{Definition}
We define the character $\chi_{E_v,\phi}:=\bigl((\phi^{-1}\circ\chi_{\phi(E_{v,i})})^{\delta_{i,i(\phi)}}\bigr)_i\colon\CI_L\to\CO_{E_v}\mal$ by mapping $g\mapsto\bigl(\hat\sigma^{-j(\phi)}(\chi(g))^{\delta_{i,i(\phi)}}\bigr)_i=\bigl(1,\ldots,1,\phi^{-1}\circ\chi_{\phi(E_{v,i})}(g),1,\ldots,1\bigr)$.
\end{Definition}

\begin{Point}\label{Point4.9}
Due to the equivariance of the pairing \eqref{EqPairing} under $\sG_L$ and $\End_R(\hat G_{E_v,\phi})$ the action of $g\in\sG_L$ on $\Koh^1_v(\ulHM(\hat G_{E_v,\phi}),A_v)$ is given by the endomorphism $\chi_{E_v,\phi}(g)^{-1}$. We can also compute this action directly as follows. It factors through the restriction of $g$ to $\Gal\bigl(\BF_{\tilde v_i}\dpl\phi(y_i)\dpr_\infty/\BF_{\tilde v_i}\dpl\phi(y_i)\dpr\bigr)$ which we denote again by $g$. Then on the basis $(\check u_{i,j})_j$ of $\Koh^1_v(\ulHM(\hat G),A_v)$ from \eqref{EqTateModMPhi} we compute
\begin{eqnarray*}
g(\check u_{i,j})_j & = & g\Bigl(\hat\sigma^{(j,j(\phi))}(\tplus{y_i,\phi(y_i)})^{-\delta_{i,i(\phi)}}\Bigr)_j\\[2mm]
& = & \Bigl(\hat\sigma^{(j,j(\phi))}\bigl(\chi(g)\cdot\tplus{y_i,\phi(y_i)}\bigr)^{-\delta_{i,i(\phi)}}\Bigr)_j \\[2mm]
 & = & \Bigl(\hat\sigma^{j-j(\phi)}\bigl(\chi(g)^{-1}\bigr)^{\delta_{i,i(\phi)}}\cdot\hat\sigma^{(j,j(\phi))}(\tplus{y_i,\phi(y_i)})^{-\delta_{i,i(\phi)}}\Bigr)_j\\[2mm]
 & = & (\hat\sigma^{-j(\phi)}(\chi(g)^{-1})\otimes 1)^{\delta_{i,i(\phi)}}\cdot(\check u_{i,j})_j
\end{eqnarray*}
for the element $\hat\sigma^{-j(\phi)}(\chi(g)^{-1})\otimes 1\in\CO_{E_{v,i}}\otimes_{A_v}R\dbl z\dbr$. That is, the action of $g\in\sG_L$ on $\Koh^1_v(\ulHM(\hat G),A_v)$ coincides with the endomorphism $\hat\sigma^{-j(\phi)}(\chi(g)^{-1})$ and the action of $g\in\CI_L$ on $\Koh^1_v(\ulHM(\hat G_{E_v,\phi}),A_v)=\Koh^1_v(\ulHM_{E_v,\phi},A_v)$ coincides with the endomorphism $\chi_{E_v,\phi}(g)^{-1}\in\CO_{E_v}\mal$. 
\end{Point}

\begin{Proposition}
Let $\ulHM$ have complex multiplication by a commutative, semi-simple $Q_v$-algebra $E_v$ with local CM-type $\Phi=(d_\phi)_{\phi\in H_{E_v}}$. Then the action of $g\in\CI_L$ on $\Koh^1_v(\ulHM,A_v)$ coincides with the endomorphism $\prod_{\phi\in H_{E_v}}\chi_{E_v,\phi}(g)^{-d_\phi}\in\CO_{E_v}\mal$. 
\end{Proposition}

\begin{proof}
This follows from the computations in \ref{Point4.4}, \ref{Point4.9} and \ref{Point4.2} by observing that $\CI_L$ acts trivially on $\Koh^1_v(\ulHM_{E_v,0},A_v)$, because its generator $(c_{i,j})_{i,j}$ is defined over the maximal unramified extension of $L$.
\end{proof}

\begin{Point}\label{Point4.12}
To compute the absolute value $\bigl|\int_u\omega\bigr|_v$ we again treat each factor $\ulHM_{E_v,0}$ and $\ulHM_{E_v,\phi}$ of $\ulHM$ separately. We begin with $\ulHM_{E_v,\phi}$ and set $i:=i(\phi)$. Let $\omega^\circ_\psi:=1\in\Koh^\psi(\ulHM_{E_v,\phi},L\dbl y_{i}-\psi(y_{i})\dbr)$. It is a generator as $L\dbl y_{i}-\psi(y_{i})\dbr$-module as in Definition~\ref{DefValuationOfomega}, and is mapped under the period isomorphism of $\ulHM_{E_v,\phi}$ from \eqref{EqPeriodMPhi} to
\begin{equation}\label{EqOmega}
 h_{v,\dR}^{-1}(\omega^\circ_\psi)= (0,\ldots,\Bigl(\bigl(y_{i(\psi)}-\phi(y_{i(\psi)})\bigr)^{\delta_{j(\psi),j(\phi)}}\cdot \hat\sigma^{(j(\psi),j(\phi))}(\tplus{y_{i(\psi)},\phi(y_{i(\psi)})})\Bigr)^{\delta_{i(\psi),i(\phi)}},\ldots,0 )\cdot \check u\,,
\end{equation}
where the non-zero entry is in component $\psi$. We denote this entry by $\Omega(E_v,\phi,\psi)$. It is analogous to Colmez's \cite[Th\'eor\`eme~I.2.1]{Colmez93} element of $\bB_\dR$ with the same name. It satisfies the following
\end{Point}

\begin{Theorem}\label{ThmOmega}
Let $\phi,\psi\in H_{E_v}$ satisfy $i(\phi)=i(\psi)=:i$ and assume that $E_{v,i}$ is separable over $Q_v$. Then the element
\[
\Omega(E_v,\phi,\psi)\;:=\;\bigl(y_{i}-\phi(y_{i})\bigr)^{\delta_{j(\psi),j(\phi)}}\cdot \hat\sigma^{(j(\psi),j(\phi))}(\tplus{y_{i},\phi(y_{i})})\;\in\;\BC_v\dpl y_i-\psi(y_i)\dpr\;=\;\BC_v\dpl z-\zeta\dpr
\]
satisfies
\begin{enumerate}
\item \label{ThmOmega_A}
$\hat v \bigl(\Omega(E_v, \phi, \psi)\bigr) = 1$   if   $\phi = \psi$ and $\hat v \bigl(\Omega(E_v, \phi, \psi)\bigr) = 0$ if   $\phi \neq \psi$.
\item \label{ThmOmega_B}
\[ v\bigl(\Omega(E_v, \phi, \psi)\bigr) \;=\; 
\begin{cases}
                  \tfrac{1}{e_i(\tilde q_i-1)}   - v(\FD_{\psi(E_{v,i})/Q_v}) &    \text{ if}\quad  \phi = \psi, \\[2mm]
                  \tfrac{1}{e_i(\tilde q_i-1)}  + v\big(\psi(y_i)-\phi(y_i)\big) &    \text{ if} \quad \phi \neq \psi\text{ and }j(\phi)=j(\psi),\\[3mm]
                 \tfrac{ q_v^{(j(\psi),j(\phi))}}{e_i(\tilde q_i-1)}  &    \text{ if}  \quad j(\phi)\ne j(\psi)\,,
\end{cases}
\]
where $\FD_{\psi(E_{v,i})/Q_v}$ is the different of $\psi(E_{v,i})$ over $Q_v$.
\item \label{ThmOmega_C}
If $g\in \CI_L$, then $g\bigl(\Omega(E_v, \phi, \psi)\bigr) = \psi\bigl(\chi_{E_v,\phi}(g)\bigr)\cdot \Omega(E_v, \phi, \psi)$. Note that if $L$ is separable over $Q_v$ then $\psi(\chi_{E_v,\phi}(g)\bigr)=\psi\bigl(\phi^{-1}(N_{L/\phi(E_{v,i})}\chi_L(g))\bigr)$.
\item \label{ThmOmega_D}
Let $u\in\Koh_{1,v}(\ulHM_{E_v,\phi},A_v):=\Hom_{A_v}\bigl(\Koh^1_v(\ulHM_{E_v,\phi},A_v),A_v\bigr)$ be a generator as $\CO_{E_v}$-module and let $\omega^\circ_\psi$ be an $L\dbl y_{i}-\psi(y_{i})\dbr$-generator of $\Koh^\psi(\ulHM_{E_v,\phi},L\dbl y_{i}-\psi(y_{i})\dbr)$ subject to the conditions in Definition~\ref{DefValuationOfomega}, that is subject to $\omega^\circ_\psi\mod y_i-\phi(y_i)\in\Koh^1_v(\ulHM_{E_v,\phi},L)$ being an $R$-generator of the free $R$-module of rank one $\Koh^\psi(\ulHM,R)\big/(y_i-\psi(y_i))\Koh^\psi(\ulHM,R)$. Moreover, let $D_\psi$ be a generator as $\psi(\CO_{E_{v,i}})$-module of the different $\FD_{\psi(E_{v,i})/Q_v}$. Then 
\[
\int_u\omega^\circ_\psi\;:=\; u\otimes\id_{\BC_v\dpl z-\zeta\dpr}\bigl(h^{-1}_{v,\dR}(\omega^\circ_\psi)\bigr)\;\in\;\BC_v\dpl z-\zeta\dpr
\]
equals $\Omega(E_v, \phi, \psi)\cdot D_\psi^{-1}$ up to multiplication by an element of $R^\times +(z-\zeta)\cdot L\dbl z-\zeta\dbr$.
\end{enumerate}
\end{Theorem}

\begin{Remark}
Note that in contrast to the number field case \cite[Th\'eor\`eme~I.2.1]{Colmez93} the element $\Omega(E_v,\phi,\psi)\in\BC_v\dpl z-\zeta\dpr$ is by \ref{ThmOmega_A}, \ref{ThmOmega_B} and \ref{ThmOmega_C} uniquely determined only up to multiplication by an element of $\CO_{\wt L}^\times +(z-\zeta)\cdot \wt L\dbl z-\zeta\dbr$, where $\wt L$ is the completion of the compositum of $\BF_q^\alg$ with the perfect closure of $L$ in $Q_v^\alg$, because the fixed field of $\CI_L$ in $\BC_v\dpl z-\zeta\dpr$ equals $\wt L\dpl z-\zeta\dpr$ by the Ax-Sen-Tate theorem \cite{Ax70}.
\end{Remark}

\begin{proof}[Proof of Theorem~\ref{ThmOmega}]
In \ref{Point4.3} we have seen that the coefficients of the series $\tplus{y_{i},\phi(y_{i})}=\sum_{n=0}^\infty\tplusminus_n y_i^n$ satisfy $v(\tplusminus_n)=v(\phi(y_{i}))\cdot{\tilde q_i^{-n}/(\tilde q_i-1)}$. From $v(\psi(y_i))=1/e_i=v(\phi(y_i))$ it follows that the evaluation of $\hat\sigma^{(j(\psi),j(\phi))}(\tplus{y_{i},\phi(y_{i})})\big|_{y_i=\psi(y_i)}=\sum_{n=0}^\infty\tplusminus_n^{q_v^{(j(\psi),j(\phi))}}\psi(y_i)^n$ at $y_i=\psi(y_i)$ satisfies 
\begin{equation}\label{EqValuationTPlus}
v(\tplusminus_n^{q_v^{(j(\psi),j(\phi))}}\psi(y_i)^n) \; = \; \tfrac{1}{e_i}\cdot\Bigl(n+\tfrac{q_v^{(j(\psi),j(\phi))}}{\tilde q_i^{n}(\tilde q_i-1)}\Bigr).
\end{equation}
Since $0\le(j(\psi),j(\phi))\le f_i-1$ the second fraction in the parenthesis is strictly smaller than $1$, and so the valuations in \eqref{EqValuationTPlus} are strictly increasing with $n$ and attain their minimum $\tfrac{q_v^{(j(\psi),j(\phi))}}{e_i(\tilde q_i-1)}$ for $n=0$. This shows that $\hat\sigma^{(j(\psi),j(\phi))}(\tplus{y_{i},\phi(y_{i})})\big|_{y_i=\psi(y_i)}$ is non-zero in $L$ and 
\[
v\Bigl(\hat\sigma^{(j(\psi),j(\phi))}(\tplus{y_{i},\phi(y_{i})})\big|_{y_i=\psi(y_i)}\Bigr)\;=\; \tfrac{q_v^{(j(\psi),j(\phi))}}{e_i(\tilde q_i-1)}\,.
\]
In particular the valuation $\hat v\bigl(\hat\sigma^{(j(\psi),j(\phi))}(\tplus{y_{i},\phi(y_{i})})\bigr)=0$. 

\medskip\noindent
\ref{ThmOmega_A} Lemma~\ref{LemDerivative} implies that in the $\psi$-component of $E_v\otimes_{Q_v}L\dbl z-\zeta\dbr$ we have $\ord_{y_i-\psi(y_i)}=\ord_{z-\zeta}$. If $j(\phi)=j(\psi)$, that is $\phi|_{\BF_{\tilde v_i}}=\psi|_{\BF_{\tilde v_i}}$ then $\phi\ne\psi$ implies $\psi(y_i)-\phi(y_i)\ne0$ in $L$, because $E_{v,i}=\BF_{\tilde v_i}\dpl y_i\dpr$. Therefore the valuation $\hat v$ of $y_i-\phi(y_i)=(\psi(y_i)-\phi(y_i))+(y_i-\psi(y_i))$ equals zero for $\phi\ne\psi$ and $j(\phi)=j(\psi)$. This implies \ref{ThmOmega_A}.

\medskip\noindent
\ref{ThmOmega_B} We will calculate  $v\bigl((\Omega(E_v, \phi, \psi)\bigr)$ in three different cases separately as follows.

\medskip\noindent
Case1: $\psi = \phi$. In this case  $\hat v\bigl(\Omega(E_v, \phi, \psi)\bigr) = 1$ and so 
\begin{eqnarray*}
v\bigl(\Omega(E_v, \phi, \psi)\bigr) & = & v\bigl(\Bigl(\tfrac{y_i-\phi(y_i)}{z-\zeta}\cdot \tplus{y_i,\phi(y_i)}\Bigr)\Big|_{y_i =\phi(y_i)}\bigr) \\[2mm]
& = & v\bigl(\tfrac{y_i-\phi(y_i)}{z-\zeta}\big|_{y_i =\phi(y_i)}\bigr) + v\bigl(\tplus{y_{i},\phi(y_{i})}|_{y_i=\psi(y_i)}\bigr)\\[2mm]
& = & - v(\FD_{\psi(E_{v,i})/Q_v}) + \tfrac{1}{e_i(\tilde q_i-1)}
\end{eqnarray*}
by Corollary~\ref{CorDifferent}.

\medskip\noindent
Case 2: $\psi \neq \phi$ and $j(\psi) = j(\phi)$. In this case  $\hat v\bigl(\Omega(E_v, \phi, \psi)\bigr) = 0$ and so 
\begin{eqnarray*}
v\bigl(\Omega(E_v, \phi, \psi)\bigr) & = & v\Bigl(\bigl((y_i-\phi(y_i))\cdot \tplus{y_i,\phi(y_i)}\bigr)\big|_{y_i =\psi(y_i)}\Bigr) \\[2mm]
& = & v\bigl(\psi(y_i)-\phi(y_i)\bigr) + v\bigl(\tplus{y_{i},\phi(y_{i})}|_{y_i=\psi(y_i)}\bigr)\\[2mm]
& = & v\bigl(\psi(y_i)-\phi(y_i)\bigr) + \tfrac{1}{e_i(\tilde q_i-1)}\,.
\end{eqnarray*}

\medskip\noindent
Case 3: $j(\psi)\neq j(\phi)$. In this case  $\hat v\bigl(\Omega(E_v, \phi, \psi)\bigr) = 0$ and so 
\[
v\bigl(\Omega(E_v, \phi, \psi)\bigr) \;=\; v\Bigl(\hat\sigma^{(j(\psi),j(\phi))}(\tplus{y_{i},\phi(y_{i})})|_{y_i=\psi(y_i)}\Bigr)\;=\; \tfrac{q_v^{(j(\psi),j(\phi))}}{e_i(\tilde q_i-1)}\,.
\]

\medskip\noindent
\ref{ThmOmega_C} For the $\CO_{E_v}$-basis $\check u$ of $\Koh^1_v(\ulHM_{E_v,\phi},A_v)$ from \eqref{EqTateModMPhi} we have seen in \ref{Point4.9} that $g(\check u)=\chi_{E_v,\phi}(g)^{-1}\cdot\check u$ and $g\bigl(0,\ldots,\Omega(E_v,\phi,\psi),\ldots,0\bigr)\cdot g(\check u)=h_{v,\dR}^{-1}(g(\omega^\circ_\psi))$ $=h_{v,\dR}^{-1}(\omega^\circ_\psi)=\bigl(0,\ldots,\Omega(E_v,\phi,\psi),\ldots,0\bigr)\cdot \check u$. Thus $g$ acts on the coefficient $\bigl(0,\ldots,\Omega(E_v,\phi,\psi),\ldots,0\bigr)$ as multiplication with $\chi_{E_v,\phi}(g)$ and on its $\psi$-component $\Omega(E_v,\phi,\psi)$ by multiplication with $\psi(\chi_{E_v,\phi}(g))$.

\medskip\noindent
\ref{ThmOmega_D} Again we consider the $\CO_{E_v}$-basis $\check u$ of $\Koh^1_v(\ulHM_{E_v,\phi},A_v)$ from \eqref{EqTateModMPhi}. Let $D=(D_i)_i\in\CO_{E_v}=\prod_i\CO_{E_{v,i}}$ be a generator of the different $\FD_{E_v/Q_v}=\prod_i\FD_{E_{v,i}/Q_v}=D\cdot\CO_{E_v}$ and let $c=(c_i)_i\in\CO_{E_v}\mal=\prod_i\CO_{E_{v,i}}\mal$ be the element(s) from Lemma~\ref{LemmaTrace} below for which the pairing $\langle\,.\,,\,.\,\rangle\colon\Koh_{1,v}(\ulHM_{E_v,\phi},A_v)\times\Koh_v^1(\ulHM_{E_v,\phi},A_v)\to A_v$ takes the value $\langle a\,u\,,\,b\,\check u\rangle=\Tr_{E_v/Q_v}(abcD^{-1})$ for $a,b\in\CO_{E_v}$. If $\omega^\circ_\psi=1$ is the generator from \ref{Point4.12} then
\[
{\TS\int_u}\omega^\circ_\psi\;=\;\Tr_{E_v/Q_v}\bigl(0,\ldots,\Omega(E_v,\phi,\psi)\cdot\psi(cD^{-1}),\ldots,0\bigr)\;=\;\Omega(E_v,\phi,\psi)\cdot\psi(c_iD_i^{-1})\,.
\]
Any other generator $\omega^\circ_\psi$ differs from $\omega^\circ_\psi=1$ by multiplication by an element 
\[
x\;\in\; R^\times +(y_i-\phi(y_i))\cdot L\dbl y_i-\phi(y_i)\dbr\subset E_v\otimes_{Q_v}L\dbl y_i-\phi(y_i)\dbr\,. 
\]
Under the pairing $\langle\,.\,,\,.\,\rangle$ this leads to $\int_u\omega^\circ_\psi=\Omega(E_v,\phi,\psi)\cdot \psi(c_ixD_i^{-1})$ with $\psi(D_i)=D_\psi$ and
\[
\psi(c_ix)\;\in\; R^\times +(y_i-\phi(y_i))\cdot L\dbl y_i-\phi(y_i)\dbr\;=\; R^\times +(z-\zeta)\cdot L\dbl z-\zeta\dbr\,.
\]
\end{proof}

It remains to record the following well known

\begin{Lemma}\label{LemmaTrace}
If $E_{v,i}/Q_v$ is separable let $D_i\in\CO_{E_{v,i}}$ be a generator of the different $\FD_{E_{v,i}/Q_v}=D_i\cdot\CO_{E_{v,i}}$. Then for any perfect pairing $\langle\,.\,,\,.\,\rangle\colon \CO_{E_{v,i}}\times \CO_{E_{v,i}} \to A_v$ satisfying $\langle a,b\rangle = \langle ab,1\rangle = \langle 1, ab\rangle$, there is an element $c_i \in \CO_{E_{v,i}}\mal$ with $\langle a,b\rangle = \Tr_{E_{v,i}/Q_v}(abc_iD_i^{-1})$.
\end{Lemma}

\begin{proof}
The set of bilinear forms $\CO_{E_{v,i}} \times \CO_{E_{v,i}} \to A_v$ equals $\Hom_{A_v}(\CO_{E_{v,i}} \otimes_{A_v} \CO_{E_{v,i}}, A_v)$ and the condition $\langle a,b\rangle = \langle ab,1\rangle = \langle 1, ab\rangle$ implies that $\langle\,.\,,\,.\,\rangle$ lies in $\Hom_{A_v}(\CO_{E_{v,i}} \otimes_{\CO_{E_{v,i}}} \CO_{E_{v,i}}, A_v) = \Hom_{A_v}(\CO_{E_{v,i}}, A_v)$. The condition that $\langle\,.\,,\,.\,\rangle$ is perfect implies that 
\begin{equation}\label{EqLemmaTrace1}
\CO_{E_{v,i}}\;\isoto\;\Hom_{A_v}(\CO_{E_{v,i}}, A_v),\quad a\;\longmapsto\;[\,b\mapsto \langle a, b\rangle\,]
\end{equation}
is an isomorphism of $\CO_{E_{v,i}}$-modules. On the other hand, by definition of the different in \cite[\S\,III.3]{SerreLF}, there are also isomorphisms of $\CO_{E_{v,i}}$-modules 
\[
\FD_{E_{v,i}/Q_v}^{-1}\;=\;D_i^{-1}\cdot\CO_{E_{v,i}}\;\isoto\;\Hom_{A_v}(\CO_{E_{v,i}}, A_v),\quad \tilde aD_i^{-1}\;\longmapsto\;[\,b\mapsto \Tr_{E_{v,i}/Q_v}(\tilde abD_i^{-1})\,]
\]
for $\tilde a\in\CO_{E_{v,i}}$, and 
\begin{equation}\label{EqLemmaTrace2}
\CO_{E_{v,i}}\;\isoto\;\Hom_{A_v}(\CO_{E_{v,i}}, A_v),\quad \tilde a\;\longmapsto\;[\,b\mapsto \Tr_{E_{v,i}/Q_v}(\tilde abD_i^{-1})\,]\,.
\end{equation}
Comparing \eqref{EqLemmaTrace1} and \eqref{EqLemmaTrace2} yields a unit $c_i\in\CO_{E_{v,i}}\mal$ with $\tilde a=c_ia$ and $\langle a,b\rangle = \Tr_{E_{v,i}/Q_v}(abc_iD_i^{-1})$.
\end{proof}

\begin{Point}
Analogously to \cite[\S\,I.2]{Colmez93}, we can give a uniform formula for $v\bigl(\Omega(E_v, \phi, \psi)\bigr)$ by introducing certain measures on $\sG_{Q_v}$. Let $\CC(\sG_{Q_v}, \BQ)$ be the $\BQ$-vector space of locally constant functions $a\colon\sG_{Q_v}\to\BQ$. If $K$ is a finite separable extension of $Q_v$, and $\phi, \psi \in H_K$, let $a_{K, \phi, \psi} \in \CC(\sG_{Q_v}, \BQ) $  be the function given by 
\[ a_{K, \phi, \psi}(g) \;:=\;
    \begin{cases}
            1 &    \text{if}\quad  g\phi = \psi,\\
           0 &         \text{otherwise}.
    \end{cases} 
\]
Note that the $a_{K, \phi, \psi}$ span $\CC(\sG_{Q_v}, \BQ)$.

If $L\subset Q_v^\sep$ is a finite Galois extension of $Q_v$, let $\mu_L \in \CC(\sG_{Q_v}, \BQ)$ be the function given by the formula 
\[ \mu_L(g) \;:=\;
    \begin{cases}
            0 &    \text{if}\quad  g\notin \CI_{Q_v},\\
         - v(g(\pi_L)-\pi_L) &  \text{if}\quad  g\in \CI_{Q_v} \ \text{and}\ g(\pi_L)\neq \pi_L      ,\\
          v(\FD_{L/Q_v})&  \text{if}\quad  g\in \CI_{Q_v} \ \text{and}\ g(\pi_L)=  \pi_L,
    \end{cases} 
\]
where $\pi_L$ is a uniformizer of $L$ and $\FD_{L/Q_v}$ is the different of $L$ over $Q_v$.

Moreover, we let $e_K$ be the ramification index of $K$ over $Q_v$ and $f_K$ the degree of the residue field of $K$ over $\BF_v$. We let $W_L^n:=\bigl\{g\in\sG_{Q_v}\colon g(x)\equiv x^{q_v^n}\mod\Fm_{Q_v^\alg}\es\forall\;x\in\CO_{Q_v^\alg}\bigr\}/\CI_L$. It is in bijection with $\{g\in\Gal(L/Q_v)\colon g(x)\equiv x^{q_v^n}\mod(\pi_L)\}$ under the map $\sG_{Q_v}\onto\Gal(L/Q_v)$.
\end{Point}

\begin{Lemma}\label{LemmaIndL} Let $K,L\subset Q_v^\sep$ be finite separable extensions of $Q_v$ with $L$ finite Galois over $Q_v$ containing all the conjugates of $K$, and let $\phi, \psi \in H_K$. The function $a_{K, \phi, \psi}$ is constant modulo  $\sG_L$ and hence can be considered  as a function on $G_L:=\Gal(L/Q_v)$. Then 
\[ \sum_{g\in G_L}a_{K, \phi, \psi}(g)\cdot\mu_L(g) \;=\;
    \begin{cases}
            0 &    \text{if}\quad  j(\phi) \neq j(\psi),\\
          v(\FD_{\psi (K)/Q_v}) &  \text{if}\quad  \phi = \psi      ,\\
          -v\bigl(\psi(\pi_K)- \phi(\pi_K)\bigr)&  \text{if}\quad   j(\phi) = j(\psi) \ \text{and}\ \phi \neq \psi,
    \end{cases} 
\]
and 
\[
\frac{1}{e_L}\sum_{n=1}^{\infty}\sum_{g \in W^n_L}\frac{a_{K, \phi, \psi}(g)}{ q_v^{ns}} \;=\; \frac{1}{e_K}\;\frac{ q_v^{(j(\psi),j(\phi))s}}{ q_v^{f_K s}-1}.
\]
In particular, the left hand side of both equations does not depend on the choice of $L$. 
\end{Lemma}

\begin{proof} The proof follows in the same way as  \cite[Lemma I.2.4]{Colmez93}.
\end{proof}

Since the $a_{K, \phi, \psi}$ generate the vector space  $\CC(\sG_{Q_v}, \BQ)$, we get the following proposition.

\begin{Proposition}\label{PropArtinMeasure}
There exist $\BQ$-linear homomorphisms $Z_v(\,.\,,s) : \CC(\sG_{Q_v}, \BQ) \to\BC$ if $s \in\BC$ and $\mu_{\Art,v}: \CC(\sG_{Q_v}, \BQ) \to \BQ$ defined by the following formulas: if $a \in \CC(\sG_{Q_v}, \BQ)$ and if $L \subset Q_v^\sep$ is a finite Galois extension of $Q_v$ such that $a$ is constant modulo $\sG_L$, then
\[
\mu_{\Art,v}(a) \;=\; \sum_{g\in G_L}a(g)\cdot\mu_L(g)
\]
with $G_L:=\Gal(L/Q_v)$, and $Z_v(a,s)$ is obtained by meromorphic extension from the following formula,  valid for $\CR e(s)>0$:
\[
\hspace{5.3cm}Z_v(a,s) \;=\; \frac{1}{e_L}\sum_{n=1}^{\infty}\sum_{g \in W^n_L}\frac{a(g)}{ q_v^{ns}}\hspace{5.3cm}\qed
\]
\end{Proposition}

\begin{Remark} If $V$ is a finite dimensional $\BC$-vector space, $\rho\colon \sG_{Q_v} \to \Aut_\BC(V)$ is a continuous complex representation of  $\sG_{Q_v}$, and if $\chi \in \CC^0(\sG_{Q_v}, \BQ) \otimes_\BQ\BC$ is the character of $\rho$, then $\mu_{\Art,v}(\chi)$ is nothing else than the degree at $v$ of the conductor $\Ff_\chi$ of $\chi$; cf.\ \cite[Chapter~VI, \S\,2]{SerreLF}, where $\mu_{\Art,v}(\chi)$ is denoted by $f(\chi)$. And if $W$ is the sub vector space of $V$ stable by $I_{Q_v}$, we have
\[
Z_v(\chi,a) \log q_v \;=\; -\frac{d}{ds} \log \Bigl( \det(1 -  q_v^{-s}\rho(\Frob_{L/Q_v})|_W)^{-1}\Bigr)
\]
by \cite[Chapter~0, \S\,4]{Tate84} or \cite[Lemma~9.14]{Rosen02}. So the linear maps $\mu_{\Art,v}$ and $Z_v(\,.\,,s)$ coincide with the maps with the same names in Definition~\ref{DefArtinMeasure} in the introduction.
\end{Remark}

As a direct consequence of Theorem~\ref{ThmOmega}, Proposition~\ref{PropArtinMeasure} and Lemma~\ref{LemmaIndL} we get the following 

\begin{Theorem}\label{Lubintateperiod}
If $\phi,\psi\in H_{E_v}$ satisfy $i(\phi)=i(\psi)=:i$ and $E_{v,i}$ is separable over $Q_v$ then 
\[
v\bigl(\Omega(E_v, \phi, \psi)\bigr) \;=\; Z_v(a_{E_v, \psi, \phi},1)- \mu_{\Art,v}(a_{E_v, \psi, \phi})\,,
\]
where we set $a_{E_v, \psi, \phi}:=a_{E_{v,i}, \psi, \phi}\colon g\mapsto\delta_{g\psi,\phi}$.
\qed
\end{Theorem}

\begin{Definition}\label{DefVPsiLocSht}
If $u\in\Koh_{1,v}(\ulHM,Q_v):=\Hom_{A_v}\bigl(\Koh^1_v(\ulHM,A_v),Q_v\bigr)$ is an $E_v$-generator there is an $a\in E_v\mal$, unique up to multiplication with an element of $\CO_{E_v}\mal$, such that $a^{-1}u$ is an $\CO_{E_v}$-generator of $\Koh_{1,v}(\ulHM,A_v)$. Then we define the valuation $v_\psi(u):=v\bigl(\psi(a)\bigr)\in\BZ$. {\color{red} (NOTE THAT \ $v_\psi(u)\;\in\;\BQ$ \ IN GENERAL; SEE ERRATUM~\ref{Erratum1})}

Note that if $\ulM = (M, \tau_M)$ is a uniformizable $A$-motive over $L$ with good model $\ul\CM$ and $\ulHM = \ulHM_v(\ul\CM)$ is the local shtuka at $v$ associated with $\ul\CM$ as in Example \ref{AMotLocSht}, then for an $E$-generator $u\in\Koh_{1,\Betti}(\ulM,Q)$ the present definition of $v_\psi\bigl(h_{\Betti,v}(u)\bigr)$ coincides with the definition of $v_\psi(u)$ from \eqref{EqVPsi}.
\end{Definition}

\begin{Corollary}\label{CorLubintateperiod}
Let $\phi,\psi\in H_{E_v}$ with $i(\phi)=i(\psi)=:i$ and assume that $E_{v,i}$ is separable over $Q_v$. Let $u\in\Koh_{1,v}(\ulHM_{E_v,\phi},Q_v)$ be a generator as $E_v$-module and let $\omega_\psi$ be an $L\dbl y_{i}-\psi(y_{i})\dbr$-generator of $\Koh^\psi(\ulHM_{E_v,\phi},L\dbl y_{i}-\psi(y_{i})\dbr)$. Then $\int_u\omega_\psi:=u\otimes\id_{\BC_v\dpl z-\zeta\dpr}\bigl(h^{-1}_{v,\dR}(\omega_\psi)\bigr)$ has valuation
\[
v\bigl({\TS\int_u}\omega_\psi\bigr) \;=\; Z_v(a_{E_v, \psi, \phi},1)- \mu_{\Art,v}(a_{E_v, \psi, \phi})-v(\FD_{\psi(E_v)/Q_v})+v(\omega_\psi)+v_\psi(u)\,,
\]
where $v(\omega_\psi)$ and $v_\psi(u)$ were defined in Definitions~\ref{DefValuationOfomega} and \ref{DefVPsiLocSht}, and $\FD_{\psi(E_v)/Q_v}$ is the different.
\end{Corollary}

\begin{proof}
Let $a\in E_v$ be such that $u\open:=a^{-1}u$ is an $\CO_{E_v}$-generator of $\Koh_{1,v}(\ulHM_{E_v,\phi},A_v)$ and let $\omega^\circ_\psi$ be an $L\dbl y_{i}-\psi(y_{i})\dbr$-generator of $\Koh^\psi(\ulHM_{E_v,\phi},L\dbl y_{i}-\psi(y_{i})\dbr)$ such that $\omega^\circ_\psi\mod y_i-\phi(y_i)\in\Koh^1_\dR(\ulHM_{E_v,\phi},L)$ is an $R$-generator of $\Koh^1_\dR(\ulHM_{E_v,\phi},R)$. Let $x\in L\dbl y_{i}-\psi(y_{i})\dbr\mal$ such that $\omega_\psi=x\cdot\omega^\circ_\psi$. Moreover, let $D_\psi$ be a generator as $\psi(\CO_{E_{v}})$-module of the different $\FD_{\psi(E_{v,i})/Q_v}$. Then 
\[
{\TS\int_u}\omega_\psi \;=\; (a\otimes1)x\cdot{\TS\int_{u\open}}\omega^\circ_\psi \;=\; (a\otimes1)x\cdot\Omega(E_v, \phi, \psi)\cdot D_\psi^{-1}\;\in\;\BC_v\dpl z-\zeta\dpr
\]
up to multiplication by an element of $R^\times +(z-\zeta)\cdot L\dbl z-\zeta\dbr$ by Theorem~\ref{ThmOmega}\ref{ThmOmega_D}. The element $(a\otimes1)x\in E_v\otimes_{Q_v}L\dbl z-\zeta\dbr$ lies in the $\psi$-component $L\dbl y_i-\psi(y_i)\dbr$ of the product decomposition \eqref{EqDecomp1}, and in that component $a\otimes 1$ is congruent to $\psi(a)$ modulo $y_i-\psi(y_i)$. Therefore
\[
v\bigl({\TS\int_u}\omega_\psi\bigr) \;=\; v\bigl(\psi(a)x\cdot\Omega(E_v, \phi, \psi)\cdot D_\psi^{-1}\bigr) \;=\; Z_v(a_{E_v, \psi, \phi},1)- \mu_{\Art,v}(a_{E_v, \psi, \phi})-v(\FD_{\psi(E_v)/Q_v})+v(\omega_\psi)+v_\psi(u)\,.
\]
by Theorem~\ref{Lubintateperiod}. 
\end{proof}

To compute $v(\int_u\omega_\psi)$ for general $\ulHM$ we need the following

\begin{Definition}\label{DefapsiPhi}
Let $E_v$ be separable over $Q_v$ and let $\Phi=(d_\phi)_{\phi\in H_{E_v}}$ be a local CM-type. For $\psi \in H_{E_v}$ let $a_{E_v,\psi, \Phi}\in\CC(\sG_{Q_v}, \BQ)$ and $a^0_{E_v,\psi, \Phi}\in\CC^0(\sG_{Q_v}, \BQ)$ be given by the formulas
\begin{eqnarray}\label{EqapsiPhi}
a_{E_v,\psi, \Phi}(g) & := &  {\TS\sum\limits_{\phi \in H_{E_v}}}d_\phi\cdot a_{E_v,\psi, \phi}(g) \es=\es d_{g\psi}\qquad\text{and}\\[2mm]
a^0_{E_v,\psi, \Phi}(g) & := &  \tfrac{1}{\#H_L}\TS\sum\limits_{\eta\in H_L}d_{\eta^{-1}g\eta\psi}\,.
\end{eqnarray}
Note that $a_{E_v,\psi, \Phi}$ and $a^0_{E_v,\psi, \Phi}$ factor through $\Gal(E_v^\nor/Q_v)$ where $E_v^\nor$ is the Galois closure of $\psi(E_v)$ in $Q_v^\alg$. In particular, $a^0_{E_v,\psi, \Phi}$ does not depend on the field $L$ provided $\psi(E_v)\subset L$ for all $\psi\in H_{E_v}$.

These functions are the local counterparts to the functions $a_{E,\psi, \Phi}\in\CC(\sG_Q, \BQ)$ and $a^0_{E,\psi, \Phi}\in\CC^0(\sG_Q, \BQ)$ which were defined in \eqref{Eq:a_E} and \eqref{Eq:a0_E}. The membership in $\CC(\sG_{Q_v}, \BQ)$, respectively $\CC(\sG_Q, \BQ)$ is indicated by the index which gives reference to the $Q_v$-algebra $E_v$, respectively the $Q$-algebra $E$. In fact, if $E_v=E\otimes_QQ_v$ and hence $H_{E_v}=H_E$, then $a_{E_v,\psi, \Phi}$ is equal to the image of $a_{E,\psi,\Phi}$ under the map $\CC(\sG_Q, \BQ)\to\CC(\sG_{Q_v}, \BQ)$ from Definition~\ref{DefArtinMeasure}. However, this is in general not true for $a^0_{E_v,\psi, \Phi}$ and $a^0_{E,\psi,\Phi}$, because if $L$ is the closure of $K$ in $Q_v^\alg$, then $H_L$ is in general strictly contained in $H_K$.
\end{Definition}

For general $\ulHM$ we can now prove the following

\begin{Theorem}\label{ThmDMPeriod}
Let $\ulHM$ be a local shtuka over $R$ with complex multiplication by the ring of integers $\CO_{E_v}$ in a commutative, semi-simple, separable $Q_v$-algebra $E_v$ with local CM-type $\Phi$, and assume that $\psi(E_v)\subset L$ for all $\psi\in H_{E_v}$ and that $L$ is separable over $Q_v$. Let $u\in\Koh_{1,v}(\ulHM,Q_v)$ be an $E_v$-generator and let $\omega_\psi$ be an $L\dbl y_{i(\psi)}-\psi(y_{i(\psi)})\dbr$-generator of $\Koh^\psi(\ulHM,L\dbl y_{i(\psi)}-\psi(y_{i(\psi)})\dbr)$. Then the period $\int_u\omega_\psi:=u\otimes\id_{\BC_v\dpl z-\zeta\dpr}\bigl(h^{-1}_{v,\dR}(\omega_\psi)\bigr)$ has valuation
\[
v\bigl({\TS\int_u}\omega_\psi\bigr) \;=\; Z_v(a_{E_v, \psi, \Phi},1)- \mu_{\Art,v}(a_{E_v, \psi, \Phi})-v(\FD_{\psi(E_v)/Q_v})+v(\omega_\psi)+v_\psi(u)\,,
\]
where $v(\omega_\psi)$ and $v_\psi(u)$ were defined in Definitions~\ref{DefValuationOfomega} and \ref{DefVPsiLocSht}, and $\FD_{\psi(E_v)/Q_v}$ is the different.
\end{Theorem}

\begin{proof}
As in \ref{Point4.1} the local shtuka $\ulHM$ is isomorphic to the tensor product $\ulHM_{E_v,0}\otimes\bigotimes_\phi\ulHM_{E_v,\phi}{}^{\otimes d_\phi}$ over $\CO_{E_v,R}$. Let $i:=i(\psi)$ and $j:=j(\psi)$. For every $\ulHM_{E_v,\phi}$ we fix the $L\dbl y_{i}-\psi(y_{i})\dbr$-generator $\omega^\circ_{\psi,\phi}:=1\in\Koh^\psi(\ulHM_{E_v,\phi},L\dbl y_{i}-\psi(y_{i})\dbr)$. In addition, we let $\omega^\circ_{\psi,0}:=1\in\Koh^\psi(\ulHM_{E_v,0},L\dbl y_{i}-\psi(y_{i})\dbr)$. Then we can take the tensor product $\omega^\circ_\psi:=\omega^\circ_{\psi,0}\otimes\bigotimes\limits_{\phi\in H_{E_v}}(\omega^\circ_{\psi,\phi})^{\otimes d_\phi}$ in
\[
\Koh^\psi(\ulHM,L\dbl y_{i}-\psi(y_{i})\dbr)\;\cong\;\Koh^\psi(\ulHM_{E_v,0},L\dbl y_{i}-\psi(y_{i})\dbr)\otimes\bigotimes\limits_{\phi\in H_{E_v}}\Koh^\psi(\ulHM_{E_v,\phi},L\dbl y_{i}-\psi(y_{i})\dbr)^{\otimes d_\phi}.
\]
It is an $L\dbl y_{i}-\psi(y_{i})\dbr$-generator as in Definition~\ref{DefValuationOfomega}. Let $x\in L\dbl y_{i}-\psi(y_{i})\dbr\mal$ be such that $\omega_\psi=x\cdot\omega^\circ_\psi$, and let further $a\in E_v$ be such that $u\open:=a^{-1}u$ is an $\CO_{E_v}$-generator of $\Koh_{1,v}(\ulHM,A_v)$. Moreover, let $D_\psi$ be a generator as $\psi(\CO_{E_{v}})$-module of the different $\FD_{\psi(E_{v,i})/Q_v}$. Then \eqref{EqPeriodM}, \eqref{EqOmega} and Lemma~\ref{LemmaTrace} imply that
\[
{\TS\int_{u\open}}\omega^\circ_\psi\;=\; D_\psi^{-1}\cdot\epsilon_{i,j}c_{i,j}^{-1}\prod_{\phi\in H_{E_{v,i}}}\Omega(E_v,\phi,\psi)^{d_\phi}
\]
up to multiplication by an element of $R^\times +(z-\zeta)\cdot L\dbl z-\zeta\dbr$. Since $\epsilon_{i,j}c_{i,j}^{-1}\in\bigl(\CO_{E_v}\otimes_{A_v}\CO_{\BC_v}\dbl z\dbr\bigr)\mal$ we conclude as in the proof of Corollary~\ref{CorLubintateperiod} that ${\TS\int_u}\omega_\psi = (a\otimes1)x\cdot{\TS\int_{u\open}}\omega^\circ_\psi$ and 
\begin{eqnarray*}
v\bigl({\TS\int_u}\omega_\psi\bigr) & = &  v\Bigl(D_\psi^{-1}\cdot\psi(a)x\cdot\epsilon_{i,j}c_{i,j}^{-1}\prod_{\phi\in H_{E_{v,i}}}\Omega(E_v,\phi,\psi)^{d_\phi}\Bigr)\\[2mm]
 & = & v(\omega_\psi)+v_\psi(u)-v(\FD_{\psi(E_v)/Q_v})+\sum_{\phi\in H_{E_{v,i}}}\bigl( Z_v(a_{E_v, \psi, \phi},1)- \mu_{\Art,v}(a_{E_v, \psi, \phi})\bigr)\cdot d_\phi\\[2mm]
& = & Z_v(a_{E_v, \psi, \Phi},1)- \mu_{\Art,v}(a_{E_v, \psi, \Phi})-v(\FD_{\psi(E_v)/Q_v})+v(\omega_\psi)+v_\psi(u)\,,
\end{eqnarray*}
because $i(\phi)\ne i(\psi)$ implies that $a_{E_v,\psi,\phi}(g)=\delta_{g\psi,\phi}=0$ for all $g\in\sG_{Q_v}$.
\end{proof}

\begin{Corollary}\label{CorDMPeriod}
Keep the situation of Theorem~\ref{ThmDMPeriod}. For every $\eta\in H_L$ note that $i(\eta\psi)=i(\psi)$, let $\ulHM{}^\eta$ and $\omega_\psi^\eta\in\Koh^{\eta\psi}(\ulHM{}^\eta,L\dbl y_{i(\psi)}-\eta\psi(y_{i(\psi)})\dbr)$ be obtained by extension of scalars via $\eta$, and choose an $E_v$-generator $u_\eta\in\Koh_{1,v}(\ulHM{}^\eta,Q_v)$. Then
\[
\TS\tfrac{1}{\#H_L}\sum\limits_{\eta\in H_L}v({\TS\int_{u_\eta}\omega_\psi^\eta}) \;=\; Z_v(a^0_{E_v,\psi,\Phi},1)-\mu_{\Art,v}(a^0_{E_v,\psi,\Phi})-\dfrac{v(\Fd_{\psi(E_v)/Q_v})}{[\psi(E_v):Q_v]}+\tfrac{1}{\#H_L}\sum\limits_{\eta\in H_L}\bigl(v(\omega_\psi^\eta)+v_{\eta\psi}(u_\eta)\bigr)\,,
\]
where $\Fd_{\psi(E_v)/Q_v}\subset A_v$ is the discriminant of the field extension $\psi(E_v)/Q_v$.
\end{Corollary}

\begin{proof}
Since $\ulHM{}^\eta$ has complex multiplication by $\CO_{E_v}$ with local CM-type $\eta\Phi:=(d'_\phi)_{\phi\in H_{E_v}}$ with $d'_\phi=d_{\eta^{-1}\phi}$, Theorem~\ref{ThmDMPeriod} implies
\begin{equation}\label{EqCorDMPeriod}
v\bigl({\TS\int_{u_\eta}}\omega_\psi^\eta\bigr) \;=\; Z_v(a_{E_v,\, \eta\psi,\, \eta\Phi},1)- \mu_{\Art,v}(a_{E_v,\,\eta\psi,\, \eta\Phi})-v(\FD_{\eta\psi(E_v)/Q_v})+v(\omega_\psi^\eta)+v_{\eta\psi}(u_\eta)\,.
\end{equation}
We sum over all $\eta\in H_L$, divide by $\#H_L=[L:Q_v]$, and observe that $a_{E_v,\,\eta\psi,\, \eta\Phi}(g)=d'_{g\eta\psi}=d_{\eta^{-1}g\eta\psi}$ and $\FD_{\eta\psi(E_v)/Q_v}=\eta(\FD_{\psi(E_v)/Q_v})$, and hence 
\begin{align*}
\TS\sum\limits_{\eta\in H_L}v(\FD_{\eta\psi(E_v)/Q_v}) & \;=\;\TS v\Bigl(\,\prod\limits_{\eta\in H_L}\eta(\FD_{\psi(E_v)/Q_v})\Bigr)\;=\;v\bigl(N_{L/Q_v}(\FD_{\psi(E_v)/Q_v})\bigr)\\[1mm]
&\;=\;v\Bigl(N_{\psi(E_v)/Q_v}\bigl(N_{L/\psi(E_v)}(\FD_{\psi(E_v)/Q_v})\bigr)\Bigr)\;=\;[L:\psi(E_v)]\cdot v(\Fd_{\psi(E_v)/Q_v})\,.
\end{align*}
This proves the corollary.
\end{proof}

Finally we are ready to give the 

\begin{proof}[Proof of Theorem~\ref{MainThm}]
The proof proceeds like the one of the previous corollary applied to $\ulHM:=\ulHM_v(\ul\CM)$  for a model $\ul\CM$ of $\ulM$ with good reduction. Setting $E_v:=E\otimes_Q Q_v$, still $\ulHM{}^\eta:=\ulHM_v(\ul\CM^\eta)$ has complex multiplication by $\CO_{E_v}$ with local CM-type $\eta\Phi:=(d'_\phi)_{\phi\in H_E}$ with $d'_\phi=d_{\eta^{-1}\phi}$, where $\eta$ runs over all elements in $H_K$. To translate the global situation to the local one, we also use the elements $h_{\Betti,v}(u_\eta)\in\Koh_{1,v}(\ulHM{}^\eta,Q_v)$, and $\omega_\psi^\eta\otimes_KL\in \Koh^{\eta\psi}(\ulM,L\dbl y_{i(\eta\psi)}-\eta\psi(y_{i(\eta\psi)})\dbr)=\Koh^{\eta\psi}(\ulHM,L\dbl y_{i(\eta\psi)}-\eta\psi(y_{i(\eta\psi)})\dbr)$. Then $a_{E_v,\psi,\Phi}\in\CC(\sG_{Q_v}, \BQ)$ is the image of $a_{E,\psi,\Phi}\in \CC(\sG_Q, \BQ)$, and \eqref{EqCorDMPeriod} takes the form 
\begin{equation}\label{Eq1ProofOfThm1.3}
v\bigl({\TS\int_{u_\eta}}\omega_\psi^\eta\bigr) \;=\; Z_v(a_{E,\, \eta\psi,\, \eta\Phi},1)- \mu_{\Art,v}(a_{E,\,\eta\psi,\, \eta\Phi})-v(\FD_{\eta\psi(E_v)/Q_v})+v(\omega_\psi^\eta)+v_{\eta\psi}(u_\eta)\,,
\end{equation}
because the value of $v_{\eta\psi}\bigl(h_{\Betti,v}(u_\eta)\bigr)$ from Definition~\ref{DefVPsiLocSht} used in \eqref{EqCorDMPeriod}, coincides with the value of $v_{\eta\psi}(u_\eta)$ from \eqref{EqVPsi}. This time we sum over all $\eta\in H_K$, divide by $\#H_K=[K:Q]$, and observe that $a_{E,\,\eta\psi,\, \eta\Phi}(g)=d'_{g\eta\psi}=d_{\eta^{-1}g\eta\psi}$, and hence
\begin{equation}\label{Eq2ProofOfThm1.3}
\TS\tfrac{1}{\#H_K}\sum\limits_{\eta\in H_K}Z_v(a_{E,\, \eta\psi,\, \eta\Phi},1)- \mu_{\Art,v}(a_{E,\,\eta\psi,\, \eta\Phi}) \;=\;  Z_v(a^0_{E,\psi,\Phi},1)-\mu_{\Art,v}(a^0_{E,\psi,\Phi})\,.
\end{equation}
For every place $w$ of the field $\psi(E)$ above $v$ let $\psi(E)_w$ be the completion. Then $\psi(E)\otimes_Q Q_v = \prod_{w|v}\psi(E)_w$. Via the fixed inclusion $Q^\alg\subset Q_v^\alg$ we consider every $\eta\in H_K$ as a morphism $\eta\colon K\to Q^\alg\subset Q_v^\alg$. The induced morphism $\eta\otimes\id_{Q_v}\colon K\otimes_Q Q_v\to Q_v^\alg$, when restricted to a morphism $\psi(E)\otimes_Q Q_v\to Q_v^\alg$ factors over $\psi(E)_w$ for a unique $w$ which we denote by $w(\eta)$. We set $H_{K,w}:=\{\eta\in H_K\colon w(\eta)=w\}$ and consider the map 
\begin{equation}\label{EqProofOfThm1.3}
H_{K,w}\longto H_{\psi(E)_w}\,, \quad\eta\longmapsto (\eta\otimes\id_{Q_v})|_{\psi(E)_w}\;.
\end{equation}
The map is surjective, because every element of $H_{\psi(E)_w}$ can be restricted to a $Q$-homomorphism $\psi(E)\to Q^\alg\subset Q_v^\alg$, which extends to a $Q$-homomorphism $(\eta\colon K\to Q^\alg)\in H_K$ that automatically lies in $H_{K,w}$. We claim that two elements $\eta,\tilde\eta\in H_{K,w}$ have the same image under the map \eqref{EqProofOfThm1.3} if and only if $\tilde\eta=\eta\circ\alpha$ for an element $\alpha\in\Gal\bigl(K/\psi(E)\bigr)$. Indeed the latter condition is sufficient, because $\alpha\otimes\id_{Q_v}$ induces the identity on $\psi(E)_w$. To see that it is necessary let $\eta,\tilde\eta\in H_{K,w}\subset H_K=\Gal(K/Q)$ have the same image. Then their restrictions to $\psi(E)\subset\psi(E)_w$ coincide, and hence $\alpha:=\eta^{-1}\circ\tilde\eta\in\Gal(K/Q)$ lies in $\Gal\bigl(K/\psi(E)\bigr)$ as claimed. For every $\eta\in H_{K,w}$ the different $\FD_{\eta\psi(E_v)/Q_v}$ equals $(\eta\otimes\id_{Q_v})(\FD_{\psi(E)_w/Q_v})$ and only depends on the image of $\eta$ under the map \eqref{EqProofOfThm1.3}. Therefore, we compute
\begin{align*}
\TS\sum\limits_{\eta\in H_{K,w}}v(\FD_{\eta\psi(E_v)/Q_v}) & \;=\;\TS v\Bigl(\,\prod\limits_{\eta\in H_{K,w}}(\eta\otimes\id_{Q_v})(\FD_{\psi(E)_w/Q_v})\Bigr)\\[1mm]
&\;=\;v\Bigl(N_{\psi(E)_w/Q_v}(\FD_{\psi(E)_w/Q_v})^{\#\Gal\bigl(K/\psi(E)\bigr)}\Bigr)\;=\;[K:\psi(E)]\cdot v(\Fd_{\psi(E)_w/Q_v})\,.
\end{align*}
Summing over all $w|v$ and using that $\sum_{w|v}v(\Fd_{\psi(E)_w/Q_v})=v(\Fd_{\psi(E)/Q})$ by \cite[\S\,III.4, Corollary to Proposition~10]{SerreLF} we obtain from \eqref{Eq1ProofOfThm1.3} and \eqref{Eq2ProofOfThm1.3}
\[
\TS\tfrac{1}{\#H_K}\sum\limits_{\eta\in H_K}v({\TS\int_{u_\eta}\omega_\psi^\eta}) \;=\; Z_v(a^0_{E,\psi,\Phi},1)-\mu_{\Art,v}(a^0_{E,\psi,\Phi})-\dfrac{v(\Fd_{\psi(E)/Q})}{[\psi(E):Q]}+\tfrac{1}{\#H_K}\sum\limits_{\eta\in H_K}\bigl(v(\omega_\psi^\eta)+v_{\eta\psi}(u_\eta)\bigr).
\]
This proves Theorem~\ref{MainThm}.
\end{proof}

\begin{appendix}

\section{Appendix: Product Decompositions of Certain Rings} \label{Appendix}
\setcounter{equation}{0}

In this appendix we establish certain product decompositions for the rings used in this article. We begin with the following

\begin{Lemma}\label{LemDerivative}
Let $k$ be a field and let $z=\sum_{n=0}^{\infty}b_n y^n\in k\dbl y\dbr$. Let $\psi\colon k\dbl y\dbr\to R$ be a ring homomorphism into a $k$-algebra $R$. Then in $k\dbl y\dbr\wh\otimes_{k,\psi}R:=\invlim[n] k\dbl y\dbr/(y^n)\otimes_{k,\psi}R\cong R\dbl y\dbr$ the fraction $\tfrac{z\otimes 1-1\otimes\psi(z)}{y\otimes 1-1\otimes\psi(y)}$ exists and is congruent to $1\otimes \psi\bigl(\tfrac{dz}{dy}\bigr)$ modulo $y\otimes 1-1\otimes\psi(y)$.
\end{Lemma}

\begin{proof}
The lemma follows from the computation
\begin{eqnarray*}
z\otimes 1-1\otimes\psi(z) & = & \sum_{n=0}^{\infty}\bigl(b_n y^n\otimes 1-1\otimes\psi(b_n)\psi(y)^n\bigr) \\
& = & \sum_{n=1}^\infty \bigl(1\otimes\psi(b_n)\bigr)\cdot\sum_{\nu=0}^{n-1}\bigl(y^\nu\otimes\psi(y)^{n-1-\nu}\bigr)\cdot\bigl(y\otimes 1-1\otimes\psi(y)\bigr)\\[2mm]
& = & \bigl(y\otimes 1-1\otimes\psi(y)\bigr)\cdot\sum_{\nu=0}^\infty y^\nu\otimes\psi\Bigl(\sum_{n=\nu+1}^\infty b_n y^{n-1-\nu}\Bigr)\,,
\end{eqnarray*}
where the second factor converges in $k\dbl y\dbr\wh\otimes_{k,\psi}R$. Modulo $y\otimes 1-1\otimes\psi(y)$ this factor equals 
\begin{eqnarray*}
\sum_{n=1}^\infty \bigl(1\otimes\psi(b_n)\bigr)\cdot\sum_{\nu=0}^{n-1}\bigl(y^\nu\otimes\psi(y)^{n-1-\nu}\bigr) & = & \sum_{n=1}^\infty \bigl(1\otimes\psi(b_n)\bigr)\cdot n\bigl(1\otimes\psi(y)^{n-1}\bigr)\\
& = & 1\otimes \psi\Bigl(\sum_{n=1}^\infty n b_n y^{n-1}\Bigr)\\
& = & 1\otimes \psi\Bigl(\frac{dz}{dy}\Bigr)\,.
\end{eqnarray*}

\end{proof}

We need the following well known fact from field theory. For the convenience of the reader we include a proof.

\begin{Lemma}\label{LemmaInsepExt}
Let $E$ be a finite field extension of $Q$ (or of $Q_v$) of inseparability degree $p^m$. Then the separable closure $E'$ of $Q$ (resp.\ of $Q_v$) in $E$ equals $E^{p^m}:=\{x^{p^m}\colon x\in E\}$. If $y$ is a uniformizing parameter at a place $\tilde v$ of $E$ then $y':=y^{p^m}$ is a uniformizing parameter at the place $\tilde v'$ of $E'$ below $\tilde v$ and $E=E'(y)=E'[X]/(X^{p^m}-y')$.
\end{Lemma}

\begin{proof}
This is due to the fact that $Q$ has transcendence degree one over $\BF_q$, respectively that $Q_v$ is a discretely valued field. Namely, consider the case for $Q_v$. Then $E=k\dpl y\dpr$ where $k$ is the finite residue field of $E$. Clearly $E^{p^m}=k\dpl y'\dpr$ and $E=E^{p^m}(y)=E^{p^m}[X]/(X^{p^m}-y')$, because $X^{p^m}-y'$ is irreducible in $E^{p^m}[X]$ by Eisenstein. In particular $[E:E^{p^m}]=p^m$. On the other hand, the minimal polynomial $f(X)$ of $y$ over $E'$ is of the form $g(X^{p^{m'}})$ for a separable, irreducible polynomial $g$ over $E'$ and an integer $m'\ge0$. Therefore the minimal polynomial of $y^{p^{m'}}$ over $E'$ is $g$ and $y^{p^{m'}}$ is separable over $E'$. This implies $y^{p^{m'}}\in E'$ and $\deg g=1$, whence $\deg f=p^{m'}\le p^m$. Therefore $m'\le m$ and $y^{p^m}\in E'$, and hence $E^{p^m}\subset E'$. Since $[E:E']=p^m=[E:E^{p^m}]$ it follows that $E'=E^{p^m}$. This proves the lemma for $Q_v$.

For $Q$ a proof for the equality $E'=E^{p^m}$ can be found for example in \cite[Chapter~II, Corollary~2.12]{Silverman1}. If $\CO_{E,\tilde v}$ is the valuation ring of $E$ at $\tilde v$ then $(\CO_{E,\tilde v})^{p^m}$ equals the valuation ring $\CO_{E',\tilde v'}$ of $E'$ at $\tilde v'$ and so $y'$ is a uniformizing parameter of $\CO_{E',\tilde v'}$. The last equality follows from the fact that the polynomial $X^{p^m}-y'\in\CO_{E',\tilde v'}[X]$ is irreducible by Eisenstein.
\end{proof}

In the next lemma we consider the embeddings $Q\into K\dbl z_v-\zeta_v\dbr$ and $Q_v\into L\dbl z_v-\zeta_v\dbr$ given by $z_v\mapsto z_v=\zeta_v+(z_v-\zeta_v)$.

\begin{Lemma}\label{LemDecompdR}
Let $E=E_1\times\ldots\times E_s$ be a product of finite field extensions of $Q$ and let $K\subset Q^\alg$ be a field extension of $Q$ with $\psi(E)\subset K$ for all $\psi\in H_E:=\Hom_Q(E,Q^\alg)$. Let $i(\psi)$ be such that $\psi$ factors through $E\onto E_{i(\psi)}$ and let $y_{i(\psi)}\in E_{i(\psi)}$ be a uniformizing parameter at a place of $E_{i(\psi)}$ above $v$. Then
\begin{eqnarray*}
& & E\otimes_{Q} K\dbl z_v-\zeta_v\dbr\;=\; \prod_{\psi\in H_{E}}K\dbl y_{i(\psi)}-\psi(y_{i(\psi)})\dbr \qquad\text{and}\\
& & E\otimes_{Q} K\;=\; \prod_{\psi\in H_{E}}K\dbl y_{i(\psi)}-\psi(y_{i(\psi)})\dbr/\bigl(y_{i(\psi)}-\psi(y_{i(\psi)})\bigr)^{[E_{i(\psi)}\,:\,Q]_\insep},
\end{eqnarray*}
where $[E_{i(\psi)}:Q]_\insep$ is the inseparability degree of $E_{i(\psi)}$ over $Q$.

Likewise, let $E_v=E_{v,1}\times\ldots\times E_{v,s}$ be a product of finite field extensions of $Q_v$ and let $L\subset Q_v^\alg$ be a field extension of $Q_v$ with $\psi(E_v)\subset L$ for all $\psi\in H_{E_v}:=\Hom_{Q_v}(E_v,Q_v^\alg)$. Let $i(\psi)$ be such that $\psi$ factors through $E\onto E_{v,i(\psi)}$ and let $y_{i(\psi)}\in E_{v,i(\psi)}$ be a uniformizing parameter. Then
\begin{eqnarray}
\label{EqDecomp1} & & E_v\otimes_{Q_v} L\dbl z_v-\zeta_v\dbr\;=\; \prod_{\psi\in H_{E_v}}L\dbl y_{i(\psi)}-\psi(y_{i(\psi)})\dbr \qquad\text{and}\\
\label{EqDecomp2} & & E_v\otimes_{Q_v} L\;=\; \prod_{\psi\in H_{E_v}}L\dbl y_{i(\psi)}-\psi(y_{i(\psi)})\dbr/\bigl(y_{i(\psi)}-\psi(y_{i(\psi)})\bigr)^{[E_{v,i(\psi)}\,:\,Q_v]_\insep},
\end{eqnarray}
where $[E_{v,i(\psi)}:Q_v]_\insep$ is the inseparability degree of $E_{v,i(\psi)}$ over $Q_v$.
\end{Lemma}

\begin{proof}
Fix a $\psi$, set $i:=i(\psi)$ and let $E'_i$, respectively $E'_{v,i}$ be the separable closure of $Q$ in $E_i$, respectively of $Q_v$ in $E_{v,i}$. Then $H_{E_i}=H_{E'_i}$, respectively $H_{E_{v,i}}=H_{E'_{v,i}}$, and
\begin{equation}\label{EqLemmaDecompdR1}
\TS E'_i\otimes_Q K \;\isoto\; \prod\limits_{\psi\in H_{E_i}} K\,, \qquad\text{respectively}\qquad E'_{v,i}\otimes_{Q_v} L \;\isoto\; \prod\limits_{\psi\in H_{E_{v,i}}} L\,.
\end{equation}
Let $p^m:=[E_i:Q]_\insep=[E_i:E'_i]$, respectively $p^m:=[E_{v,i}:Q_v]_\insep=[E_{v,i}:E'_{v,i}]$, and let $y'_i:=y_i^{p^m}$. Then Lemma~\ref{LemmaInsepExt} implies that $y'_i\in E'_i$ is a uniformizing parameter at a place above $v$. By Hensel's lemma the decompositions~\eqref{EqLemmaDecompdR1} extend to decompositions
\begin{eqnarray*}
E'_i\otimes_{Q}K\dbl z_v-\zeta_v\dbr & \isoto & \TS\prod\limits_{\psi\in H_{E_i}} K\dbl z_v-\zeta_v\dbr \es=\es \TS\prod\limits_{\psi\in H_{E_i}} K\dbl y'_i-\psi(y'_i)\dbr\,,\quad\text{respectively} \\
E'_{v,i}\otimes_{Q_v}L\dbl z_v-\zeta_v\dbr & \isoto & \TS\prod\limits_{\psi\in H_{E_{v,i}}} L\dbl z_v-\zeta_v\dbr \es=\es \TS\prod\limits_{\psi\in H_{E_{v,i}}} L\dbl y'_i-\psi(y'_i)\dbr\,.
\end{eqnarray*}
Here the last identifications in each line follow from \cite[Lemma~1.2 and 1.3]{HartlJuschka} which states that both $K\dbl z_v-\zeta_v\dbr$ and $K\dbl y'_i-\psi(y'_i)\dbr$ are canonically isomorphic to the completion of the ring $\CO_{E_i'}\otimes_{\BF_q}K$ at the ideal $(a\otimes 1-1\otimes\psi(a)\colon a\in\CO_{E_i'})$. The identification in the second line also follows from Lemma~\ref{LemDerivative} by observing that the derivative $\tfrac{dz_v}{dy'_i}$ equals $-\tfrac{\partial}{\partial Y'_i}m(z_v,Y'_i)\big/\tfrac{\partial}{\partial z_v}m(z_v,Y'_i)\big|_{Y'_i=y'_i}$ where $m(z_v,Y'_i)\in \BF_v\dbl z_v\dbr[Y'_i]$ is the minimal polynomial of $y'_i$ over $Q_v$, and hence $\psi\bigl(\tfrac{dz_v}{dy'_i}\bigr)$ is non-zero by the separability of $y'_i$ over $Q_v$, and the injectivity of $\psi$ on $E_{v,i}$. Now $E_i=E'_i(y_i)$, respectively $E_{v,i}=E'_{v,i}(y_i)$, and hence $E_i\otimes_{E'_i}K\dbl y'_i-\psi(y'_i)\dbr=K\dbl y'_i-\psi(y'_i)\dbr[y_i-\psi(y_i)]=K\dbl y_i-\psi(y_i)\dbr$, respectively $E_{v,i}\otimes_{E'_{v,i}}L\dbl y'_i-\psi(y'_i)\dbr=L\dbl y'_i-\psi(y'_i)\dbr[y_i-\psi(y_i)]=L\dbl y_i-\psi(y_i)\dbr$ with $\bigl(y_i-\psi(y_i)\bigr)^{p^m}=y'_i-\psi(y'_i)$. Since $H_E=\bigcup_i H_{E_i}$, respectively $H_{E_v}=\bigcup_i H_{E_{v,i}}$, the lemma follows.
\end{proof}

}

\section{Erratum}\label{SectErratum}
\setcounter{equation}{0}

\subsection{First Error} \label{Erratum1}

In \PublOrArXiv{\cite[Formulas~(1.13) and (1.12) and Definition~5.21]{HartlSingh2}}
{\eqref{EqValuationOfOmegaPsi} and \eqref{EqVPsi} and Definition~\ref{DefVPsiLocSht}} 
it is claimed that $v(\omega_\psi)$ and $v_\psi(u)$ are integers. However, in general they only lie in the rational numbers $\BQ$, because the valuation $v$ is normalized to be an isomorphism $v\colon Q_v\mal/A_v\mal\isoto\BZ$, but the arguments of $v$ in both formulas lie in $Q_v^\alg$ instead of $Q_v$. 

This error is harmless, as the integrality of $v_\psi(u)$ and $v(\omega_\psi)$ is nowhere used.

\subsection{Second Error} \label{Erratum2}

In \PublOrArXiv{\cite[Formula~(1.13) and Definition~4.10]{HartlSingh2}}
{Formula~\eqref{EqValuationOfOmegaPsi} and Definition~\ref{DefValuationOfomega}} 
there is an \emph{error in the definition} of $v(\omega_\psi)$. 

As in most of \PublOrArXiv{\cite{HartlSingh2}}{the main text} 
we fix a finite separable semi-simple $Q$-algebra $E$. That is, $E$ is a product of finite separable field extensions of $Q$. We fix a finite place $v$ of $Q$ and consider the decomposition of the separable $Q_v$-algebra $E_v:=E\otimes_Q Q_v = E_{v,1}\times \cdots \times E_{v,s}$ into a product of finite field extensions $E_{v,i}$ of $Q_v$ as after \PublOrArXiv{\cite[Definition~4.1]{HartlSingh2}}{Definition~\ref{DefCME}}. We fix a finite Galois extension $K\subset Q^\alg$ of $Q$ and we let $L:=K_v\subset Q_v^\alg$ be the closure of $K$. It is a finite Galois extension of $Q_v$. We fix a $\psi\in H_E$. The canonical extension $\psi\otimes\id_{Q_v}\colon E_v\to L$ will be denoted again by $\psi$ and factors through the quotient $E_{v,i(\psi)}$ of $E_v$; see \PublOrArXiv{\cite[Definition~4.5]{HartlSingh2}}{Definition~\ref{DefIOfPsi}}.

Let $\ulHM{}$ be a local shtuka over $R:=\CO_{L}$ with complex multiplication by $\CO_{E_v}$ as in \PublOrArXiv{\cite[Definition~4.3]{HartlSingh2}}{Definition~\ref{DefCMOE}}. It may arise from a good model $\ul\CM$ of an $A$-motive over $R$ as in \PublOrArXiv{\cite[Example~3.2]{HartlSingh2}}{Example~\ref{AMotLocSht}}. We consider the one-dimensional $L$-vector space 
\begin{eqnarray}
\Koh^{\psi}(\ulHM{},L) & := & \bigl\{\omega\in\Koh^1_\dR(\ulHM{},L)\colon [a]^*\omega=\psi(a)\cdot\omega\es\forall\;a\in\CO_{E_v}\bigr\} \nonumber \\[2mm]
& \isoto & \Koh^1_\dR(\ulHM{},L)\big/([a]^*-\psi(a)\colon a\in\CO_{E_v})\cdot\Koh^1_\dR(\ulHM{},L)\,, \label{EqErratum1}
\end{eqnarray}
where the isomorphism comes from \PublOrArXiv{\cite[Proposition~4.9]{HartlSingh2}}{Proposition~\ref{PropCMTateMod}} using that $E$ is separable over $Q$.

In \PublOrArXiv{\cite[Formula~(1.13) and Definition~4.10]{HartlSingh2}}
{Formula~\eqref{EqValuationOfOmegaPsi} and Definition~\ref{DefValuationOfomega}} 
there is an \emph{error in the definition} of $v(\omega_\psi)$ for $L\dbl y_{i(\psi)}-\psi(y_{i(\psi)})\dbr$-generators $\omega_\psi$ of $\Koh^{\psi}(\ulHM{},L\dbl y_{i(\psi)}-\psi(y_{i(\psi)})\dbr)$. Namely, there as reference integral structure on the $L$-vector space $\Koh^{\psi}(\ulHM{},L)=\Koh^{\psi}(\ulHM{},L\dbl y_{i(\psi)}-\psi(y_{i(\psi)})\dbr)\big/\bigl(y_{i(\psi)}-\psi(y_{i(\psi)})\bigr)$ the $R$-module 
\[
\wt\Koh{}^{\psi}(\ulHM{},R)\;:=\;\bigl\{\omega\in\Koh^1_\dR(\ulHM{},R)\colon [a]^*\omega=\psi(a)\cdot\omega\es\forall\;a\in\CO_{E_v}\bigr\}
\]
was used (which was denoted without the $\wt{~}$ on $\wt\Koh$ in \PublOrArXiv{\cite[Formula~(1.13) and Definition~4.10]{HartlSingh2})}{\eqref{EqValuationOfOmegaPsi} and Definition~\ref{DefValuationOfomega})}.
Then $v(\omega_\psi)$ was defined to be 
\begin{equation}\label{EqFalse_v(omega)}
v^\sim(\omega_\psi)\;:=\;v(\tilde x)\;\in\;\BQ\,,
\end{equation}
where $\tilde x\in L\mal$ satisfies that $\tilde x^{-1}\bigl(\omega_\psi\mod y_{i(\psi)}-\psi(y_{i(\psi)})\bigr)$ is an $R$-generator of $\wt\Koh{}^{\psi}(\ulHM{},R)$. (To clarify the error we write $v^\sim(\omega_\psi)$ instead of $v(\omega_\psi)$ in this erratum.)

However, in the rest of \PublOrArXiv{\cite{HartlSingh2}}{the main text} the $R$-submodule 
\[
\Koh^{\psi}(\ulHM{},R)\;:=\;\Koh^1_\dR(\ulHM{},R)\big/([a]^*-\psi(a)\colon a\in\CO_{E_v})\cdot\Koh^1_\dR(\ulHM{},R)\;\subset\;\Koh^{\psi}(\ulHM{},L) 
\]
is used as reference integral structure on $\Koh^{\psi}(\ulHM{},L)$. (See Lemma~\ref{LemmaCompIntStr} below for why the latter is an inclusion, and how the two integral structures can be compared.) Correspondingly, the following definition for $v(\omega_\psi)$ is used in \PublOrArXiv{\cite{HartlSingh2}}{the main text}. 
\begin{equation} \label{EqCorrect_v(omega)}
v(\omega_\psi)\;:=\;v(x)\;\in\;\BQ\,,
\end{equation}
where $x\in L\mal$ satisfies that $x^{-1}\bigl(\omega_\psi\mod y_{i(\psi)}-\psi(y_{i(\psi)})\bigr)$ is an $R$-generator of $\Koh^{\psi}(\ulHM{},R)$. Indeed, in \PublOrArXiv{\cite[5.12]{HartlSingh2}}{\ref{Point4.12}} the generator $\omega_\psi\open:=1$ of $\Koh^{\psi}(\ulHM{},R)$ is used, which might not lie in $\wt\Koh{}^{\psi}(\ulHM{},R)$. Afterwards, any other generator $\omega_\psi$ is compared to the generator $\omega_\psi\open$. This error occurs in \PublOrArXiv{\cite[Theorems~1.3 and 5.24 and in Corollaries~5.22 and 5.25]{HartlSingh2}}{Theorems~\ref{MainThm} and \ref{ThmDMPeriod} and in Corollaries~\ref{CorLubintateperiod} and \ref{CorDMPeriod}}. In terms of the valuation $v^\sim(\omega_\psi)$ from \eqref{EqFalse_v(omega)}, all these theorems and corollaries have to be reformulated as explained below. However, with the definition of $v(\omega_\psi)$ in \eqref{EqCorrect_v(omega)} above, all these theorems and corollaries are correct. 

Note that if $\ulHM=\ulHM_v(\ul\CM)$ arises from a good model $\ul\CM$ of an $A$-motive over $R$ as in \PublOrArXiv{\cite[Example~3.2]{HartlSingh2}}{Example~\ref{AMotLocSht}}, then $\Koh^1_\dR(\ulHM,R)=\Koh^1_\dR(\ul\CM,R):=\sigma^*\CM\otimes_{A_{R},\,\gamma\otimes\id_{R}}R$, and hence
\begin{alignat*}{4}
& \wt\Koh{}^{\psi}(\ulHM,R) \;=\; && \wt\Koh{}^{\psi}(\ul\CM,R) & \;:=\; & \bigl\{\omega\in\Koh^1_\dR(\ul\CM,R)\colon [a]^*\omega=\psi(a)\cdot\omega\es\forall\;a\in\CO_E\bigr\}\,,\\[2mm]
& \Koh^{\psi}(\ulHM{},R) \;=\; && \Koh^{\psi}(\ul\CM,R) & \;:=\; & \Koh^1_\dR(\ul\CM,R)\big/([a]^*-\psi(a)\colon a\in\CO_E)\cdot\Koh^1_\dR(\ul\CM,R)\\[2mm]
& \text{inside} && \Koh^{\psi}(\ulHM,L) & \;=\; & \Koh^{\psi}(\ul\CM,L) \;=\;\wt\Koh{}^{\psi}(\ul\CM,R)\otimes_R L \;=\;\Koh{}^{\psi}(\ul\CM,R)\otimes_R L\,.
\end{alignat*}

We next show how the two integral structures can be compared.

\begin{Lemma}\label{LemmaCompIntStr}
The integral structures $\wt\Koh{}^{\psi}(\ulHM{},R)$ and $\Koh^{\psi}(\ulHM{},R)$ are free $R$-modules of rank one and contained in the $L$-vector space $\Koh^{\psi}(\ulHM{},L)$. The natural $R$-morphism
\[
\wt\Koh{}^{\psi}(\ulHM{},R) \;\longinto\;\Koh^1_\dR(\ulHM{},R)\;\longonto\;\Koh^{\psi}(\ulHM{},R)
\]
is injective with cokernel isomorphic to $R/R\cdot\psi(\FD_{E_v/Q_v})$, where $\FD_{E_v/Q_v}$ is the different of $E_v=\prod_{i=1}^s E_{v,i}$ over $Q_v$.
\end{Lemma}

\begin{proof}
The morphism fits into the following diagram
\begin{equation}\label{EqLemmaCompIntStr1}
\xymatrix @R=1pc {
\wt\Koh{}^{\psi}(\ulHM{},R) \ar@{^{ (}->}[dd] \ar@{^{ (}->}[r] & \Koh^1_\dR(\ulHM{},R) \ar@{^{ (}->}[d] \ar@{->>}[r] & \Koh^{\psi}(\ulHM{},R)\ar@{^{ (}->}^{\TS ?}[dd] \\
& \Koh^1_\dR(\ulHM{},L) \ar@{->>}[rd] \\
\Koh^{\psi}(\ulHM{},L) \ar@{^{ (}->}[ru] \ar[rr]^{\TS\sim} & & \Koh^{\psi}(\ulHM{},L)
}
\end{equation}
in which the lower isomorphism was described in \eqref{EqErratum1}, the lower triangle is the tensor product of the upper row with $L$, and the injectivity of the right vertical arrow still has to be proved. Note that the argument will not use the specific situation of de Rham cohomology of local shtukas. It will only use the isomorphism \eqref{EqErratum1} comming from \PublOrArXiv{\cite[Proposition~4.9]{HartlSingh2}}{Proposition~\ref{PropCMTateMod}} and the freeness of the $R$-module $\Koh^1_\dR(\ulHM,R)$ over $\CO_{E,v}\otimes_{A_v} R$, see below.

The $Q_v$-algebra $E_v$ acts on $\Koh^{\psi}(\ulHM{},L)$ through the character $\psi\colon E_v\onto E_{v,i(\psi)}\into L$. By \cite[\S\,III.6, Proposition~12]{SerreLF} there exists an element $y\in\CO_{E_{v,i(\psi)}}$ such that $\CO_{E_{v,i(\psi)}}=A_v[y]=A_v[Y]/(m)$ where $m\in A_v[Y]$ is the minimal polynomial of $y$ over $A_v$. The image $\gamma(m)$ under the map $\gamma\colon A_v[Y] \into R[Y]$ has $\psi(y)$ as a zero and correspondingly factors as
\[
\gamma(m) \;=\; \bigl(Y-\psi(y)\bigr)\cdot g(Y)
\]
for a monic polynomial $g(Y)\in R[Y]$. The derivative $m':=\tfrac{dm}{dY}\in A_v[Y]$ satisfies 
\begin{equation} \label{EqErratum2} 
\psi\bigl(m'(y)\bigr) \;=\; \gamma(m)'\bigl(\psi(y)\bigr) \;=\; g\bigl(\psi(y)\bigr)\,.
\end{equation}
Recall that $A_{v, R}$ is the $v$-adic completion of $A_R$. By \PublOrArXiv{\cite[Proposition~4.8]{HartlSingh2}}{Proposition~\ref{PropFree}} we can decompose $\ulHM=\bigoplus_{i=1}^s\ulHM_i$ into local shtukas $\ulHM_i$ over $R$ with complex multiplication by $\CO_{E_{v,i}}$. In particular, 
\begin{eqnarray*}
\wt\Koh{}^{\psi}(\ulHM{},R) & := & \bigl\{\omega\in\Koh^1_\dR(\ulHM_{i(\psi)},R)\colon [a]^*\omega=\psi(a)\cdot\omega\es\forall\;a\in\CO_{E_{v,i}}\bigr\} \qquad\text{and}\\[2mm]
\Koh^{\psi}(\ulHM{},R) & := & \Koh^1_\dR(\ulHM_{i(\psi)},R)\big/([a]^*-\psi(a)\colon a\in\CO_{E_{v,i}})\cdot\Koh^1_\dR(\ulHM_{i(\psi)},R)
\end{eqnarray*}
can be computed from 
\[
\Koh^1_\dR(\ulHM_{i(\psi)},R)\;:=\;\sigma^*\hat{M}_{i(\psi)}\otimes_{A_{v,R},\;\gamma\otimes\id_R} R
\]
instead of $\Koh^1_\dR(\ulHM,R)$. Moreover, by loc.\ cit.\ $\ulHM$ is free over $\CO_{E_v,R}:=\CO_{E_v}\otimes_{A_v}R\dbl z\dbr=\CO_{E_v}\wh\otimes_{\BF_v}R$ of rank one, and we may choose a generator of $\hat M$. This generator provides an isomorphism
\[
\Koh^1_\dR(\ulHM_{i(\psi)},R)\;\cong\;(\CO_{E_{v,i(\psi)}}\wh\otimes_{\BF_q}R)\underset{A_{v,R},\;\gamma\otimes\id_R}{\otimes} R\;=\;\CO_{E_{v,i(\psi)}}\underset{A_v,\gamma}{\otimes} R 
\;=\;R[Y]/(\gamma(m))\,.
\]
Since $[a]^*-\psi(a)=[a]^*-\gamma(a)$ for $a\in A_v$ already annihilates $\Koh^1_\dR(\ulHM,R)$, this yields the upper vertical isomorphisms in the following diagram
\[
\xymatrix{ 
\wt\Koh{}^{\psi}(\ulHM{},R) \ar[d]^-{\TS\cong} \ar@{^{ (}->}[r] & \Koh^{\psi}(\ulHM{},R) \ar[d]^-{\TS\cong} \\
\bigl\{\,f\in \CO_{E_{v,i(\psi)}}\otimes_{A_v,\gamma}R\colon (y\otimes1-1\otimes\psi(y))\cdot f=0\,\bigr\} \ar@{=}[d] \ar@{^{ (}->}[r] & (\CO_{E_{v,i(\psi)}}\otimes_{A_v,\gamma}R) \big/(y\otimes1-1\otimes\psi(y)) \ar@{=}[d] \\
\bigl\{\,f\in R[Y]/(\gamma(m))\colon (Y-\psi(y))\cdot f=0\,\bigr\} \ar@{^{ (}->}[r] \ar@{=}[d] & R[Y]/\bigl(\gamma(m),Y-\psi(y)\bigr) \ar@{=}[d] \\
g(Y)\cdot R[Y]/(\gamma(m)) \ar@{^{ (}->}[r] & R[Y]/(Y-\psi(y))
}
\]
The injectivity of the horizontal maps follows from diagram~\eqref{EqLemmaCompIntStr1}. The lower left equality holds because $R[Y]$ has no $(Y-\psi(y))$-torsion. Next, $\Koh^{\psi}(\ulHM{},R)\cong R[Y]/(Y-\psi(y))\cong R$ is free, and hence contained in $\Koh^{\psi}(\ulHM{},R)\otimes_RL=\Koh^{\psi}(\ulHM{},L)$. Finally, the image of the lower horizontal map is the ideal 
\[
R\cdot g\bigl(\psi(y)\bigr) \;=\; R\cdot \psi\bigl(m'(y)\bigr) \;=\; R\cdot\psi(\FD_{E_{v,i(\psi)}/Q_v}) \;=\; R\cdot\psi(\FD_{E_v/Q_v})\,;
\]
 see \cite[\S\,III.4, Proposition~10 and \S\,III.6, Corollary~2]{SerreLF}.
\end{proof}

\begin{Corollary}\label{CorCompIntStr}
For an $L$-generator $\omega_\psi$ of $\Koh^\psi(\ulHM,L)$, the two valuations \eqref{EqFalse_v(omega)} and \eqref{EqCorrect_v(omega)} satisfy
\[
v(\omega_\psi)-v^\sim(\omega_\psi)\;=\;v(\FD_{\psi(E_v)/Q_v})\;=\;v\bigl(\psi(\FD_{E_v/Q_v})\bigr)\;=\;v\bigl(\psi(\FD_{E/Q})\bigr)\,.
\]
\end{Corollary}

\begin{proof}
Let $x,\tilde x\in L\mal$ be elements such that $x^{-1}\omega_\psi$ is an $R$-generator of $\Koh^\psi(\ulHM,R)$ and $\tilde x^{-1}\omega_\psi$ is an $R$-generator of $\wt\Koh{}^\psi(\ulHM,R)$. Then $x/\tilde x$ is an $R$-generator of $\psi(\FD_{E_v/Q_v})$ by Lemma~\ref{LemmaCompIntStr} and the corollary follows.
\end{proof}

Now let $\ulM$ be an $A$-motive over a finite Galois extension $K\subset Q^\alg$ of $Q$ with complex multiplication by a finite separable semi-simple $Q$-algebra $E$. Assume that $\psi(E)\subset K$ for all $\psi\in H_E$. Fix a $\psi\in H_E$ and let $\omega_\psi$ be a generator of the $K\dbl y_{\psi}-\psi(y_{\psi})\dbr$-module $\Koh^{\psi}(\ulM^\eta,K\dbl y_{\psi}-\psi(y_{\psi})\dbr)$. For every $\eta\in H_K$ let $\ulM^\eta$ and $\omega_\psi^\eta\in\Koh^{\eta\psi}(\ulM^\eta,K\dbl y_{\eta\psi}-\eta\psi(y_{\eta\psi})\dbr)$ be obtained by extension of scalars via $\eta$. With the corollary and the computation
\begin{align}
\nonumber \TS\sum\limits_{\eta\in H_K} v(\omega_\psi^\eta)-v^\sim(\omega_\psi^\eta) & \;=\; \TS\sum\limits_{\eta\in H_K} v\bigl(\eta\psi(\FD_{E/Q})\bigr) \;=\;\TS v\Bigl(\,\prod\limits_{\eta\in H_K}\eta\psi(\FD_{E/Q})\Bigr)\;=\;v\bigl(N_{K/Q}(\FD_{\psi(E)/Q})\bigr)\\[1mm]
\nonumber &\;=\;v\Bigl(N_{\psi(E)/Q}\bigl(N_{K/\psi(E)}(\FD_{\psi(E)/Q})\bigr)\Bigr)\;=\;[K:\psi(E)]\cdot v(\Fd_{\psi(\CO_E)/A})
\end{align}
we obtain a reformulation of \PublOrArXiv{\cite[Theorems~1.3 and 5.24 and Corollaries~5.22 and 5.25]{HartlSingh2}}{Theorems~\ref{MainThm} and \ref{ThmDMPeriod} and Corollaries~\ref{CorLubintateperiod} and \ref{CorDMPeriod}} in terms of $v^\sim(\omega_\psi)$, which is even more analogous to \cite[Theorem~II.1.1(i)]{Colmez93}.

\bigskip\noindent
{\bfseries Theorem~1.3'.}
{\itshape
Let $\omega_\psi$ be a generator of the $K\dbl y_\psi-\psi(y_\psi)\dbr$-module $\Koh^\psi(\ulM,K\dbl y_\psi-\psi(y_\psi)\dbr)$. For every $\eta\in H_K$ let $\ulM^\eta$ and $\omega_\psi^\eta\in\Koh^{\eta\psi}(\ulM^\eta,K\dbl y_{\eta\psi}-\eta\psi(y_{\eta\psi})\dbr)$ be obtained by extension of scalars via $\eta$, and choose an $E$-generator $u_\eta\in\Koh_{1,\Betti}(\ulM^\eta,Q)$. Then for every place $v\ne\infty$ of $C$ we have
\[
\TS\tfrac{1}{\#H_K}\sum\limits_{\eta\in H_K}v({\TS\int_{u_\eta}\omega_\psi^\eta}) \;=\; Z_v(a^0_{E,\psi,\Phi},1)-\mu_{\Art,v}(a^0_{E,\psi,\Phi})+\tfrac{1}{\#H_K}\sum\limits_{\eta\in H_K}\bigl(v^\sim(\omega_\psi^\eta)+v_{\eta\psi}(u_\eta)\bigr)\,.
\qed
\]
}

\bigskip\noindent
{\bfseries Corollary~5.22'.}
{\itshape
Let $\phi,\psi\in H_{E_v}$ with $i(\phi)=i(\psi)=:i$ and assume that $E_{v,i}$ is separable over $Q_v$. Let $u\in\Koh_{1,v}(\ulHM_{E_v,\phi},Q_v)$ be a generator as $E_v$-module and let $\omega_\psi$ be an $L\dbl y_{i}-\psi(y_{i})\dbr$-generator of $\Koh^\psi(\ulHM_{E_v,\phi},L\dbl y_{i}-\psi(y_{i})\dbr)$. Then $\int_u\omega_\psi:=u\otimes\id_{\BC_v\dpl z-\zeta\dpr}\bigl(h^{-1}_{v,\dR}(\omega_\psi)\bigr)$ has valuation
\[
v\bigl({\TS\int_u}\omega_\psi\bigr) \;=\; Z_v(a_{E_v, \psi, \phi},1)- \mu_{\Art,v}(a_{E_v, \psi, \phi})+v^\sim(\omega_\psi)+v_\psi(u)\,,
\]
where $v^\sim(\omega_\psi)$ and $v_\psi(u)$ were defined in \eqref{EqFalse_v(omega)} and \PublOrArXiv{\cite[Definition~5.21]{HartlSingh2}}{Definition\ref{DefVPsiLocSht}}, respectively.
\qed
}

\bigskip\noindent
{\bfseries Theorem~5.24'.}
{\itshape
Let $\ulHM$ be a local shtuka over $R$ with complex multiplication by the ring of integers $\CO_{E_v}$ in a commutative, semi-simple, separable $Q_v$-algebra $E_v$ with local CM-type $\Phi$, and assume that $\psi(E_v)\subset L$ for all $\psi\in H_{E_v}$ and that $L$ is separable over $Q_v$. Let $u\in\Koh_{1,v}(\ulHM,Q_v)$ be an $E_v$-generator and let $\omega_\psi$ be an $L\dbl y_{i(\psi)}-\psi(y_{i(\psi)})\dbr$-generator of $\Koh^\psi(\ulHM,L\dbl y_{i(\psi)}-\psi(y_{i(\psi)})\dbr)$. Then the period $\int_u\omega_\psi:=u\otimes\id_{\BC_v\dpl z-\zeta\dpr}\bigl(h^{-1}_{v,\dR}(\omega_\psi)\bigr)$ has valuation
\[
v\bigl({\TS\int_u}\omega_\psi\bigr) \;=\; Z_v(a_{E_v, \psi, \Phi},1)- \mu_{\Art,v}(a_{E_v, \psi, \Phi})+v^\sim(\omega_\psi)+v_\psi(u)\,,
\]
where $v^\sim(\omega_\psi)$ and $v_\psi(u)$ were defined in \eqref{EqFalse_v(omega)} and \PublOrArXiv{\cite[Definition~5.21]{HartlSingh2}}{Definition\ref{DefVPsiLocSht}}, respectively.
\qed
}

\bigskip\noindent
{\bfseries Corollary~5.25'.}
{\itshape
Keep the situation of Theorem~5.24'. For every $\eta\in H_L$ note that $i(\eta\psi)=i(\psi)$, let $\ulHM{}^\eta$ and $\omega_\psi^\eta\in\Koh^{\eta\psi}(\ulHM{}^\eta,L\dbl y_{i(\psi)}-\eta\psi(y_{i(\psi)})\dbr)$ be obtained by extension of scalars via $\eta$, and choose an $E_v$-generator $u_\eta\in\Koh_{1,v}(\ulHM{}^\eta,Q_v)$. Then
\[
\TS\tfrac{1}{\#H_L}\sum\limits_{\eta\in H_L}v({\TS\int_{u_\eta}\omega_\psi^\eta}) \;=\; Z_v(a^0_{E_v,\psi,\Phi},1)-\mu_{\Art,v}(a^0_{E_v,\psi,\Phi})+\tfrac{1}{\#H_L}\sum\limits_{\eta\in H_L}\bigl(v^\sim(\omega_\psi^\eta)+v_{\eta\psi}(u_\eta)\bigr)\,.
\qed
\]
}

\end{appendix}


\vfill

\begin{minipage}[t]{0.5\linewidth}
\noindent
Urs Hartl\\
Universit\"at M\"unster\\
Mathematisches Institut \\
Einsteinstr.~62\\
D -- 48149 M\"unster
\\ Germany
\\[1mm]
\href{https://www.uni-muenster.de/Arithm/hartl/}{https:/\hspace{-1mm}/www.uni-muenster.de/Arithm/hartl/}
\end{minipage}
\begin{minipage}[t]{0.45\linewidth}
\noindent
Rajneesh Kumar Singh \\
Ramakrishna Mission Vivekananda University\\
PO Belur Math\\
Dist Howrah 711202\\
West Bengal\\
India
\\[1mm]
\end{minipage}

\end{document}